\documentclass[11pt]{article}
\usepackage[top=2cm, left=2cm,right=2cm]{geometry}
\usepackage{graphicx}
\usepackage{epstopdf}
\usepackage{amsfonts}
\usepackage{amsmath}
\usepackage{mathrsfs}
\usepackage{amsthm}
\usepackage{subfigure}
 \usepackage{multicol,multicap}
 \usepackage{cite}
 \usepackage{hyperref}\hypersetup{urlcolor=blue, citecolor=blue}

\newtheorem{theorem}{Theorem}[section]
\newtheorem{lemma}[theorem]{Lemma}

\newtheorem{remark}[theorem]{Remark}

\numberwithin{equation}{section}

\begin{document}

\markboth{Z. Shen $\&$ J. Wei}{Spatiotemporal patterns in a delay-diffusion mussel-algae model}

\title{Spatiotemporal patterns near the Turing-Hopf bifurcation in a delay-diffusion mussel-algae model}

\author{ZUOLIN SHEN\footnote{Email: mathust\_lin@foxmail.com} ~and JUNJIE WEI\footnote{Corresponding author. Email: weijj@hit.edu.cn}\\
{\small Department of Mathematics,
Harbin Institute of Technology, \hfill{\ }}\\
{\small Harbin, Heilongjiang, 150001, P.R.China\hfill{\ }}\\
}

\maketitle

\begin{abstract}
The spatiotemporal patterns of a reaction diffusion mussel-algae system with a delay
subject to Neumann boundary conditions is considered. The paper is a continuation of our previous
studies on delay-diffusion mussel-algae model. The global existence and positivity of solutions are obtained.
The stability of the positive constant steady state and existence of Hopf bifurcation and Turing bifurcation
are discussed by analyzing the distribution of eigenvalues. Furthermore, the dynamic classifications near the
Turing-Hopf bifurcation
point are obtained in the dimensionless parameter space by calculating the normal form on the center manifold,
and the spatiotemporal patterns
consisting of spatially homogeneous periodic solutions, spatially inhomogeneous steady states, and spatially
inhomogeneous periodic solutions are identified in this parameter space through  some numerical simulations.
Both theoretical and numerical results reveal that the Turing-Hopf bifurcation can enrich the diversity of
spatial distribution of populations.
\end{abstract}

\section{Introduction}

Two typical features of biological systems are the complexity of their organization structure and the interactions
of various factors. Mussel beds are a typical system for the study of pattern formation and patterns develop at two
distinctly separate scales in mussel beds \cite{Liu-2014}, large-scale banded patterns, and small-scale net-shaped patterns.
One of the models used to describe the process of large-scale patterns is
\begin{subequations}\label{eq_ma}
\begin{equation}\label{eq_01}
 \begin{cases}
 \cfrac{\partial }{\partial t}M(x,t)=D_{M}\Delta M(x,t) +e c M(x,t) \left(A(x,t) -d_{M}\cfrac{k_{M}}{k_{M}+M(x,t)}\right),  \\
 \cfrac{\partial }{\partial t}A(x,t)=D_{A}\Delta A(x,t)+(A_{up}-A(x,t))f-\cfrac{c}{H}M(x,t)A(x,t).\\
\end{cases}
\end{equation}
with the following initial data and Neumann boundary conditions
\begin{equation}\label{eq_02}
\begin{array}{l}
\partial_{\nu}M=\partial_{\nu}A=0, ~x\in \partial\Omega, ~t>0,\\
 M(x,0)=M_{0}(x)\geq 0,~A(x,0)=A_{0}(x)\geq 0, ~x\in\Omega.
\end{array}
\end{equation}
\end{subequations}
where $\Omega$ is a bounded domain in $\mathbb{R}^n$ with a smooth boundary $\partial{\Omega}$.
$M(x,t)$ represents the mussel biomass density at location $x$ and time $t>0$ on the sediment, and
$A(x,t)$ represents the algae concentration in the lower water layer overlying the mussel bed
while $A_{up}$ describes the uniform concentration of algae in the upper reservoir water layer.
Here, $e$ is a conversion constant relating ingested algae to mussel biomass production,
$c$ is the consumption constant,
$d_{M}$ is the maximal per capita mussel mortality rate,
$k_{M}$ is the value of $M$ at which mortality is half maximal, and the mussel mortality is assumed to decrease when mussel density increases because of a reduction of dislodgment and predation in dense clumps.
$f$ is the rate of exchange between the lower and upper water layers,
$H$ is the height of the lower water layer,
$D_{M}$ and $D_{A}$ are the diffusion coefficients of the mussel and algae, respectively.
$\nu$ is the outward unit normal vector on $\partial \Omega$. The homogeneous Neumann boundary condition indicates that there is no biomass input and output at the boundary.

Such a mussel-algae model was first proposed by van de Koppel {\it et al.} \cite{Kop-2005} to
investigate the importance of self-organization in affecting the emergent properties of nature systems of large spatial scales.
One thing that's different from Koppel's original model is that there is no random Brownian dispersion term $D_{A}\Delta A$, but a unidirectional advection term $V\nabla A$ instead used
to describe the affect of tidal current. Cangelosi {\it et al.} \cite{Can} modified the model to the way it is now,
and the modification is an extension of original model which as a first approximation to the field experiment
of van de Koppel {\it et al.} \cite{Kop-2008} and Liu {\it et al.} \cite{Liu-2013} and
in exact accordance with their laboratory experiment.
Both models (original and modified) have been discussed by scholars, see \cite{Liu-2012, Liu-2014, ShM, WLS}.
Ghazaryan and Manukian \cite{GhM} have captured the nonlinear mechanisms
of pattern and wave formation of Koppel's original model by applying the geometric singular
perturbation theory.
Sherratt and Mackenzie \cite{ShM} have considered the implications of the algae's advection for
pattern formation with the advection oscillating with tidal flow.
Based on the normal form method, Song {\it et al.} \cite{SJL}
have studied the Turing-Hopf bifurcation of \eqref{eq_ma} with a Neumann boundary conditions,
and obtained the explicit dynamical classification in the corresponding critical point.

As is well known that delay can lead to the periodic solutions \cite{ChY, XuW}, while diffusion can cause
Turing patterns \cite{Kla, NiT, OuS, Tur}. An obvious idea is how their interaction will affect the dynamics of the system.
In this paper, we mainly study the following delay-diffusion mussel-algae system
\begin{equation}\label{eq_MA_tau}
 \begin{cases}
 \cfrac{\partial }{\partial t}M(x,t)=D_{M}\Delta M(x,t) +e c M(x,t) \left(A(x,t-\tau) -d_{M}\cfrac{k_{M}}{k_{M}+M(x,t-\tau)}\right),  \\
 \cfrac{\partial }{\partial t}A(x,t)=D_{A}\Delta A(x,t)+(A_{up}-A(x,t))f-\cfrac{c}{H}M(x,t)A(x,t).\\
 \partial_{\nu}M=\partial_{\nu}A=0, ~x\in \partial\Omega, ~t>0,\\
 M(x,t)=M_{0}(x,t)\geq 0,~A(x,t)=A_{0}(x,t)\geq 0, ~x\in\Omega,~-\tau\leq t\leq 0.
\end{cases}
\end{equation}
where $\tau$ is the digestion period of mussel and the mortality of mussels depends on the state whether they have eaten in the past.
By employing the rescaling
$$
\begin{array}{l}
m=\cfrac{M}{k_{M}},~a=\cfrac{A}{A_{up}},~\omega=\cfrac{c k_{M}}{H},~\hat{t}=d_{M}t,~\alpha=\cfrac{f}{\omega},\\
r=\cfrac{e c A_{up}}{d_{M}},~\gamma=\cfrac{d_{M}}{\omega},~d=\cfrac{D_{M}}{\gamma D_{A}},~\hat{x}=x\sqrt{\cfrac{\omega}{D_{A}}}.\\
\end{array}
$$
we have
\begin{equation}\label{eq_ma_tau}
 \begin{cases}
 \cfrac{\partial }{\partial t}m(x,t)=d\Delta m(x,t)+m(x,t)\Big(ra(x,t-\tau) -\cfrac{1}{1+m(x,t-\tau)}\Big), &x\in\Omega, ~t>0,\\
 \gamma  \cfrac{\partial}{\partial t} a(x,t)=\Delta a(x,t)+\alpha(1-a(x,t))-m(x,t)a(x,t),&x\in\Omega, ~t>0,\\
\partial_{\nu}m=\partial_{\nu}a=0, &x\in\partial\Omega, ~t>0,\\
 m(x,t)=m_{0}(x,t)\geq 0,~a(x,t)=a_{0}(x,t)\geq 0, &x\in\Omega, ~-\tau\leq t\leq 0.
\end{cases}
\end{equation}
For simplicity, we have drop the `~$\hat{}$~'.

This paper is a continuation of our previous studies on delay-diffusion mussel-algae model.
We mainly concern the spatiotemporal dynamics of Eq.\eqref{eq_ma_tau} near the Turing-Hopf bifurcation point with $\tau$ and $d$ as
the bifurcation parameters. The study of Turing-Hopf bifurcation is not a new topic \cite{BGF, DLDB, HaR, SoZ, YaS}. Most of the studies have
focused on the emergence of spatiotemporal patterns or the non-degenerate cases, but not many have been done on degenerate cases (Hopf bifurcation and Turing instability occur simulta
neously).
Recently, An and Jiang \cite{AnJ} extend the normal form methods proposed by Faria \cite{Far} to Turing-Hopf singularity of a general two-components delayed reaction diffusion system, and present a detailed calculation formulas.
Motivated by their work, we study the spatiotemporal dynamics of system \eqref{eq_ma_tau}. Compared with
the work of \cite{ShW-2, SJL}, we focus more on the common effects of delay and diffusion. Hence, a basic assumption is that the positive constant steady state is locally asymptotically stable under a homogeneous perturbation when time delay is equal to zero, and this assmption
allows us to identify the importance of delay and diffusion in the process of pattern formation.
The main contribution of this article can be concluded as: first, the proof of wellposedness of system \eqref{eq_ma_tau}; second, a detail bifurcation analysis with $\tau$ and $d$ as the bifurcation parameters;
third, we show a rational explanation of different spatiotemporal distribution of mussel beds from both theoretical results and numerical
simulations.

The rest of this paper is organized as follows.
In section 2, we firstly give the proof of wellposedness of solutions, then study the stability of positive constant steady state including the existence of the Hopf bifurcation, Turing instability, and Turing-Hopf interaction. We take
$\tau$ and $d$ as the bifurcation parameters which can reflect their effect on the dynamics of the system.
In section 3, we show a detailed formulas for calculating the normal form of system \eqref{eq_ma_tau} with the method proposed by \cite{AnJ}.
In section 4, we discuss the dynamic classification and spatiotemporal patterns near the Turing-Hopf
bifurcation point, and for each dynamic region, some numerical simulations are presented to illustrate our theoretical analysis.
In section 5, we end this paper with conclusions and some discussions about the following work of this model.
Throughout the paper, we denote $\mathbb{N}$ as the set of positive integers, and $\mathbb{N}_0=\mathbb{N}\cup\{0\}$ as the set of non-negative integers.

\section{Existence and stability analysis}

\subsection{Existence and boundedness}
In this subsection, we first state the wellposedness result of the solutions of the initial value problem \eqref{eq_ma_tau}, for more details of abstract theory, refer to \cite{Pao-1996, Tay}.

\begin{theorem}\label{th-exis}
Suppose that $\alpha$, $\gamma$, $r$ and $d$ are all positive, the initial data satisfies $m_{0}(x,t)\geq 0, a_{0}(x,t)\geq 0$  for $(x,t)\in \overline{\Omega}\times[-\tau,0]$. Then the system \eqref{eq_ma_tau} has a unique solution $(m(x,t),a(x,t))$ satisfying
$$0\leq m(x,t),~~0\leq a(x,t)\leq \max\{\|a_0\|_{\infty}, 1\}~~ \text{for}~~ (x,t)\in \overline{\Omega}\times[0,+\infty).$$
where $\|\psi\|_{\infty}=\sup_{x\in\overline{\Omega},t\in[-\tau,0]} \psi(x,t)$.
Moreover, if $m_{0}(x,0) \not\equiv 0, a_{0}(x,0) \not\equiv 0$, then $m(x,t)>0, a(x,t)>0$ for $(x,t)\in \overline{\Omega}\times(0,+\infty)$.
\end{theorem}

\begin{proof}
Define $F=\Big(f(m,a,m_{\tau},a_{\tau}),g(m,a,m_{\tau},a_{\tau})\Big)^T$ with
$$
f(m,a,m_{\tau},a_{\tau})=rma_{\tau} -\cfrac{m}{1+m_{\tau}}, ~g(m,a,m_{\tau},a_{\tau})=\alpha(1-a)-ma.
$$
where $m=m(x,t)$, $a=a(x,t)$, $m_{\tau}=m(x,t-\tau)$, $a_{\tau}=a(x,t-\tau)$.

 It is easy to prove that $F$ possesses a mixed quasi-monotone property since $D_af=0, D_{m_{\tau}}f>0, D_{a_{\tau}}f>0$ and $D_mg<0, D_{m_{\tau}}g= D_{a_{\tau}}g=0$ for $(m,a,m_{\tau},a_{\tau})\in \mathbb{R}^4_+$.

Let $(m^*(t),~a^*(t))$ be the unique solution of the following ODE system
\begin{equation}\label{eq_ode}
\begin{cases}
    \cfrac{\text{d}m}{\text{d}t}=rma -\cfrac{m}{1+m},\\
   \gamma \cfrac{\text{d}a}{\text{d}t}=\alpha(1-a),  \\
    m(0)=\phi_1, a(0)=\phi_2.
\end{cases}
\end{equation}
with
$$
  \phi_1=\sup_{x\in\overline{\Omega},t\in[-\tau,0]}m_0(x,t), ~~~~\phi_2=\sup_{x\in\overline{\Omega},t\in[-\tau,0]}a_0(x,t);
$$
Denote $(\widetilde{m}, \widetilde{a})=(m^*(t),~a^*(t))$, $(\widehat{m}, \widehat{a})=(0, 0)$. Since
$$
\cfrac{\partial \widetilde{m}}{\partial t}-d\Delta\widetilde{m}-r\widetilde{m}\widetilde{a}-\cfrac{\widetilde{m}}{1+\widetilde{m}}=0\geq 0=
\cfrac{\partial \widehat{m}}{\partial t}-d\Delta\widehat{m}-r\widehat{m}\widehat{a}-\cfrac{\widehat{m}}{1+\widehat{m}},
$$
$$
\gamma\cfrac{\partial \widetilde{a}}{\partial t}-\Delta\widetilde{a}-\alpha (1-\widetilde{a})-\widehat{m}\widetilde{a}=0\geq -\alpha=
\gamma\cfrac{\partial \widehat{a}}{\partial t}-\Delta\widehat{a}-\alpha(1-\widehat{a})-\widetilde{m}\widehat{a}.
$$
and
$$
0\leq m_0(x,t)\leq \phi_1, ~0\leq a_0(x,t)\leq \phi_2 ~~\text{for}~~ (x,t)\in \Omega\times[-\tau,0].
$$

Then  $(\widetilde{m}, \widetilde{a})$ and  $(\widehat{m}, \widehat{a})$ are the coupled upper and lower solutions of system \eqref{eq_ma_tau}. Hence, from Theorem 2.1 in \cite{Pao-1996}, the system \eqref{eq_ma_tau} has a unique solution $(m(x,t), a(x,t))$ which satisfies
$$
0\leq m(x,t)\leq m^*(t),~~0\leq a(x,t)\leq a^*(t)$$

Applying the comparison principle to the second equation of system \eqref{eq_ma_tau}, we can easily get $a(x,t)\leq \max\{\|a_0\|_{\infty}, 1\}$.
To prove the positivity, we set $t\in[0,\tau]$, then $m_{\tau}$, $a_{\tau}$ coincide with the initial data $m_0(x,t-\tau), a_0(x,t-\tau)$. Since $m_{0}(x,0)\not\equiv 0$, $a_{0}(x,0)\not\equiv 0$, then $m(x,t)>0, a(x,t)>0$ for $(x,t)\in \Omega\times (0,\tau]$ from the standard maximum principle for semilinear parabolic equations. Repeating this process, we can obtain that $m(x,t)>0$, $a(x,t)>0$ for  $(x,t)\in \Omega\times (0,\infty)$.
\end{proof}

\subsection{Stability analysis}\label{linear and Hopf}

For the convenience of further discussion, we first define the following real-value Sobolev space
$$ X:=\Big\{(u,v)\in H^2(\Omega)\times H^2(\Omega)|\partial_{\nu}u=\partial_{\nu}v=0, x\in \partial\Omega\Big\}
$$
and its complexification, $X_{\mathbb{C}}:=X\oplus iX=\{x_1+i x_2|x_1,x_2 \in X\}$ with a complex-valued $L^2$ inner product $<\cdot,\cdot>$ which defined as
$$
<U_1,U_2>=\int_{\Omega}(\bar{u}_1u_2+\bar{v}_1v_2)dx
$$
with $U_i=(u_i, v_i)^T \in X_{\mathbb{C}}, i=1,2$.

The system \eqref{eq_ma_tau} always has a non-negative constant solution $E_0(0,1)$ which corresponds to the bare sediment biologically, and
the system also has a positive equilibrium $E_*(m^*,a^*)$ with $m^*=\cfrac{\alpha (r-1)}{1-\alpha r}, a^*=\cfrac{1-\alpha r}{r(1-\alpha)}$
if the following assumption satisfies:
$$
\textsc{(H1)}~~~~~~\qquad~~~ 0<\alpha<1<r<\alpha^{-1}.~~~~~~~~
$$

Suppose that the spatial domain $\Omega=(0,l\pi)$, that is $\Omega$ is an interval in one space dimension. Here let the phase space $\mathscr{C}:=C([-\tau,0],X_{\mathbb{C}})$. Our main focus is the stability of  positive constant steady state $E_*(m^*,a^*)$ with respect to the model \eqref{eq_ma_tau}, and the results of the boundary steady state $E_0(0,1)$ can be seen in \cite{ShW-2}.

The linearization of system \eqref{eq_ma_tau} at $E_*(m^*,a^*)$ is given by
\begin{equation}\label{eq_linear}
\dot{U}(t)=D\Delta U(t)+L(U_t),
\end{equation}
where $D= \text{diag}(d, \gamma^{-1})$, and $L:\mathscr{C}\to X_{\mathbb{C}}$ is defined as
$$
L(\phi)=L_1\phi(0)+L_2\phi(-\tau),
$$
with
$$\begin{array}{l}
L_1=\gamma^{-1}\left( \begin{array}{cc}
         0 ~&~ 0\\
         -a^*  ~&~ -(\alpha+m^*)
       \end{array}\right),~~
\quad L_2=\left( \begin{array}{cc}
         \cfrac{m^*}{(1+m^*)^2} ~&~ r m^* \\
         0 ~&~ 0
       \end{array}\right),
\end{array}$$
$$
\phi(t)=\big(\phi_1(t),~\phi_2(t)\big)^{^T},~~\phi_t(\cdot)=\big(\phi_1(t+\cdot),~\phi_2(t+\cdot)\big)^{^T}.
$$

It is well known that the eigenvalue problem
$$
-\Delta \xi=\sigma \xi,~~x\in(0,l\pi),~~ \xi'(0)=\xi'(l\pi)=0
$$
has eigenvalues $\sigma_n=\frac{n^2}{l^2}$, $n\in\mathbb{N}_0$, with corresponding eigenfunctions $\xi_n(x)=\cos\frac{n}{l}x$.
Let $U(x,t)=e^{\lambda t}\xi(x)$, we have that the corresponding characteristic equation of system \eqref{eq_linear} satisfies
\begin{equation}\label{c-eq1}
    \lambda \xi-D\Delta \xi-L(e^{\lambda\,\cdot}\xi)=0,
\end{equation}
Then \eqref{c-eq1} can be transform into
$$\det\Big(\lambda \text{I}+D\cfrac{n^2}{l^2}-L_1-L_2 e^{-\lambda\tau}\Big)=0,~~n\in\mathbb{N}_0.$$
That is, there exists some $n\in \mathbb{N}_0$ such that $\lambda$ satisfies the following characteristic equation
\begin{equation}\label{c-eq2}
    E_n(\lambda,\tau,d):=\gamma\lambda^2+T_n\lambda+(B\lambda+M_n)e^{-\lambda\tau}+D_n=0, ~~n\in\mathbb{N}_0,
\end{equation}
where
\begin{equation}\label{TDMB}
\begin{array}{l}
     T_n=\alpha+m^*+(1+\gamma d)\cfrac{n^2}{l^2}, \quad  M_n=ra^*m^*(1-\alpha r-ra^*\cfrac{n^2}{l^2});\\
     D_n=d(\alpha+m^*+\cfrac{n^2}{l^2})\cfrac{n^2}{l^2},\quad B=-\gamma r^2{a^*}^2m^*.
\end{array}
\end{equation}

In the following, we analyze the existence of Turing-Hopf bifurcation for the positive constant steady state. In order to understand how delay and diffusion coefficient affect the Turing-Hopf bifurcation, we choose $\mu=(\tau, d)$ as the bifurcation parameters since Turing-Hopf is a codimension-two bifurcation. For a general case, $\mu=(\mu_1,\mu_2)\in\mathbb{R}^2$, the conditions for the occurrence of Turing-Hopf bifurcation can be described as:

\textbf{(TH)} There exists a neighborhood $\mathscr{N}(\mu_0)$ of $\mu_0=(\mu_{10}, \mu_{20})$, and $n_1, n_2 \in \mathbb{N}_0$ such that characteristic equation \eqref{c-eq2} has a pair of complex simple conjugate eigenvalues $\beta_{n_1}(\mu)\pm i\omega_{n_1}(\mu)$ and a simple real eigenvalue $\alpha_{n_2}(\mu)$ for $\mu\in\mathscr{N}(\mu_0)$, both continuously differentiable in $\mu$, and satisfy $\beta(\mu_0)=0, \omega(\mu_0)=\omega_0>0, \frac{\partial}{\partial{\mu_1}}\beta(\mu_0)\neq0, \alpha(\mu_0)=0, \frac{\partial}{\partial{\mu_2}}\alpha(\mu_0)\neq 0$; all other eigenvalues have non-zero real parts.

In our previous paper \cite{ShW-2}, it has been proved that, the system \eqref{eq_ma} without diffusion can undergo Hopf bifurcation when parameters are chosen appropriately.

\begin{lemma}\cite{ShW-2}\label{hopf}
Assume that $\textsc{(H1)}$ is satisfied. For system \eqref{eq_ma} without diffusion,
\begin{enumerate}
 \item  If $\mathcal{H}_{0}^2(r)<\mathcal{P}_{0}(r)$, the positive equilibrium $E_*(m^*,a^*)$ of system \eqref{eq_ma} is locally asymptotically stable;
 \item  If $\mathcal{H}_{0}^2(r)>\mathcal{P}_{0}(r)$, the positive equilibrium $E_*(m^*,a^*)$ of system \eqref{eq_ma} is unstable;
 \item  If $r_{_H}\in S$ satisfies the equation $\mathcal{H}_{0}^2(r)=\mathcal{P}_{0}(r)$, the system \eqref{eq_ma} undergoes a Hopf bifurcation at $r=r_{_H}$ which corresponds to spatially homogeneous periodic solution; the critical curve of Hopf bifurcation is defined by $\mathcal{H}_{0}^2(r)=\mathcal{P}_{0}(r)$, where  \\
 $\mathcal{H}_{0}(r)=\cfrac{1-\alpha r}{1-\alpha}$, $\mathcal{P}_{0}(r)=\cfrac{r(1-\alpha)}{\gamma(r-1)}$.

\end{enumerate}
\end{lemma}

Lemma \ref{hopf} indicated that the positive equilibrium $E_*(m^*,a^*)$ of system \eqref{eq_ma} is stable to homogeneous perturbations when $\mathcal{H}_{0}^2(r)<\mathcal{P}_{0}(r)$. Since system \eqref{eq_ma} is a special case when $\tau=0$ of system \eqref{eq_ma_tau}, our next work is to discuss the stability of $E_*(m^*,a^*)$ when $\tau>0$. To ensure our stability analysis valid, we make the following
assumption:
$$
\textsc{(H2)}~~~~~~\quad~~~ \mathcal{H}_{0}^2(r)<\mathcal{P}_{0}(r).~~~~~~~~
$$

Hence, we let $\pm i\omega(\omega>0)$ be
solutions of Eq.\eqref{c-eq2}, then we have

$$-\gamma\omega^2+i T_n \omega+(i\omega B+M_n)e^{-i\omega\tau}+D_n=0.$$
Separating the real and imaginary parts, it follows that
\begin{equation}\label{re_im}
\begin{cases}
M_n\cos\omega\tau +\omega B\sin\omega\tau= \gamma\omega^2-D_n,\\
M_n\sin\omega\tau -\omega B\cos\omega\tau =T_n\omega . \\
\end{cases}
\end{equation}
that is
\begin{equation}\label{omega}
\gamma^2\omega^4+(T_n^2-2\gamma D_n-B^2)\omega^2+D_n^2-M^2_n=0.
\end{equation}
Let $z=\omega^2$. Then \eqref{omega} can be converted to

\begin{equation}\label{z_eq}
\gamma^2 z^2+(T_n^2-2\gamma D_n-B^2)z+D_n^2-M^2_n=0.
\end{equation}
where $T_n^2-2\gamma D_n-B^2>0$ can be deduced by \textsc{(H2)}. Solving Eq.\eqref{z_eq} for $z$, we have
\begin{equation}\label{z_n}
z_n=\cfrac{-(T_n^2-2\gamma D_n-B^2)+\sqrt{(T_n^2-2\gamma D_n-B^2)^2-4\gamma^2(D_n^2-M_n^2)}}{2\gamma^2},
\end{equation}

Clearly, $D_0+M_0>0$, $D_0-M_0<0$, then $z_0=\omega_0^2$ is always exists, and $(\omega_0, \tau^j_0)$ always satisfies the characteristic equation \eqref{c-eq2}. This corresponds to a spatially homogeneous Hopf bifurcation.
In the following, we shall look for the spatially inhomogeneous Hopf bifurcation.
Note that
\begin{equation}
D_n+M_n=d\cfrac{n^4}{l^4}+(d\cfrac{\alpha}{a^*}-r^2 a^{*2}m^*)\cfrac{n^2}{l^2}+ar(r-1)a^*
\end{equation}
we can always choose a set of parameters ${d, \alpha, r}$ appropriately such that $D_n+M_n>0$ for all $n\in \mathbb{N}, l>0$. Hence, denote
\begin{equation*}
  \Gamma=\left\{(d,\alpha,r)|~~ D_n+M_n>0~\text{for all}~n\in\mathbb{N}, l>0\right\}
\end{equation*}
Now, the existence of $z_n$ is determined by the signal of $D_n-M_n$ when $(d,\alpha,r)\in \Gamma$. If  $D_n-M_n<0$, then  the
(n+1)th equation of \eqref{c-eq2} has a pair of simple pure imaginary $\pm i\omega_n$, and if $D_n-M_n>0$,  the
(n+1)th equation of \eqref{c-eq2} has no pure imaginary.

Define
\begin{equation}\label{l_n}
l_n=n\cfrac{1}{\sqrt{S(d,\alpha,r)}},~~~~ n\in\mathbb{N}.
\end{equation}
where
\begin{equation*}
   S(d, \alpha,r)=
-\cfrac{1}{2}\left(\frac{\alpha}{a^*}+\cfrac{ r^2a^{*2}m^*}{d}\right)+
\cfrac{1}{2d}\sqrt{(d\frac{\alpha}{a^*}+r^2a^{*2}m^*)^2+4d\alpha r (r-1)a^*}.
\end{equation*}
Then for $l_n<l<l_{n+1}$, and $1\leq {n_1}\leq n$, we have
$$
\cfrac{n_1^2}{l^2}<S(d, \alpha,r)
$$
which yields to $D_{n_1}-M_{n_1}<0$.
Hence, we can find a series of root $z_{{n_1}}$ of Eq.\eqref{z_n} and critical values $\tau_{{n_1}}^j$ satisfies

\begin{equation}\label{tau}
\tau_{{n_1}}^j=\begin{cases}
\cfrac{1}{\omega_{n_1}}\Big(\arccos\cfrac{(\gamma M_{{n_1}}-BT_{{n_1}} )\omega_{{n_1}}^{2}-M_{{n_1}}D_{n_1}}{M_{n_1}^2+\omega_{n_1}^2B^2}+2j\pi\Big),&\sin\omega_{n_1}\tau_{n_1}^j >0 \\
\cfrac{1}{\omega_{n_1}}\Big(-\arccos\cfrac{(\gamma M_{n_1}-BT_{n_1} )\omega_{n_1}^{2}-M_{n_1}D_{n_1}}{M_{n_1}^2+\omega_{n_1}^2B^2}+2(j+1)\pi\Big), &\sin\omega_{n_1}\tau_{n_1}^j <0
       \end{cases}~~0\leq{n_1}\leq n,~j\in\mathbb{N}_0.
\end{equation}
such that Eq.\eqref{c-eq2} has a pair of  purely imaginary roots $\pm i\omega_{n_1}$.

Following the work of \cite{Cooke}, it is easy to verify that the following transversality condition holds.
\begin{lemma}\label{trans_tau}
Suppose that \textsc{(H1)} and \textsc{(H2)} are satisfied, $(d,\alpha,r)\in \Gamma$, and $l\in(l_n, l_{n+1}]$ with $l_n$ is defined as in \eqref{l_n}. Then
$$
\cfrac{\partial}{\partial\tau}\beta(\tau^j_{n_1},d)>0,~~for~0\leq {n_1}\leq n, j\in\mathbb{N}_0,
$$
where $\beta(\tau,d)= ~\textrm{Re}~ \lambda(\tau,d)$.
\end{lemma}
\begin{proof}
Substituting $\lambda(\tau,d)$ into Eq.\eqref{c-eq2} and taking the derivative with respect to $\tau$ on both side, we obtain that
$$
\Big(2\gamma\lambda+T_{n_1}+Be^{-\lambda\tau}-\tau(B\lambda+M_{n_1})e^{-\lambda\tau}\Big)\cfrac{\text{d}\lambda}{\text{d}\tau}-\lambda(B\lambda+M_{n_1})e^{-\lambda\tau}=0.
$$
Thus
$$
\left(\cfrac{\text{d}\lambda}{\text{d}\tau}\right)^{-1}=\cfrac{2\gamma\lambda+T_{n_1}+Be^{-\lambda\tau}-\tau(B\lambda+M_{n_1})e^{-\lambda\tau}}{\lambda(B\lambda+M_{n_1})e^{-\lambda\tau}}.
$$
By Eq.\eqref{c-eq2} and Eq.\eqref{re_im}, we have
\begin{equation*}
\begin{array}{ll}
\text{Re}\Big(\cfrac{\text{d}\lambda}{\text{d}\tau}\Big)^{-1}\Big|_{\tau =\tau_{n_1}^j}
&=\text{Re}\Big[\cfrac{(2\gamma\lambda+T_{n_1})e^{\lambda\tau}}{\lambda(B\lambda+M_{n_1})}+\cfrac{B}{\lambda(B\lambda+M_{n_1})}\Big]_{\tau =\tau_{n_1}^j}\\
&=\text{Re}\Big[\cfrac{(2\gamma\lambda+T_{n_1})e^{\lambda\tau}}{-\lambda(\gamma\lambda^2+T_{n_1}\lambda+D_{n_1})}+\cfrac{B}{\lambda(B\lambda+M_{n_1})}\Big]_{\tau =\tau_{n_1}^j}\\
&=\cfrac{2\gamma^2\omega_{n_1}^2-2\gamma D_{n_1}+T^2_{n_1}}{(\gamma\omega_{n_1}^2-D_{n_1})^2+\omega_{n_1}^2 T_{n_1}}+\cfrac{-B^2}{B^2\omega_{n_1}^2+M_{n_1}^2}\\
&=\cfrac{\sqrt{(T_{n_1}^2-2\gamma D_{n_1}-B^2)^2-4\gamma^2(D_{n_1}^2-M_{n_1}^2)}}{B^2\omega_{n_1}^2+M_{n_1}^2}.
\end{array}
\end{equation*}
Since $\text{Sign} ~\beta(\tau,d)=\text{Sign} ~\beta^{-1}(\tau,d)$, the lemma follows immediately.
\end{proof}

Let $\tau_0$ be the smallest value of $\tau^j_{n_1}$, that is
$$
\tau_0=\min\{\tau_{n_1}^j(l), 0\leq{n_1}\leq n,  j\in \mathbb{N}_0, ~\text{and}~l\in(l_n, l_{n+1}] ~\text{is defined as}~ \eqref{l_n}\}.
$$

Summarizing the above analysis, we have the following result.
\begin{theorem}
Suppose that \textsc{(H1)} and \textsc{(H1)} are satisfied, $(d,\alpha,r)\in \Gamma$, and $l_n$ is defined as in \eqref{l_n}. Then
 \begin{enumerate}
   \item If $l\in(l_n, l_{n+1}]$, there exists $n+1$ series of points $\{\tau_{n_1}^j\}$ such that the system \eqref{eq_ma_tau} undergoes a Hopf bifurcation at $\tau=\tau_{n_1}^j, 0\leq{n_1}\leq n$, $j\in\mathbb{N}_0$.
   \item Moreover, all the roots of Eq.\eqref{c-eq2} have negative real parts for
$\tau\in[0,\tau_0)$, and Eq.\eqref{c-eq2}
has at least one pair of conjugate complex roots with positive real parts for $\tau>\tau_0$. Especially for $l\leq l_1$, the Hopf bifurcation only occurs when $\tau=\tau_{0}^j ,j\in\mathbb{N}_0$ which corresponds to a spatially homogeneous periodic solution.
 \end{enumerate}
\end{theorem}

\begin{remark}
  The condition \textsc{(H2)} ensure that the positive spatially homogeneous steady state is stable to a linear homogeneous perturbation when $\tau=0$. That is, the Hopf bifurcation was entirely induced by delay $\tau$. Biologically, the population will have a periodic oscillation if the digestion period $\tau$ is greater than a critical value $\tau_0$.
\end{remark}

For the Turing instability to be realized and the spatial patterns to form, the real part of eigenvalue $\lambda$ of \eqref{c-eq2} must be greater than zero for some $n\neq 0$, moreover, there exists a real eigenvalue $\lambda^T$ pass through the origin from the left side of the complex plane to the right side. That is, if the system undergoes a  Turing bifurcation, then the characteristic equation has a simple zero eigenvalue. Hence, Eq.\eqref{c-eq2} can be written as
\begin{equation}\label{DM}
h(d,n^2):=D_n+M_n=0
\end{equation}
Clearly, Eq.\eqref{DM} is a quadratic equation with $n^2$, the critical $n^2$ can be obtained by the following formula
\begin{equation}\label{n_c}
n_2^2=\cfrac{l^2}{2d}\left(\cfrac{m^*}{(1+m^*)^2}-\cfrac{d \alpha}{a^*}\right).
\end{equation}
and the steady state is marginally stable at $n=n_2$ when
\begin{equation}\label{d_nc}
h(d,n_2^2)=0
\end{equation}
Solving \eqref{d_nc} for $d$, we can get that
\begin{equation}\label{d_0}
d_0(\alpha,r)=\cfrac{\alpha(r-1)(1-\alpha r)^2}{(1-\alpha)^3(2\sqrt{1-\alpha r}+2-\alpha r)}
\end{equation}

Now we are in the position to investigate the Turing instability that driven by diffusion coefficient $d$. Using the similar method in Lemma \ref{trans_tau}, we can obtain the following transversality without difficulty.
\begin{lemma}\label{trans_d}
Suppose that \textsc{(H1)} and \textsc{(H2)} are satisfied. Then
$$
\cfrac{\partial}{\partial d}\alpha(\tau, d_0)<0.
$$
where $\alpha(\tau, d)$ is the real eigenvalue of the characteristic
equation \eqref{c-eq2} .
\end{lemma}
\begin{lemma}\label{d_lemma}
Suppose that \textsc{(H1)} and \textsc{(H2)} are satisfied. Then
\begin{enumerate}
 \item  If $d>d_0(\alpha, r)$, there is no Turing instability;
 \item  If $d<d_0(\alpha, r)$, there exists at least one $n\in\mathbb{N}$ such that $h(d,n^2)>0$, and the system undergoes a Turing bifurcation at $d=d_0$.
\end{enumerate}
\end{lemma}
\begin{proof}
It is easy to see $h(d, n_2^2)>0$ from the definition \eqref{DM} and \eqref{d_nc} when $d>d_0$ and $h(d, n_2^2)<0$ when $d<d_0$.
\end{proof}

\begin{remark}
  Noting that $d$ is only the diffusion coefficient of mussel, while the diffusion coefficient of algae is rescaled to $\frac{1}{\gamma}$. Lemma \ref{d_lemma} indicates that if mussel diffusivity is sufficiently large, there is no spatial patterns, but if it less than the threshold,
Turing instability will happenㄛ and we shall observe the spatial distribution of the two species. This result is also suitable for high dimensional space where the patterns are more complicated and interesting.
\end{remark}

The following Turing-Hopf bifurcation theorem is a direct result of the previous analysis.
\begin{theorem}\label{theorem_TH}
  Assume that \textsc{(H1)} and \textsc{(H2)} are satisfied, and $l\in(l_n, l_{n+1}]$ with $l_n$ is defined as in \eqref{l_n}.
Then
\begin{enumerate}
  \item the constant steady state $E_*(m^*,a^*)$ is locally asymptotically stable when $\tau<\tau_0$ and $d>d_0$.

  \item the $(n_1+1)$th equation of \eqref{c-eq2} has a pair of simple pure imaginary roots $\pm i\omega_{n_1}$, the $(n_2+1)$th equation of \eqref{c-eq2} has a simple zero when $\tau=\tau_{n_1}^j, d=d_0$, $j\in\mathbb{N}_0$, with $d_0$ is defined by \eqref{d_0}, $n_2$ is defined by \eqref{n_c} and $n_1$ is define as $0\leq n_1\leq n$, if $n_2>n$ or $0\leq n_1\leq n$, $n_1\neq n_2$, if $n_2\leq n$, and all other eigenvalues have non-zero real parts.

  \item the system
\eqref{eq_ma_tau} undergoes a Turing-Hopf bifurcation at $(\tau_{n_1}^j, d_0)$, where $n_1$ is well defined in (2).
\end{enumerate}

Moreover, if $l\leq l_1$, then the characteristic equation \eqref{c-eq2} only has a pair of imaginary roots $\pm i\omega_0$ with $n_1=0$ and a simple zero with $n_2>0$, and all other eigenvalues with $(\tau, d)=(\tau_0, d_0)$ have strictly negative real parts.

\end{theorem}
\begin{remark}
Since $\cfrac{\partial}{\partial \lambda}E_{n_2}(0,\tau,d_0)=T_{n_2}+b-\tau M_{n_2}>0$ for any $\tau>0$, then $\lambda(\tau,d_0)=0$ is a simple root of characteristic equation \eqref{c-eq2}. This is determined by the model, in other words, $0$ may be a eigenvalue with multiplicity two for some other models(see \cite{AnJ}), in that case, there might exist a $\tau^*$ such that
$E_{n_2}(0,\tau^*,d_0)=0$
and $\cfrac{\partial}{\partial \lambda}E_{n_2}(0,\tau^*,d_0)=0$ , and if other eigenvalues have non-zero real part, the system will undergoes a Bogdanov-Takens bifurcation or even a Turing-Turing-Hopf bifurcation at
$(\tau^*, d_0)$ .
\end{remark}

\begin{figure}[htp]
\centering
\subfigure[] {\includegraphics[width=0.45\textwidth]{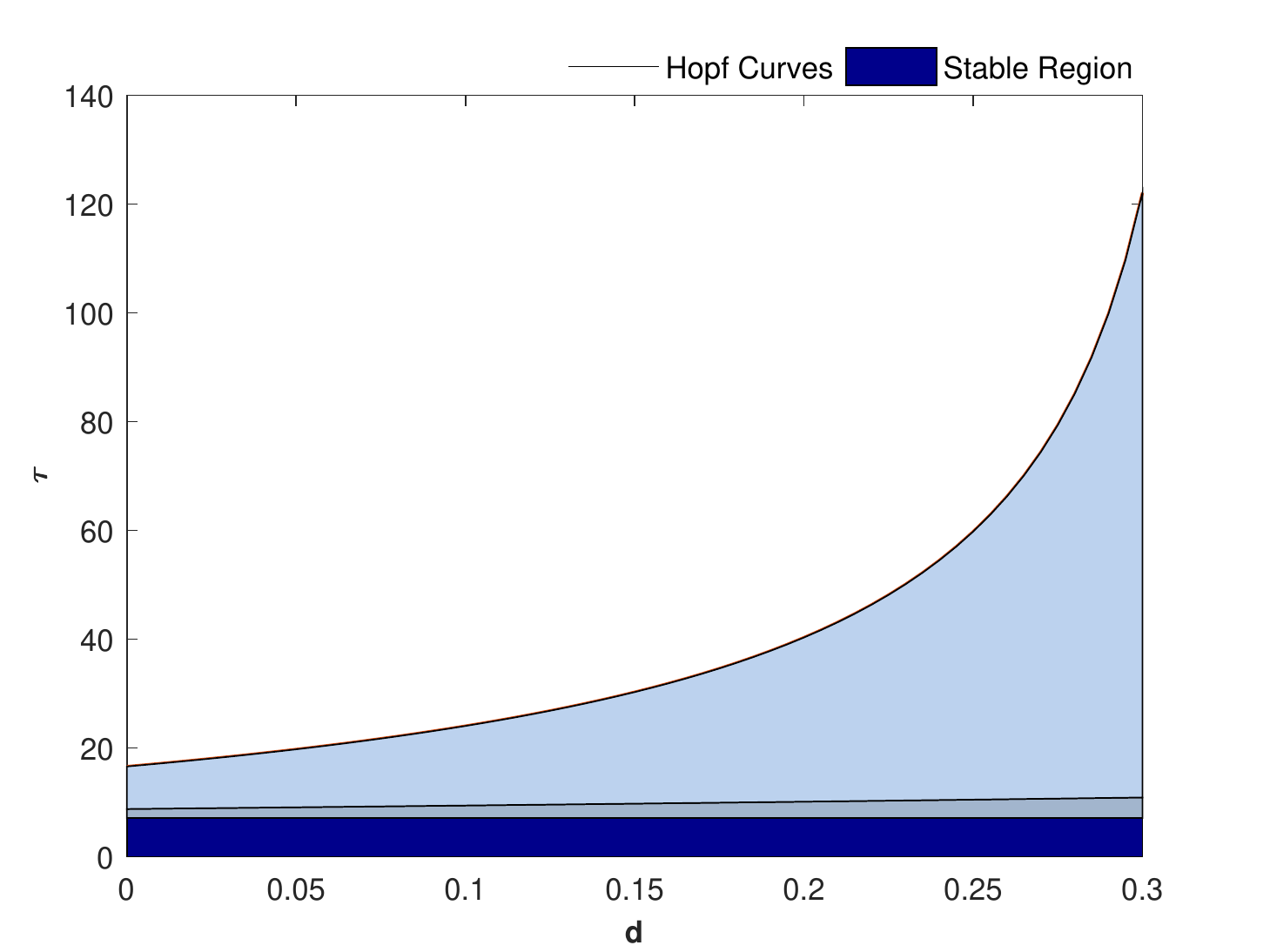}}
\subfigure[] {\includegraphics[width=0.45\textwidth]{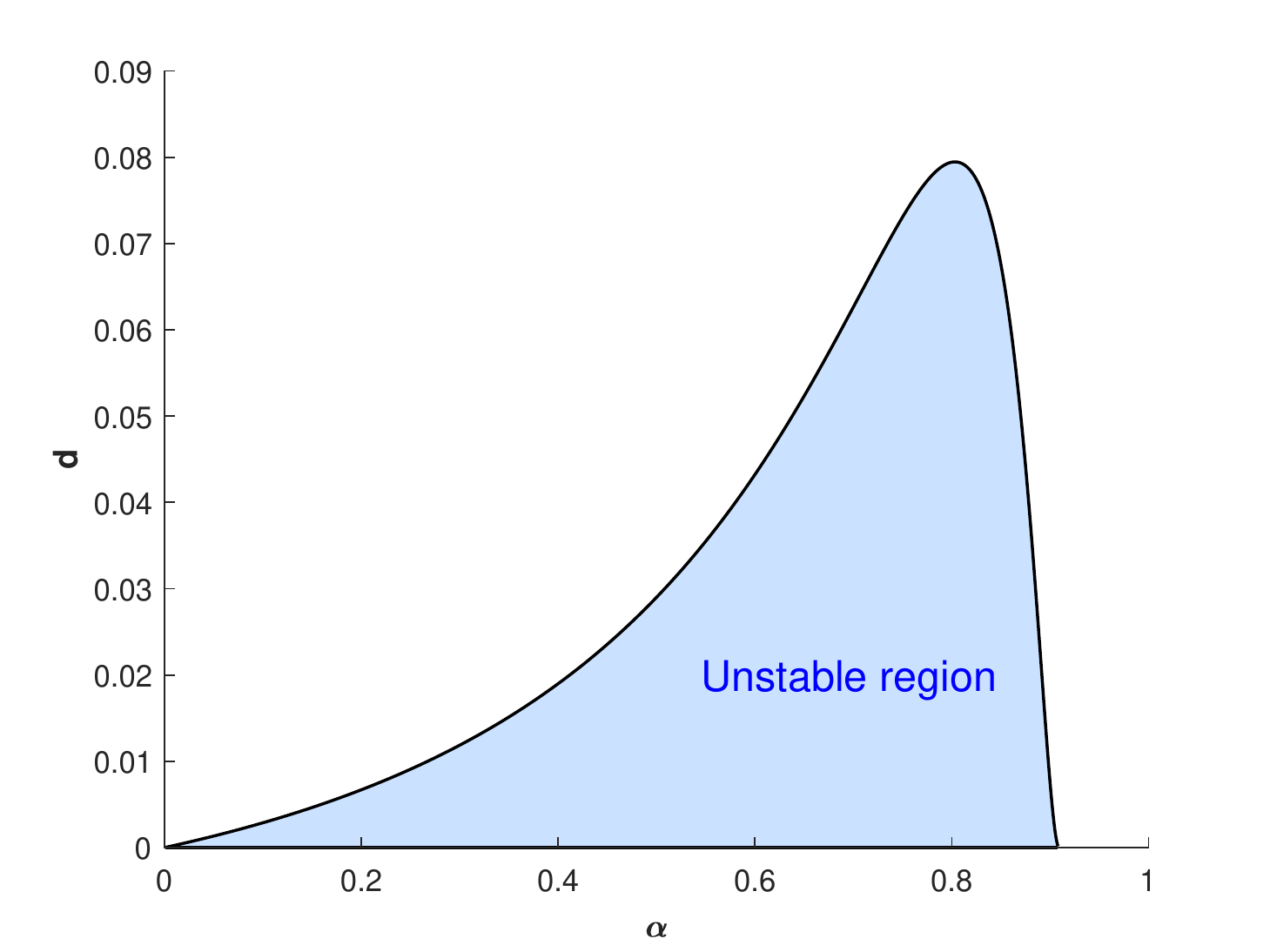}}
\caption{(a) The Hopf bifurcation curves and stable region in $d-\tau$ plane, the solid lines are Hopf bifurcation curves with $n=0, 1, 2$ from bottom to top respectively, and values of parameters are chosen as follows: $\gamma=4, r=1.1, \alpha=0.65$; (b) The critical curve of Turing bifurcation and unstable region in $\alpha-d$ plane with $r=1.1$.}
\label{fig-H-T}       
\end{figure}

Clearly, $\tau^0_{n_1}<\tau^j_{n_1}$ for all $j>0$, and through a mass of numerical simulations, we have observed the trend of $\tau^0_{n_1}$ as $n_1$ get bigger. The result reveals that the smallest value $\tau_0$ is always obtained when $n_1=0$. Fig.\ref{fig-H-T}(a)
is the geometric interpretation under the set of parameters that we used to run the numerical solutions in Section \ref{dynamic classifiction}. Hence, the first Turing-Hopf bifurcation point $(\tau_0, d_0)$ is $(\tau^0_0, d_0)$, our results below is the detailed analysis about this point and its neighborhood.

\section{Normal form of Turing-Hopf bifurcation}

In this section, we shall study the spatiotemporal dynamics of system \eqref{eq_ma_tau} by using the center manifold reduction \cite{LSW, Wu} and normal form theory \cite{AnJ, Far, SJL}. The amplitude equations are finally obtained to describe to dynamics near the critical Turing-Hopf bifurcation point, the truncated normal form is exactly the same to that of the ODE system with Hopf-Hopf bifurcation. In what follows, we will give a specific process and some explicit calculation formulas.

Let $\widetilde{m}(x,t)=m(x,t)-m^*$, $\widetilde{a}(x,t)=a(x,t)-a^*$, and $t\mapsto t/\tau$, dropping the tilde, then we have
\begin{equation}\label{eq_origin}
\begin{cases}
\cfrac{\partial m}{\partial t}=\tau[d\Delta m+r^2a^{*2}m^* m_t(-1)+r m^*a_t(-1)+f_1(m_t,a_t)], & x\in\Omega,~t>0,\\
\gamma \cfrac{\partial a}{\partial t}=\tau[\Delta a-a^*m-(\alpha+m^*)a+f_2(m_t,a_t)], & x\in\Omega,~t>0,\\
\cfrac{\partial m}{\partial n_1}=0,~\cfrac{\partial a}{\partial {n_1}}=0, & x\in\partial \Omega,~t>0,\\
   m(x,t)=m_0(x,t)-m^*,~ a(x,t)=a_0(x,t)-a^*, & x\in\Omega,-1\leq t\leq 0,
\end{cases}
\end{equation}
where
$$
m_t(\theta)=m(x,t+\theta),~a_t(\theta)=a(x,t+\theta),~~\theta\in [-1,0],
$$
and for $\phi_1, \phi_2\in \mathcal{C}:=C([-1,0],X_{\mathbb{C}})$
\begin{equation}\label{f1}
\begin{array}{ll}
f_1(\phi_1,\phi_2)=&r\phi_1(0)\phi_2(-1)-\cfrac{m^*}{(1+m^*)^3}\phi_1^2(-1)+\cfrac{1}{(1+m^*)^2}\phi_1(0)\phi_1(-1)\\
&+\cfrac{m^*}{(1+m^*)^4}\phi_1^3(-1)-\cfrac{1}{(1+m^*)^3}\phi_1(0)\phi_1^2(-1)+\mathcal{O}(4),\\
f_2(\phi_1,\phi_2)=&-\phi_1(0)\phi_2(0).
\end{array}
\end{equation}

In order to study the dynamics near the Turing-Hopf bifurcation,
we need to extend the domain of solution
operator to a space of some discontinuous:
$$
\mathcal{BC}:=\big\{\psi:[-1,0]\rightarrow X_{\mathbb{C}}~|~\psi~\text{is continuous on}~ [-1, 0), \exists\lim_{\theta\rightarrow 0^-}\psi(\theta)\in X_{\mathbb{C}}\big\}
$$

Let $\mu=\mu_0+\mu_{\varepsilon}$, where $\mu=(\tau, d)$, $\mu_0=(\tau_0, d_0)$, and $\mu_{\varepsilon}=(\tau_{\varepsilon}, d_{\varepsilon})$. Then system \eqref{eq_origin} undergoes a Turing-Hopf bifurcation at the equilibrium $(0,0)$ when $\mu_{\varepsilon}=(0,0)$ and we can rewrite system \eqref{eq_origin} in an abstract form in the space $\mathcal{BC}$ as
\begin{equation}\label{eq_abs_origin}
\cfrac{d}{dt}U(t)=A U_t+X_0\mathcal{F}(\mu_\varepsilon,U_t),
\end{equation}
where
\begin{equation*}
    X_0(\theta)=\begin{cases}
                    0 , &\theta\in[-1, 0),\\
                    I, \quad &\theta=0.
                    \end{cases}
\end{equation*}
and $A$ is a operator from $\mathcal{C}_0^1:=\big\{\varphi\in \mathcal{C}: \dot{\varphi}\in \mathcal{C}, \varphi(0)\in \text{dom}(\Delta)\big\}$ to $\mathcal{BC}$ \cite{Paz}, defined by
$$
A\varphi=\dot{\varphi}+X_0\big[\tau_0 D_0\Delta \varphi(0)+\tau_0 L_0(\varphi)-\dot{\varphi}(0)\big]
$$
with $D_0=D(\mu_0)$, $L_0:\mathcal{C}\rightarrow X_{\mathbb{C}}$ is a linear operator given by $L_0(\varphi)= L(\mu_0)(\varphi)$ with
$$ L(\mu)(\varphi)=L_1\varphi(0)+L_2\varphi(-1)$$
and $\mathcal{F}:\mathbb{R}^2 \times \mathcal{C}\to X_{\mathbb{C}}$ is a nonlinear operator and defined by
\begin{equation*}
\mathcal{F}(\mu_\varepsilon,\phi)=(\tau_0+\tau_{\varepsilon})\big[D\Delta\varphi(0) +L(\mu)(\varphi)+F(\mu_\varepsilon,\varphi)\big]-A\varphi(0)
\end{equation*}
with
\begin{equation}\label{F}
F(\mu_\varepsilon,\varphi)=(\tau_0+\tau_\varepsilon)(f_1(\varphi_1,\varphi_2),~\gamma^{-1}f_2(\varphi_1,\varphi_2))^{^T},
\end{equation}
where $f_1$ and $f_2$ are defined by \eqref{f1}.

We denote
$$
b_n=\cfrac{\cos (nx/l)}{\|\cos(nx/l)\|},~ ~\beta_n=\{\beta_n^1, \beta_n^2\}=\{(b_n, 0)^{T}, (0, b_n)^{T}\},
$$
where
$$
\|\cos(nx/l)\|=\left(\int_0^{l\pi}\cos^2(nx/l)\text{d}x\right)^{\frac{1}{2}}.
$$

For $\phi=(\phi^{^{(1)}},\phi^{^{(2)}})^{T}\in\mathcal{C}$, denote
$$
\phi_n=\langle \phi,\beta_n\rangle=\left(\langle \phi,\beta_n^1\rangle, \langle \phi,\beta_n^2\rangle\right)^{T}.
$$
Define $ A_{\varepsilon, n}$ as
\begin{equation}\label{An}
   A_{\varepsilon, n}(\phi_n(\theta)\beta_n)=\begin{cases}
                        \dot{\phi}_n(\theta)\beta_n,& \theta\in[-1,0), \\
                        \int_{-1}^{0}\text{d}\eta_n(\mu_\varepsilon,\theta)\phi_n(\theta)\beta_n ,\qquad &\theta=0,
                     \end{cases}
\end{equation}
where
$$
\int_{-1}^{0}\text{d}\eta_n(\mu_\varepsilon,\theta)\phi_n(\theta)=-\cfrac{n^2}{l^2} (\tau_0+\tau_\varepsilon)D\phi_n(0)+L_{\varepsilon,n}(\phi_n),
$$
with
$$
L_{\varepsilon, n}(\phi_n)=(\tau_0+\tau_\varepsilon)L_1\phi_n(0)+(\tau_0+\tau_\varepsilon)L_2\phi_n(-1),
$$
and
\begin{equation*}
    \eta_n(\mu_\varepsilon,\theta)=\begin{cases}\begin{array}{ll}
    -(\tau_0+\tau_\varepsilon)L_2, & \theta=-1,\\
    0, & \theta\in(-1,0),\\
    (\tau_0+\tau_\varepsilon)\left(L_1-\cfrac{n^2}{l^2}D\right), & \theta=0.
    \end{array}\end{cases}
\end{equation*}

Denote $A^*$ as the adjoint operator of $A$ on $\mathcal{C}^*:=C([0,1],X_{\mathbb{C}})$.
\begin{equation*}
A^*\psi(s)=\begin{cases}\begin{array}{ll}
-\dot{\psi}(s),~& s\in(0,1],\\
\sum_{n=0}^\infty\int_{-1}^0\psi_n(-\theta)\text{d}\eta_n^{^T}(0,\theta)\beta_n,~&s=0.
\end{array}\end{cases}
\end{equation*}
Now, we introduce the bilinear formal $(\cdot,\cdot)$ on $\mathcal{C}^*\times\mathcal{C}$
\begin{equation*}
(\psi,\phi)=\sum_{j_1,j_2=0}^\infty(\psi_{j_1},\phi_{j_2})\int_\Omega b_{j_1}b_{j_2}\text{d}x, ~k=1,2
\end{equation*}
where
$$
\psi=\sum_{n=0}^\infty \psi_n \beta_n\in\mathcal{C}^*,~\phi=\sum_{n=0}^\infty \phi_n \beta_n\in\mathcal{C},
$$
and
$$
\phi_n\in C:=C([-1,0],\mathbb{C}^2),~~\psi_n\in C^*:=C([0,1],\mathbb{C}^2).
$$
Notice that
 $$
 \int_\Omega b_{j_1}b_{j_2}\text{d}x=0~~\mbox{for}~~j_1\neq j_2,
 $$
we have
\begin{equation*}
(\psi,\phi)=\sum_{n=0}^\infty(\psi_n,\phi_n)|b_n|^2:=\sum_{n=0}^\infty(\psi_n,\phi_n)_n|b_n|^2,
\end{equation*}
where $(\cdot,\cdot)_n$(or $(\cdot,\cdot)$) is the bilinear form defined on $C^*\times C$
\begin{equation*}
(\psi_n,\phi_n)_n=\psi_n(0)\phi_n(0)-\int_{-1}^0\int_{\xi=0}^\theta\psi_n(\xi-\theta)
\text{d}\eta_n(0,\theta)\phi_n(\xi)\text{d}\xi.
\end{equation*}

Let $\{\phi_1(\theta)b_{n_1}, \phi_2(\theta)b_{n_2}\}$ and $\{\psi_1(s)b_{n_1}, \psi_2(s)b_{n_2}\}$ are the eigenfunctions of $A$ and its dual $A^*$ relative to $\Lambda=\{i\omega_0\tau_0, 0\}$ such that $\phi_1, \phi_2\in C$, $\psi_1, \psi_2\in C^*$ and
$$
(\psi_1,\phi_1)_1=1,~~(\psi_1,\overline{\phi}_1)_1=0,~~(\psi_2,{\phi}_2)_2=1
$$
By a straight forward calculation, we have
$$
\begin{array}{ll}
\phi_1(\theta)=q(0)e^{i\omega_0\tau_0\theta},& \psi_1(s)=M_1 q^*(0)e^{-i\omega_0\tau_0 s},\\
\phi_2(\theta)=p(0),&\psi_2(s)=M_2p^*(0),
\end{array}$$
where $q(0)=(1, q_1)^{^T}$, $q^*(0)=(q_2, 1)$, $p(0)=(1, p_1)^{^T}$, $p^*(0)=(p_2, 1)$ and
$$\begin{array}{ll}
&q_1=\cfrac{a^*}{i\gamma\omega_0+\alpha+m^*},~~q_2=\cfrac{i \gamma \omega_0+\alpha+m^*}{r m^*e^{-i\omega_0\tau_0}},~~p_1=\cfrac{-a^*}{\frac{n^2_2}{l^2}+\alpha+m^*},~~p_2=\cfrac{\frac{n^2_2}{l^2}+\alpha+m^*}{r m^*},\\
&M_1=\cfrac{1}{q_1+q_2+\tau_0q_2e^{-i\omega_0\tau_0}(r^2a^{*2}m^*+rm^*q_1)},~~M_2=\cfrac{1}{p_1+p_2+\tau_0r m^*p_2(ra^{*2}+p_1)}

\end{array}$$

Denote $\Phi_1=(\phi_1, \overline{\phi}_1)$, $\Psi_1=(\psi^{^T}_1, \bar{\psi}^{^T}_1)^{^T}$ and $\Phi_2=\phi_2$, $\Psi_2=\psi_2$. From the discussion  above, we know that the phase space $\mathcal{BC}$ can be decomposed as
$$
\mathcal{BC}=\mathcal{P}\bigoplus\text{Ker}\pi,
$$
where $\mathcal{P}$ is the is the 3-dimensional center subspace spanned by the basis eigenfunctions of the linear operator $A$ associated with the eigenvalues $\{\pm i\omega_0\tau_0, 0\}$ and $\text{Ker}\pi$ is the complementary space of $\mathcal{P}$ with $\pi: \mathcal{BC}\rightarrow \mathcal{P}$ is the projection defined by
$$
\pi\varphi=\sum^2_{k=1}\Phi_k(\Psi_k,<\varphi(\cdot),\beta_{n_k}>)_k \cdot \beta_{n_k}
$$
with $c\cdot \beta_{n_k}=c_1 \beta_{n_k}^1+c_2 \beta_{n_k}^2$ for $c=(c_1, c_2)^{_T}\in C$.

Then $U_t\in \mathcal{C}_0^1$ can be decomposed as
\begin{equation*}
\begin{split}
U_t(\theta)&=\sum^2_{k=1}\Phi_k(\theta)(\Psi_k,<U_t,\beta_{n_k}>)_k \beta_{n_k}+y(\theta)\\
&=\sum^2_{k=1}\Phi_k(\theta)\tilde{z}_k(t)\cdot \beta_{n_k}+y(\theta),
\end{split}
\end{equation*}
with $\tilde{z}_1=(z_1, \bar{z}_1)$, $\tilde{z}_2=z_2$, and $y\in Q^1:=\mathcal{C}_0^1\bigcap\text{Ker}\pi$. Then system \eqref{eq_abs_origin}
on $\mathcal{BC}$ is equivalent to the following system
\begin{equation}
\begin{split}
  & \dot{z}=Bz+\Psi(0)\left(\begin{array}{l}
     <\mathcal{F}(\mu_\varepsilon,\sum^2_{k=1}\Phi_k\tilde{z}_k(t)\cdot \beta_{n_k}+y), \beta_{n_1}>\\
      <\mathcal{F}(\mu_\varepsilon,\sum^2_{k=1}\Phi_k\tilde{z}_k(t)\cdot \beta_{n_k}+y), \beta_{n_2}>
   \end{array}\right),\\
  & \cfrac{d}{dt}y=A_{Q^1}y-(I-\pi)X_0\mathcal{F}\big(\mu_\varepsilon,\sum^2_{k=1}\Phi_k\tilde{z}_k(t)\cdot \beta_{n_k}+y\big),
\end{split}
\end{equation}
where $z=(z_1, \bar{z}_1, z_2)$, $B=\text{diag}(i\omega_0\tau_0, -i\omega_0\tau_0, 0)$, $\Psi=\text{diag}(\Phi_1, \Phi_2)$, and $A_{Q^1}$ is the restriction of $A$ as an operator from $Q^1$ to $\text{Ker}\pi$.

From the Theorem 3.2 in \cite{AnJ}, the normal forms of system \eqref{eq_ma_tau} up to three order near a Turing-Hopf singularity $\mu=\mu_0$ are obtained
\begin{equation}\label{z_norm}
\begin{array}{ll}
  \dot{z}_1=~~i\omega_0\tau_0z_1+&\cfrac{1}{2}f^{11}_{11}\alpha_1z_1+\cfrac{1}{2}f^{11}_{21}\alpha_2z_1+\cfrac{1}{6}g^{11}_{210}z^2_1\bar{z}_1
  +\cfrac{1}{6}g^{11}_{102}z_1z^2_2+h.o.t.\\
   \dot{\bar{z}}_1=-i\omega_0\tau_0z_1+&\cfrac{1}{2}f^{12}_{11}\alpha_1\bar{z}_1+\cfrac{1}{2}f^{12}_{21}\alpha_2\bar{z}_1
   +\cfrac{1}{6}g^{12}_{210}\bar{z}^2_1 z_1+\cfrac{1}{6}g^{12}_{102}\bar{z}_1z^2_2+h.o.t.\\
   z_2=&\cfrac{1}{2}f^{13}_{12}\alpha_1z_2+\cfrac{1}{2}f^{13}_{22}\alpha_2z_2+\cfrac{1}{6}g^{13}_{111}z_1\bar{z}_1z_2
   +\cfrac{1}{6}g^{13}_{003}z_2^3+h.o.t.
  \end{array}
\end{equation}
with $(\alpha_1,\alpha_2)=(\tau_{\varepsilon}, d_{\varepsilon})$, $f^{12}_{mn}=\overline{f^{11}_{mn}}$, $g^{12}_{mnk}=\overline{g^{11}_{mnk}}$ and
\begin{equation*}
  \begin{array}{ll}
    f^{11}_{11}&=2\psi_1(0)\Big[\cfrac{\partial}{\partial\tau}A(\mu_0)\phi_1(0)+\cfrac{\partial}{\partial\tau}B(\mu_0)\phi_1(-1)\Big],\\
    f^{11}_{21}&=2\psi_1(0)\Big[\cfrac{\partial}{\partial d}A(\mu_0)\phi_1(0)+\cfrac{\partial}{\partial d}B(\mu_0)\phi_1(-1)\Big],\\
    f^{13}_{12}&=2\psi_2(0)\Big[-\cfrac{n_2^2}{l^2}\cfrac{\partial}{\partial \tau}\widetilde{D}(\mu_0)\phi_2(0)+
                  \cfrac{\partial}{\partial \tau}A(\mu_0)\phi_2(0)+\cfrac{\partial}{\partial \tau}B(\mu_0)\phi_2(-1)\Big],\\
    f^{13}_{22}&=2\psi_2(0)\Big[-\cfrac{n_2^2}{l^2}\cfrac{\partial}{\partial d}\widetilde{D}(\mu_0)\phi_2(0)+
                  \cfrac{\partial}{\partial d}A(\mu_0)\phi_2(0)+\cfrac{\partial}{\partial d}B(\mu_0)\phi_2(-1)\Big],\\
\end{array}
\end{equation*}
\begin{equation*}
\begin{array}{ll}
    g^{11}_{210}&=f^{11}_{210}+\cfrac{3}{2i\omega_0\tau_0}\Big(-f^{11}_{110}f^{11}_{200}+f^{11}_{110}f^{12}_{110}+\cfrac{2}{3}f^{11}_{020}f^{12}_{200}\Big)\\
                 &~~+\cfrac{3}{2}\psi_1(0)\Big[S_{yz_1}(<h_{110}(\theta)b_{n_1},b_{n_1}>)+S_{y\bar{z}_1}(<h_{200}(\theta)b_{n_1},b_{n_1}>)\Big],\\
    g^{11}_{102}&=f^{11}_{102}+\cfrac{3}{2i\omega_0\tau_0}\Big(-2f^{11}_{002}f^{11}_{200}+f^{12}_{002}f^{11}_{110}+2f^{11}_{002}f^{13}_{101}\Big)\\
                 &~~+\cfrac{3}{2}\psi_1(0)\Big[S_{yz_1}(<h_{002}(\theta)b_{n_1},b_{n_1}>)+S_{yz_2}(<h_{101}(\theta)b_{n_2},b_{n_1}>)\Big],\\
    g^{13}_{111}&=f^{13}_{111}+\cfrac{3}{2i\omega_0\tau_0}\Big(-2f^{13}_{101}f^{11}_{110}+f^{13}_{011}f^{12}_{110}\Big)\\
                 &~~+\cfrac{3}{2}\psi_2(0)\Big[S_{yz_1}(<h_{011}(\theta)b_{n_1},b_{n_2}>)+S_{y\bar{z}_1}(<h_{101}(\theta)b_{n_1},b_{n_2}>)
                 +S_{yz_2}(<h_{110}(\theta)b_{n_2},b_{n_2}>)\Big],\\
    g^{13}_{003}&=f^{13}_{003}+\cfrac{3}{2i\omega_0\tau_0}\Big(-f^{11}_{002}f^{13}_{101}+f^{12}_{002}f^{13}_{011}\Big)
                    +\cfrac{3}{2}\psi_2(0)\Big[S_{yz_2}(<h_{002}(\theta)b_{n_2},b_{n_2}>)\Big],
  \end{array}
\end{equation*}
where $\widetilde{D}=\tau D$, $A(\mu)=\tau L_1$, $B(\mu)=\tau L_2$, $f^{12}_{mnk}=\overline{f^{11}_{mnk}}$, and
\begin{equation*}
  \begin{array}{ll}
    f^{11}_{mnk}=\cfrac{1}{\sqrt{l\pi}}\psi_1(0)F_{mnk},~~~ & f^{13}_{mnk}=\cfrac{1}{\sqrt{l\pi}}\psi_2(0)F_{mnk},~~\text{when}~~m+n+k=2,\\    f^{11}_{mnk}=\cfrac{1}{l\pi}\psi_1(0)F_{mnk},~~~  & f^{13}_{mnk}=\cfrac{1}{l\pi}\psi_2(0)F_{mnk},~~\text{when}~~m+n+k=3.\\
  \end{array}
\end{equation*}

\begin{equation*}
  \begin{array}{ll}
    <h_{200}(\theta)b_{n_1},b_{n_1}>=\cfrac{e^{2i\omega_0\tau_0\theta}}{l\pi}\Big[2i\omega_0\tau_0-\tau_0 L_0(e^{2i\omega_0\tau_0\cdot}Id)\Big]^{-1}F_{200}
            -\cfrac{1}{i\omega_0\tau_0\sqrt{l\pi}}\Big[f^{11}_{200}\phi_1(\theta)+\cfrac{1}{3}f^{12}_{200}\overline{\phi}_1(\theta)\Big],\\
    <h_{110}(\theta)b_{n_1},b_{n_1}>=-\cfrac{1}{\sqrt{l\pi}}\Big[\tau_0L_0(Id)\Big]^{-1}F_{110}+\cfrac{1}{i\omega_0\tau_0\sqrt{l\pi}}
            \Big[f^{11}_{110}\phi_1(\theta)-f^{12}_{110}\overline{\phi}_1(\theta)\Big],\\
    <h_{110}(\theta)b_{n_2},b_{n_2}>= <h_{110}(\theta)b_{n_1},b_{n_1}>,\\
    <h_{101}(\theta)b_{n_2},b_{n_1}>=\cfrac{e^{i\omega_0\tau_0\theta}}{l\pi}\Big[i\omega_0\tau_0+\cfrac{n_2^2}{l^2}\widetilde{D}(\mu_0)
           -\tau_0L_0(e^{i\omega_0\tau_0 \cdot}Id)\Big]^{-1}F_{101}-\cfrac{1}{i\omega_0\tau_0\sqrt{l\pi}}f^{13}_{101}\phi_2(0),\\
  \end{array}
\end{equation*}
\begin{equation*}
  \begin{array}{ll}
   <h_{011}(\theta)b_{n_1},b_{n_2}>=\cfrac{e^{-i\omega_0\tau_0\theta}}{l\pi}\Big[-i\omega_0\tau_0+\cfrac{n_2^2}{l^2}\widetilde{D}(\mu_0)
          -\tau_0L_0(e^{-i\omega_0\tau_0\cdot}Id)\Big]^{-1}F_{011}+\cfrac{1}{i\omega_0\tau_0\sqrt{l\pi}}f^{13}_{011}\phi_2(0),\\
   <h_{002}(\theta)b_{n_1},b_{n_1}>=-\cfrac{1}{l\pi}\Big[\tau_0L_0(Id)\Big]^{-1}F_{002}+\cfrac{1}{i\omega_0\tau_0\sqrt{l\pi}}
            \Big[f^{11}_{002}\phi_1(\theta)-f^{12}_{002}\overline{\phi}_1(\theta)\Big],\\
   <h_{002}(\theta)b_{n_2},b_{n_2}>=\cfrac{1}{2l\pi}\Big[\cfrac{(2n_2)^2}{l^2}\widetilde{D}(\mu_0)-\tau_0L_0(Id)\Big]^{-1}F_{002}+ <h_{002}(\theta)b_{n_1},b_{n_1}>,
  \end{array}
\end{equation*}
and $S_{yz_i}(i=1,2)$, $S_{y\bar{z}_1}$ are linear operators from $Q_1$ to $X_{\mathbb{C}}$ given by
\begin{equation*}
  \begin{array}{ll}
    S_{yz_i}(\varphi)=(F_{y_1(0)z_i},~ F_{y_2(0)z_i})\varphi(0)+(F_{y_1(-1)z_i}, F_{y_2(-1)z_i})\varphi(-1),\\
    S_{y\bar{z}_1}(\varphi)=(\overline{F_{y_1(0)z_1}}, \overline{F_{y_2(0)z_1}})\varphi(0)+(\overline{F_{y_1(-1)z_1}}, \overline{F_{y_2(-1)z_1}})\varphi(-1).\\
  \end{array}
\end{equation*}
For specific expressions of formulas $F_{y_i(\cdot)z_{j}}, F_{mnk}$, please refer to Appendix.

With the cylindrical coordinate transformation:
$$
z_1=\widetilde{\rho} e^{i \sigma}, ~\bar{z}_1=\widetilde{\rho} e^{-i\sigma},~ z_2=\widetilde{\eta}
$$
and variable substitution:
$$
\rho=\sqrt{\cfrac{|\text{Re}(g^{11}_{210})|}{6}}\widetilde{\rho}, ~~\eta=\sqrt{\cfrac{|g^{13}_{003}|}{6}}\widetilde{\eta}, ~~\varepsilon=\text{Sign}\big(\text{Re}(g^{11}_{210})\big), ~~\widetilde{t}=t/\varepsilon,
$$
the amplitude equation \eqref{z_norm} can be rewritten as
\begin{equation}\label{rho-eta}
\begin{array}{ll}
  \cfrac{\text{d}\rho}{\text{d}\widetilde{t}}=\rho\Big(\epsilon_1(\mu_{\varepsilon})+\rho^2+b\eta^2\Big),\\
  \cfrac{\text{d}\eta}{\text{d}\widetilde{t}}=\eta\left(\epsilon_2(\mu_{\varepsilon})+c \rho^2+\hat{d}\eta^2\right),\\
\end{array}
\end{equation}
where
\begin{equation*}
  \begin{array}{ll}
    \epsilon_1(\mu_{\varepsilon})=\cfrac{\varepsilon}{2}\left[\text{Re}(f^{11}_{11})\tau_{\varepsilon}+\text{Re}
       (f^{11}_{21})d_{\varepsilon}\right],\\
    \epsilon_2(\mu_{\varepsilon})=\cfrac{\varepsilon}{2}\left[f^{13}_{12}\tau_{\varepsilon}+
       f^{13}_{22}d_{\varepsilon}\right],\\
    b=\cfrac{\varepsilon\text{Re}(g^{11}_{102})}{|g^{13}_{003}|},~~c=\cfrac{\varepsilon g^{13}_{111}}{|\text{Re}(g^{11}_{210})|},
       ~~\hat{d}=\cfrac{\varepsilon g^{13}_{003}}{|g^{13}_{003}|}=\pm1.
  \end{array}
\end{equation*}
Notice that $\rho\geq 0$, and $\eta$ is arbitrarily real number. Hence, system \eqref{rho-eta} always has a zero equilibrium $E_1(0, 0)$ for all
$\epsilon_1, \epsilon_2$, and three boundary equilibria
\begin{equation*}
\begin{array}{ll}
  E_2\left(\sqrt{-\epsilon_1}, 0\right), ~~~&\text{for} ~~~\epsilon_1<0\\
  E_3^{\pm}\Big(0,\pm \sqrt{-\cfrac{\epsilon_2}{\hat{d}}}\Big), ~~~&\text{for}~~~\epsilon_2\hat{d}<0,\\
  \end{array}
\end{equation*}
and two possible positive equilibria
\begin{equation*}
  \begin{array}{ll}
  E_4^{\pm}=\Big(\sqrt{\cfrac{b\epsilon_2-\hat{d}\epsilon_1}{\hat{d}-bc}}, \pm \sqrt{\cfrac{c\epsilon_1-\epsilon_2}{\hat{d}-bc}}\Big),
     ~~~\text{for} ~~~\sqrt{\cfrac{b\epsilon_2-\hat{d}\epsilon_1}{\hat{d}-bc}}>0,  \sqrt{\cfrac{c\epsilon_1-\epsilon_2}{\hat{d}-bc}}>0.
  \end{array}
\end{equation*}
There are 12 distinct types of unfoldings \cite{GuH} according to the signs of coefficients $b, c, \hat{d}$ and $\hat{d}-bc$.

\section{Numerical Simulations}\label{dynamic classifiction}
In this section, we choose a set of parameters. Under these parameters, the dynamic classification of the system \eqref{eq_ma_tau}
 near the Turing-Hopf bifurcation point is given and some simulations are carried out.
\subsection{Dynamic classification}
In this subsection, we apply the normal form method and the theoretical results obtained in previous sections to the system \eqref{eq_ma_tau}. The bifurcation diagram of system \eqref{rho-eta} with certain parameters near the Turing-Hopf bifurcation point in the $\tau_{\varepsilon}-d_{\varepsilon}$ parameter plane is firstly shown to determine the existential area of solutions, the critical lines separate the plane into six regions, and for each region, we shall given a detail analysis.

Take
$$
{\bf (A)}~~~~~~~~~~~~r=1.10,~~~\gamma=4,~~~\alpha=0.654,~~~l=6.~~~~~~~~~~~~~~~~~~~~~~~~~~~~~~~~~~~~
$$
Then $m^*=0.233073$, $a^*=0.737257$. From \eqref{tau} and \eqref{d_0}, we have $\tau_0=7.084102$ with $n_1=0$, $d_0=0.0531255$ with $n_2=6$, and by a simple calculation, we have
\begin{equation*}
  \begin{array}{ll}
  \epsilon_1 = -1.20727\times 10^{-2}\tau_{\varepsilon};~~
                    \epsilon_2 = 1.629874\times 10^{-8}\tau_{\varepsilon} + 6.085844 d_{\varepsilon};\\
  \varepsilon = -1; ~~~\hat{d} = 1;~~~b = 1.582903; ~~~c = 0.993790;~~~\hat{d}-bc = -0.573073.
  \end{array}
\end{equation*}
Then \eqref{rho-eta} becomes
\begin{equation}\label{28}
  \begin{array}{ll}
  \dot{\rho}=\rho\left(-1.20727\times 10^{-2}\tau_{\varepsilon}+\rho^2+1.582903\eta^2\right),\\
  \dot{\eta}=\eta\left(1.629874\times 10^{-8}\tau_{\varepsilon} + 6.085844 d_{\varepsilon}+0.993790\rho^2+\eta^2\right).\\
  \end{array}
\end{equation}
According to the classification for the planar vector field \eqref{rho-eta} in [Page 399, \cite{GuH}], Case Ia occurs under this set of parameters. The detailed bifurcation diagram and corresponding phase portraits are shown in Fig.\ref{fig-TH}, in which the two blue lines are two pitchfork bifurcation curves:
\begin{equation*}
\begin{array}{ll}
  T_1:~~d_{\varepsilon}=-1.253237 \times 10^{-3}\tau_{\varepsilon};\\
  T_2:~~d_{\varepsilon}=-1.971431 \times 10^{-3}\tau_{\varepsilon},
\end{array}
\end{equation*}
and the other two solid lines $L_1$ and $L_2$ are
\begin{equation*}
  L_1:~~\tau_{\varepsilon}=0;~~~~~~~L_2:~~~d_{\varepsilon}=-2.678139\times 10^{-9}\tau_{\varepsilon}.
\end{equation*}
Notice that, under the parameters {\bf (A)}, the dynamics of original system \eqref{eq_ma_tau} near the $(\tau, d)=(\tau_0, d_0)$ is topologically equivalent to that of normal form system \eqref{28} at $(\tau_{\varepsilon}, d_{\varepsilon})=(0, 0)$. For system \eqref{28}, the equilibrium in the $\rho-$axis $(E_2)$ identifies the characteristics of the solutions of \eqref{eq_ma_tau} in time, while equilibrium in the $\eta-$axis $(E_3^{\pm})$ identifies the characteristics in space. Moreover,
the positive equilibrium in the $\rho-\eta$ plane $(E_4^{\pm})$ identifies the characteristics of solutions of system \eqref{eq_ma_tau} both in time and space.

From Fig.\ref{fig-TH}, we see that the solid lines $L_1, L_2, T_1$ and $ T_2$ divide the plane into six regions, and in different regions there are different dynamics which can be summarized as follows.

When $(\tau_{\varepsilon}, d_{\varepsilon})\in D_1$, the amplitude system \eqref{28} has a stable trivial equilibrium $E_1(0,0)$, which means the constant steady state $E_*(m^*, a^*)$ of original system \eqref{eq_ma_tau} is locally asymptotically stable;

When $(\tau_{\varepsilon}, d_{\varepsilon})$ passes through $L_1$ into $D_2$, the constant steady state $E_*(m^*, a^*)$  lost its stability with a new stable spatially homogeneous periodic solution bifurcating from $E_*(m^*, a^*)$.

When $(\tau_{\varepsilon}, d_{\varepsilon})$ enters $D_3$ from $D_2$, two unstable non-constant steady states newly appear since a Turing bifurcation occurs at $L_2$. Moreover, $E_1$ of system \eqref{28} becomes an unstable node from a saddle.

When $(\tau_{\varepsilon}, d_{\varepsilon}) \in D_4$, two unstable spatially inhomogeneous periodic solutions newly appear and do coexist. The non-constant steady states become stable compared with its stability in region $D_3$.

When $(\tau_{\varepsilon}, d_{\varepsilon})$ enters $D_5$ from $D_4$, the two unstable spatially inhomogeneous solutions disappear since the parameters pass through another Turing bifurcation curve $T_2$, and the spatially homogeneous periodic solution loses its stability.

When $(\tau_{\varepsilon}, d_{\varepsilon})$ finally enters region $D_6$, the spatially homogeneous periodic solution disappears with a Hopf bifurcation occuring at $L_1$. Moreover, $E_1$ of system \eqref{28} becomes a saddle from an unstable node, and it will regain its stability when $(\tau_{\varepsilon}, d_{\varepsilon})$ passes through $L_2$ into $D_1$.

\begin{figure}[htp]
\centering
  \includegraphics[width=6in]{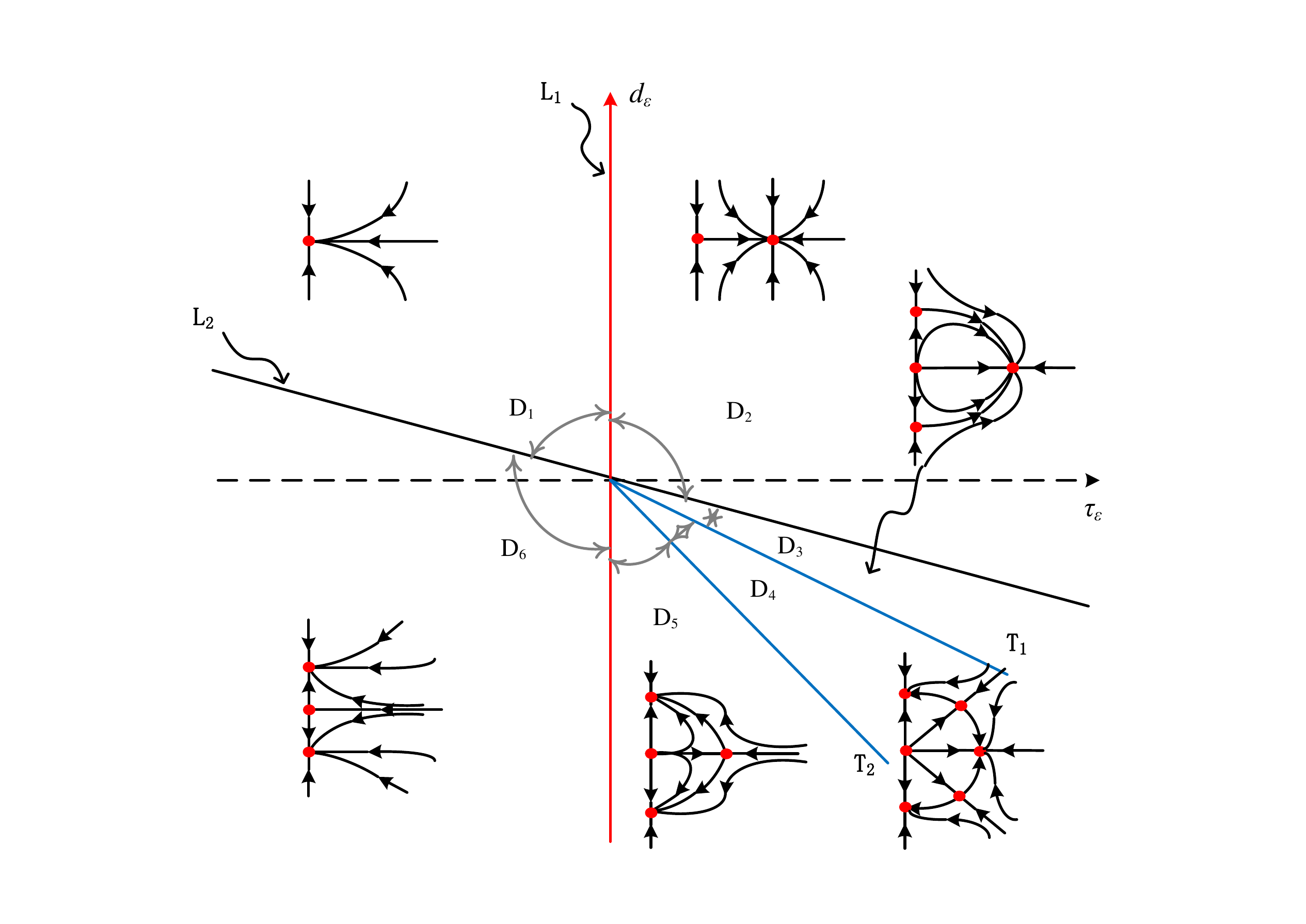}\vspace{0cm}
\caption{Bifurcation diagram of system \eqref{rho-eta} in the $\tau_{\varepsilon}-d_{\varepsilon}$ plane and the corresponding phase portraits.}
\label{fig-TH}       
\end{figure}

\subsection{Simulations}

Numerical simulations of dynamics for original system \eqref{eq_ma_tau} at the Turing-Hopf bifurcation point are carried out in this subsection. For each region in Fig.\ref{fig-TH}, we shall select a set of parameters $(\tau_{\varepsilon}, d_{\varepsilon})$, and for obvious contrast, the parameters are always selected from a rectangle, see Fig.\ref{fig-parameter}. The little pink circles represent the points that that we choose in each region.
\begin{figure}[htp]
\centering
  \includegraphics[width=5in]{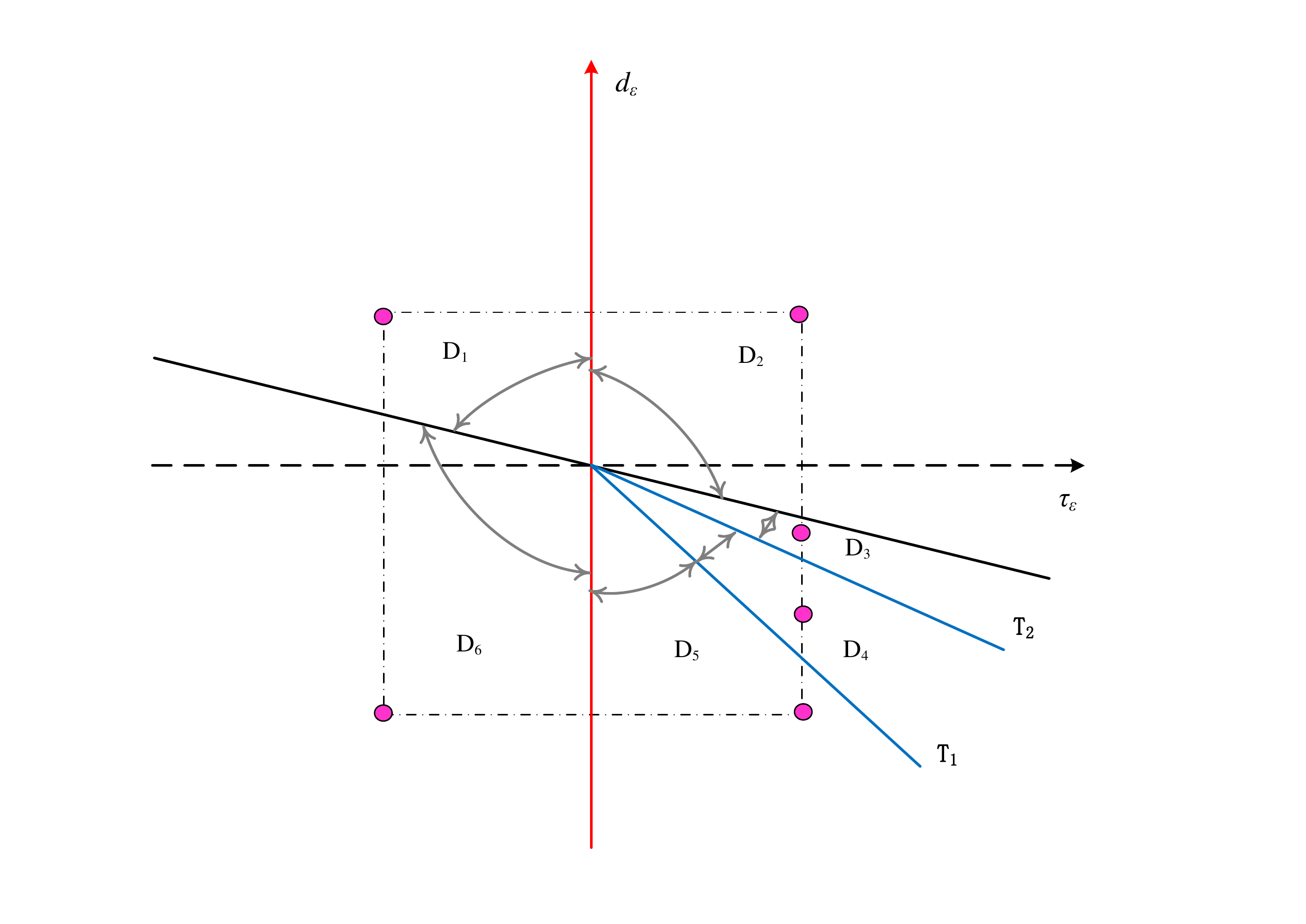}\vspace{0cm}
\caption{The selection of parameters in the $\tau_{\varepsilon}-d_{\varepsilon}$ plane.}
\label{fig-parameter}       
\end{figure}

Fig.\ref{fig-stable-pattern} shows the stable patterns in region $D_1, D_3, D_5$. The stable patterns in region $D_2$ and  $D_6$ are similar with that in $D_3$ and $D_5$, respectively. In region $D_1$, there exists a stable spatially homogeneous steady state; in $D_3$, a stable spatially homogeneous periodic solution exists, and in $D_5$, two stable spatially inhomogeneous steady states coexist. For $D_4$, the non-constant steady state and spatially homogeneous periodic solution both can be considered as the stable patterns, which is related to the initial values. In addition, some transitions that connecting two state can be observed in our numerical simulations, detailed results refer to Fig.\ref{fig-D3}--Fig.\ref{fig-D5}.

%

\begin{figure}[htp]
\begin{multicols}{3}
\begin{center}
\subfigure{\includegraphics[width=2in]{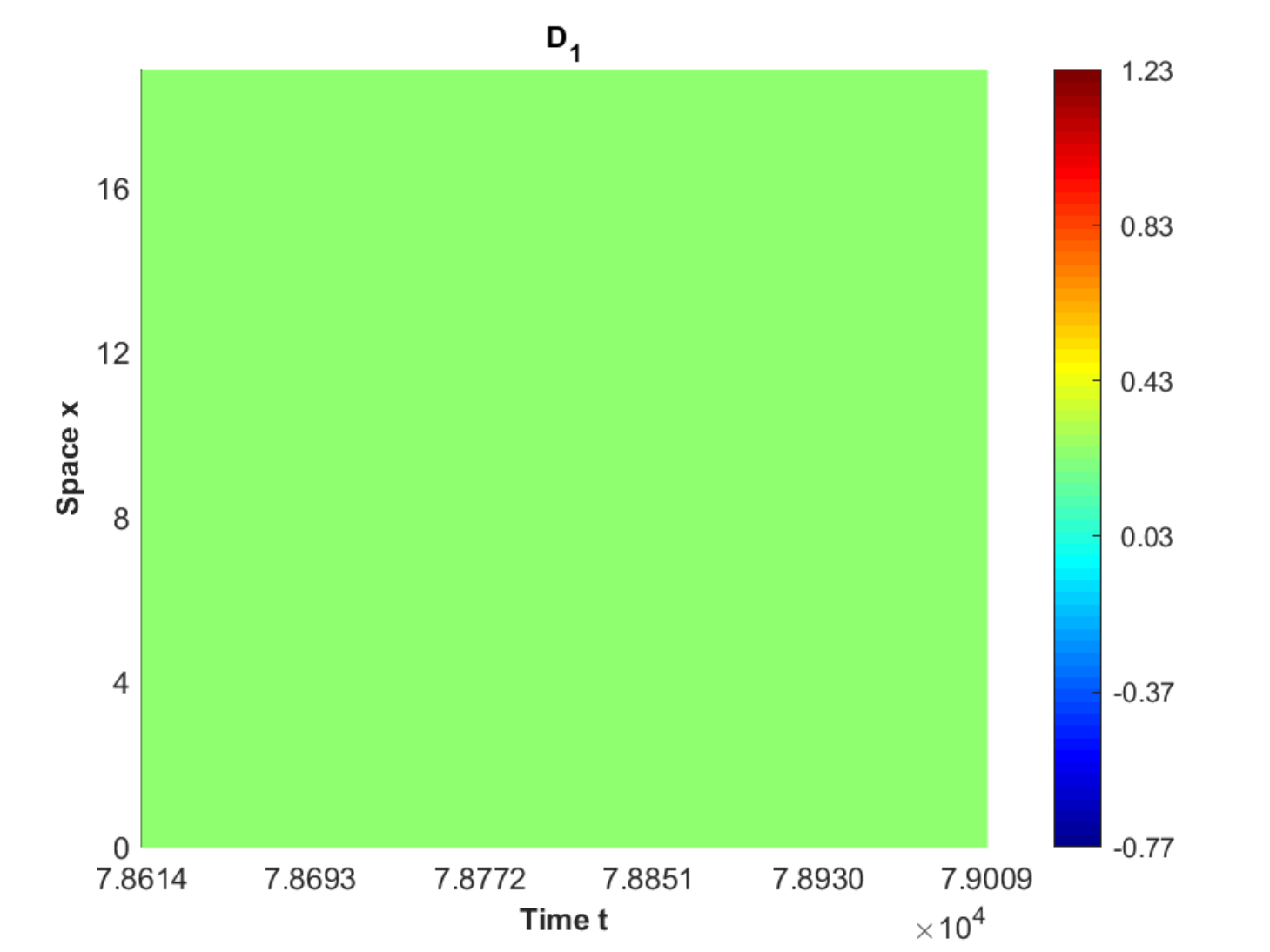}}\\
\subfigure{\includegraphics[width=2in]{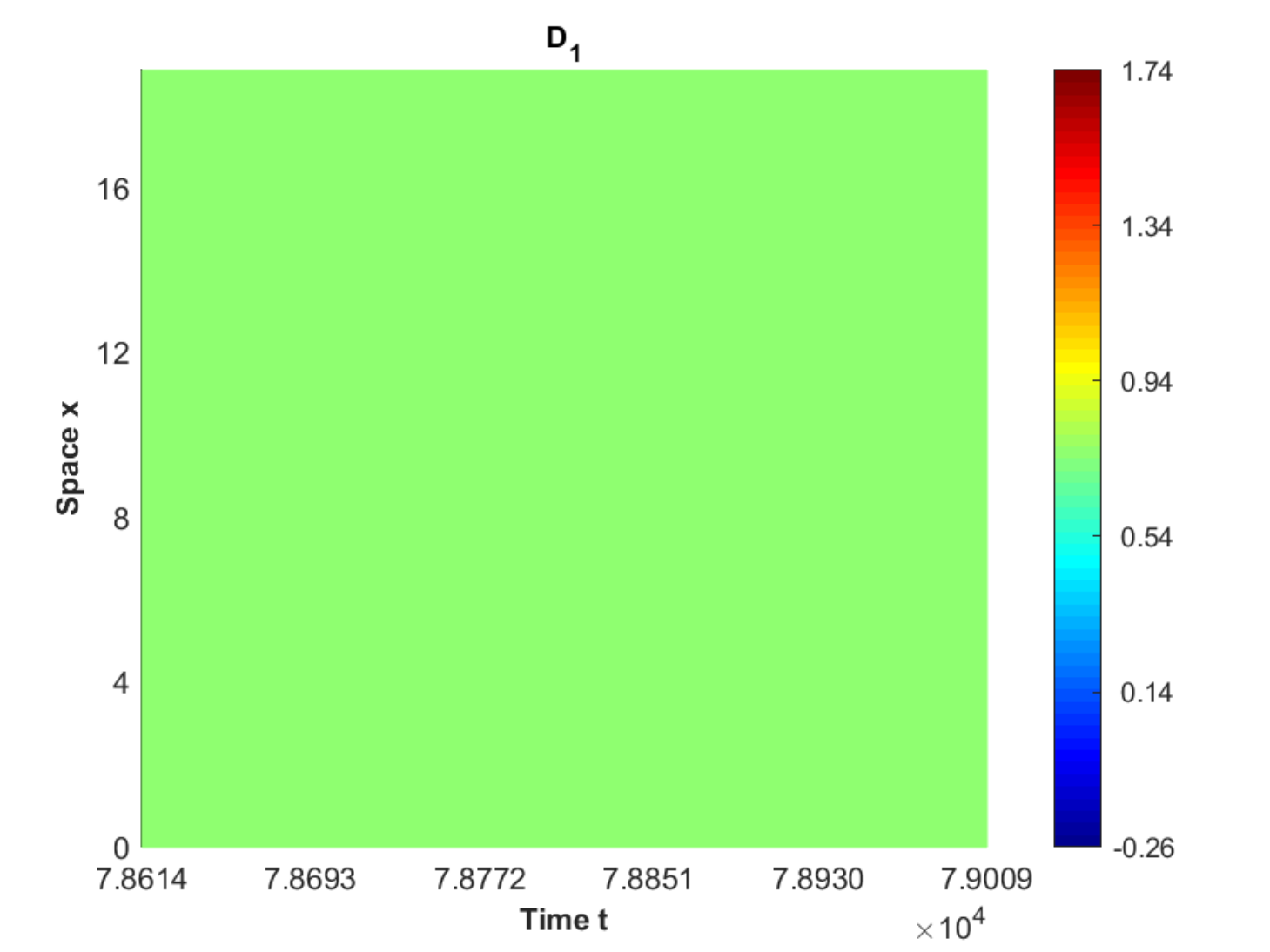}}
\end{center}

\begin{center}
  \subfigure{\includegraphics[width=2in]{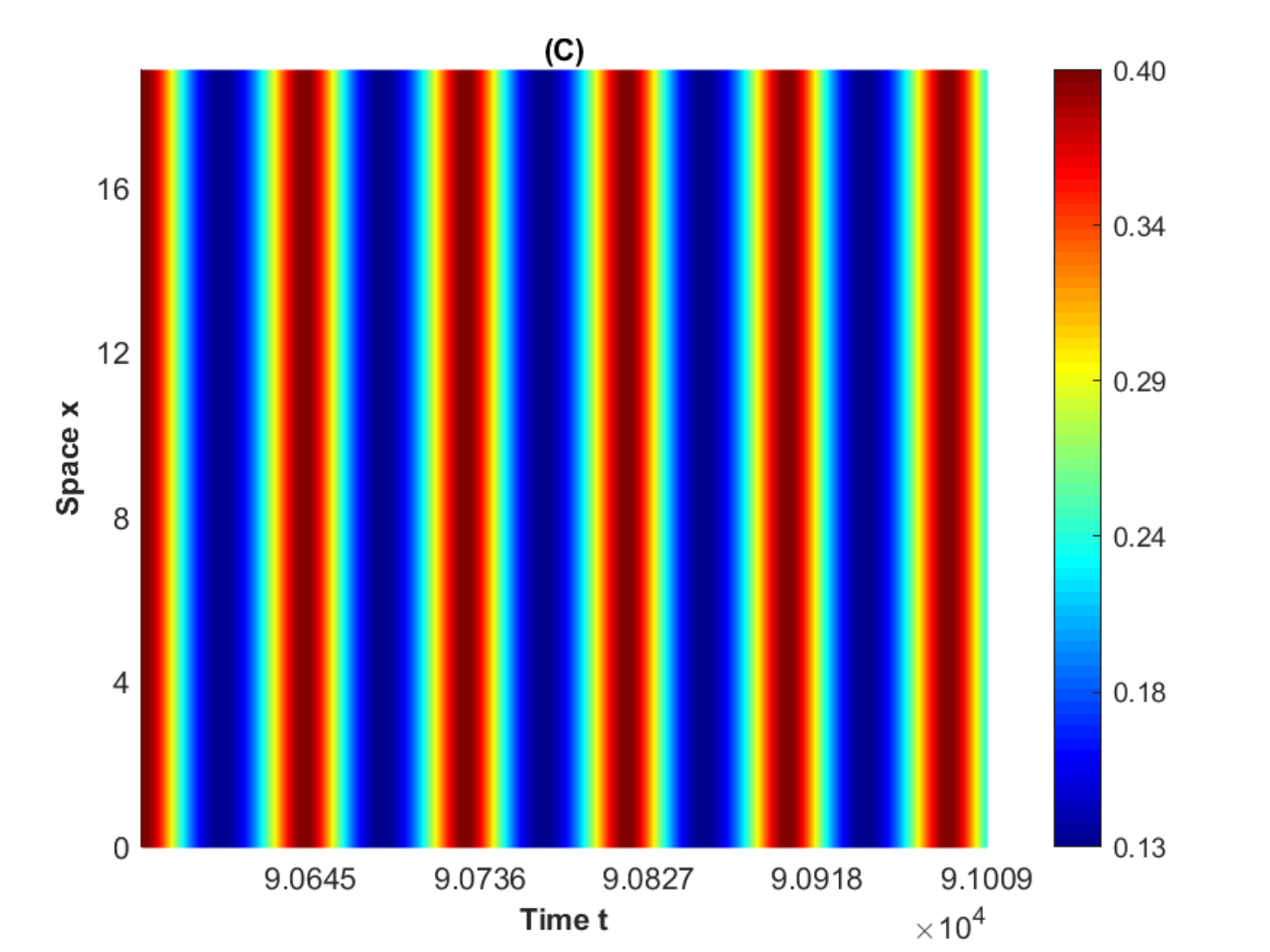}}\\
  \subfigure{\includegraphics[width=2in]{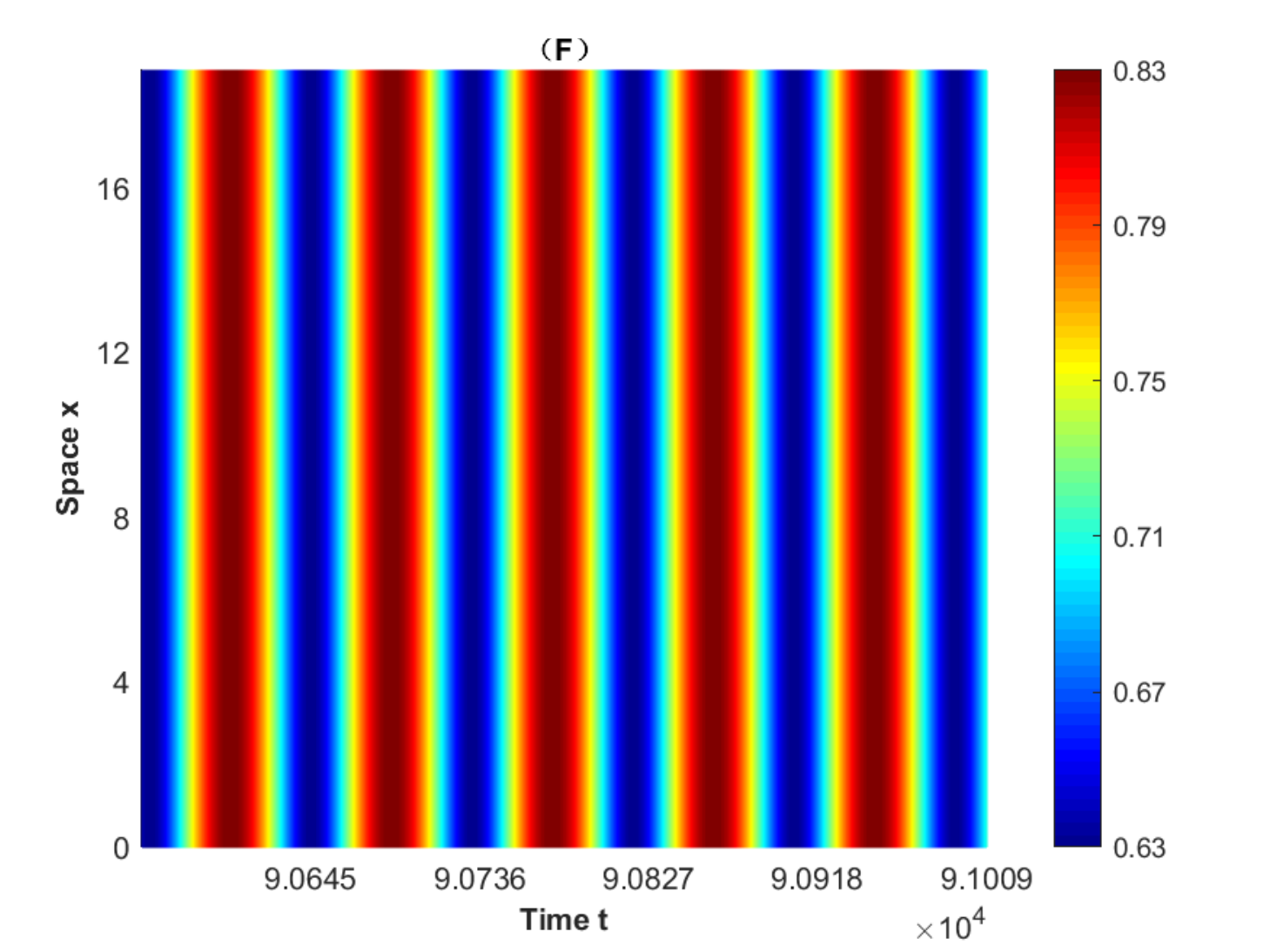}}
\end{center}
\begin{center}
\subfigure{\includegraphics[width=2in]{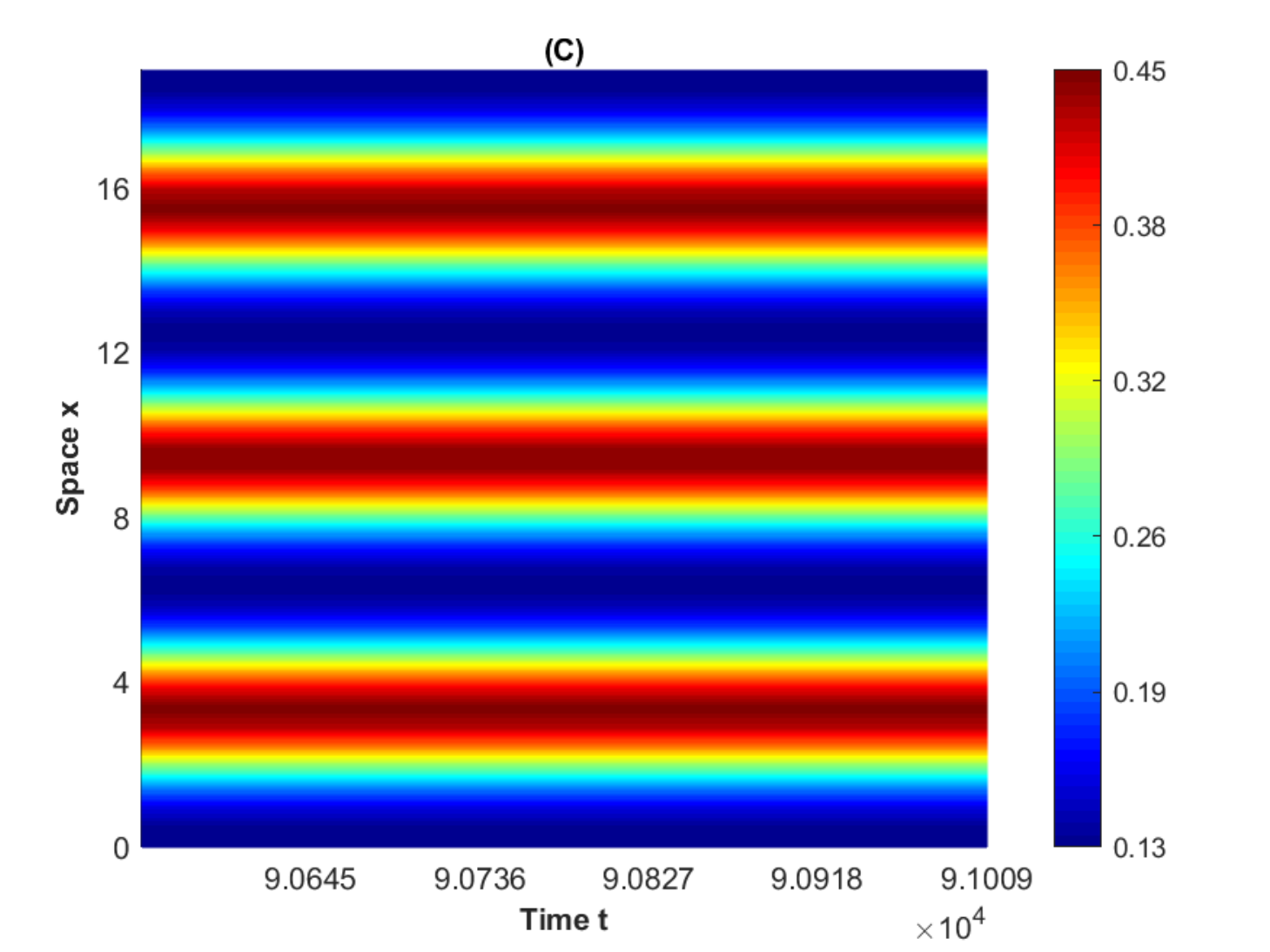}}\\
\subfigure{\includegraphics[width=2in]{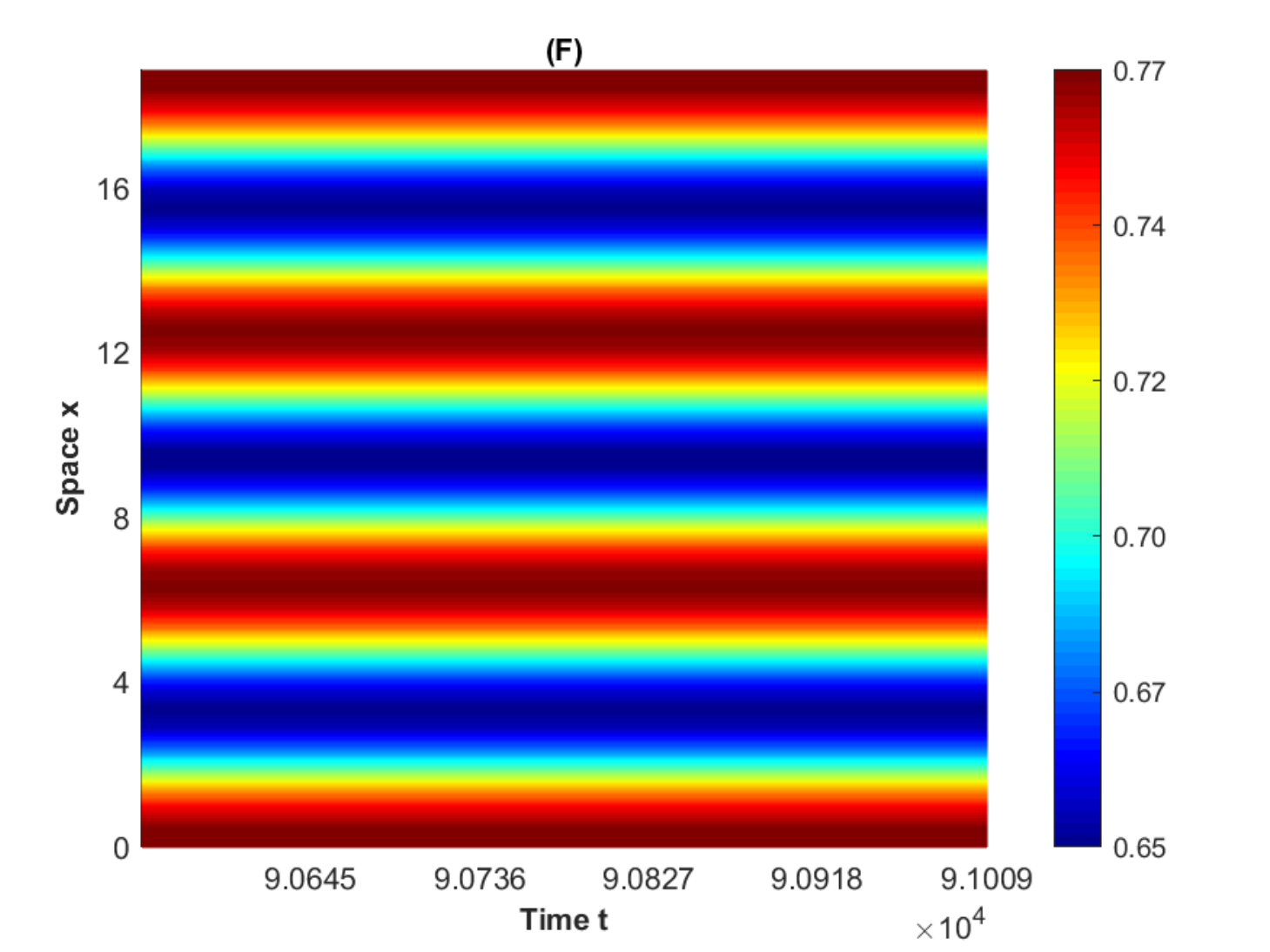}}
\end{center}

\end{multicols}
 \caption{The stable patterns in regions $D_1$, $D_3$, $D_5$. The above graphs shows the dynamics of mussel while the belows, the algae.  }\label{fig-stable-pattern}

\end{figure}

\begin{figure}[htp]
\begin{multicols}{3}
\begin{center}
\subfigure{\includegraphics[width=2in]{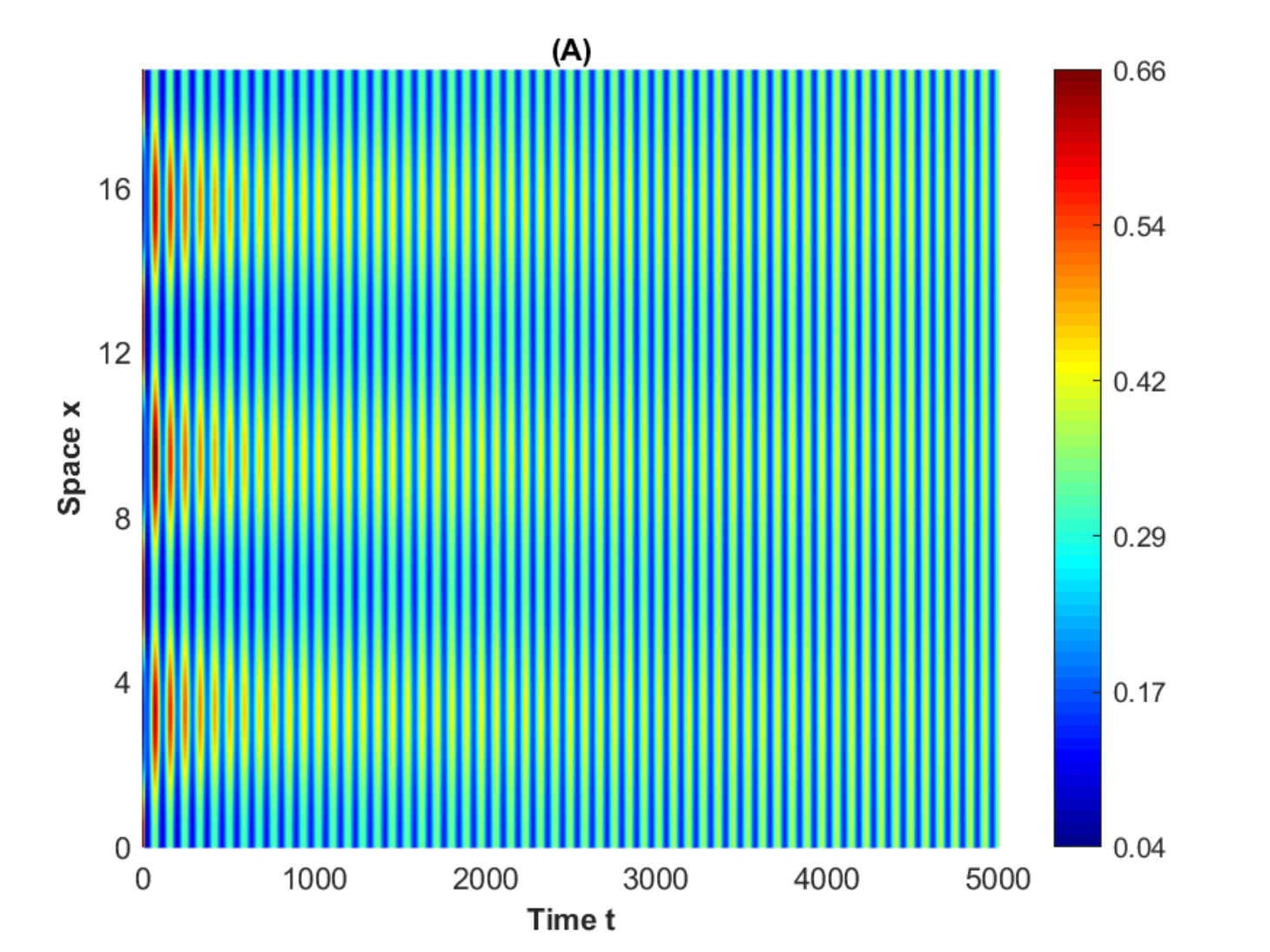}}\\
\subfigure{\includegraphics[width=2in]{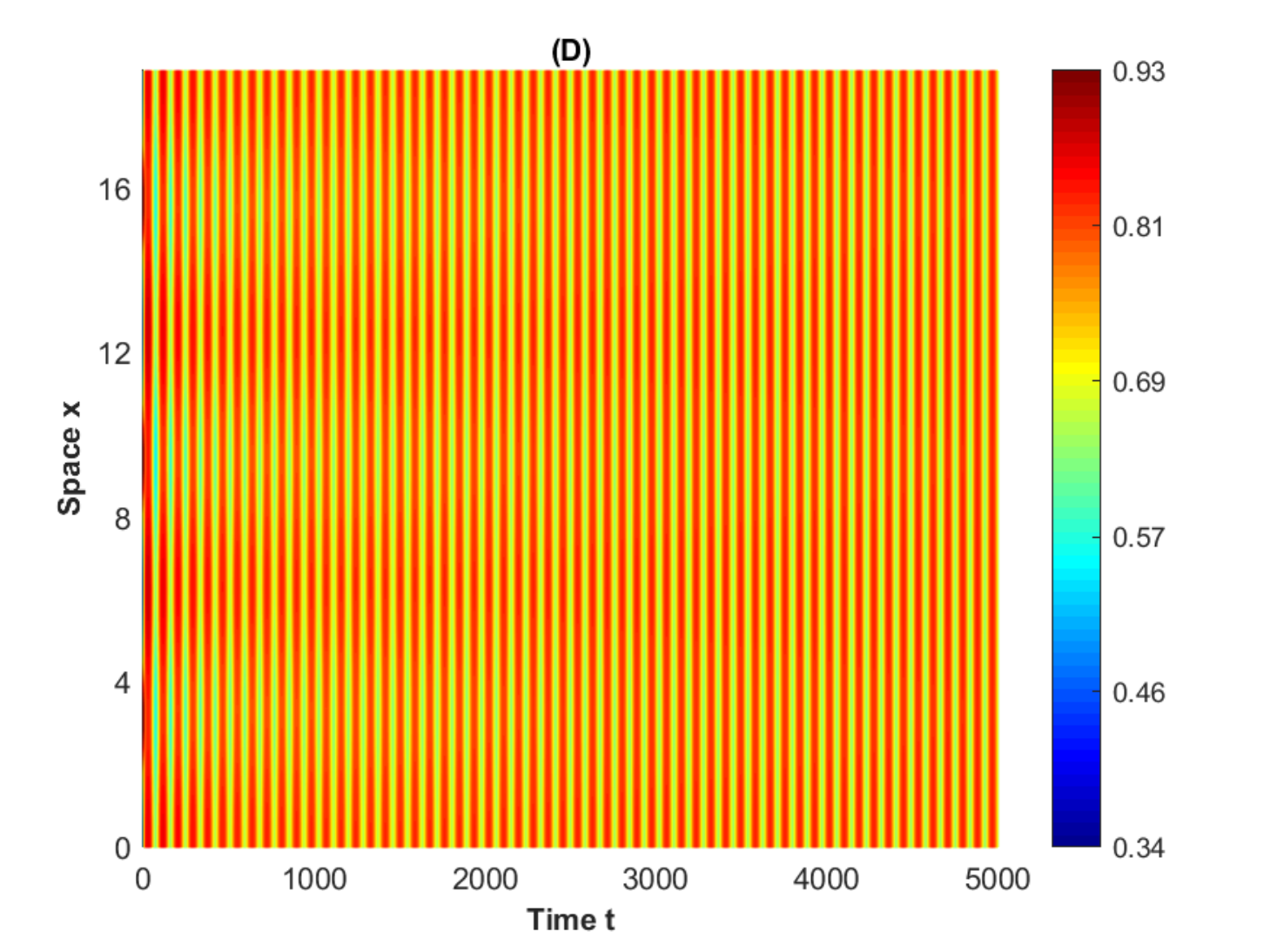}}
\end{center}

\begin{center}
\subfigure{\includegraphics[width=2in]{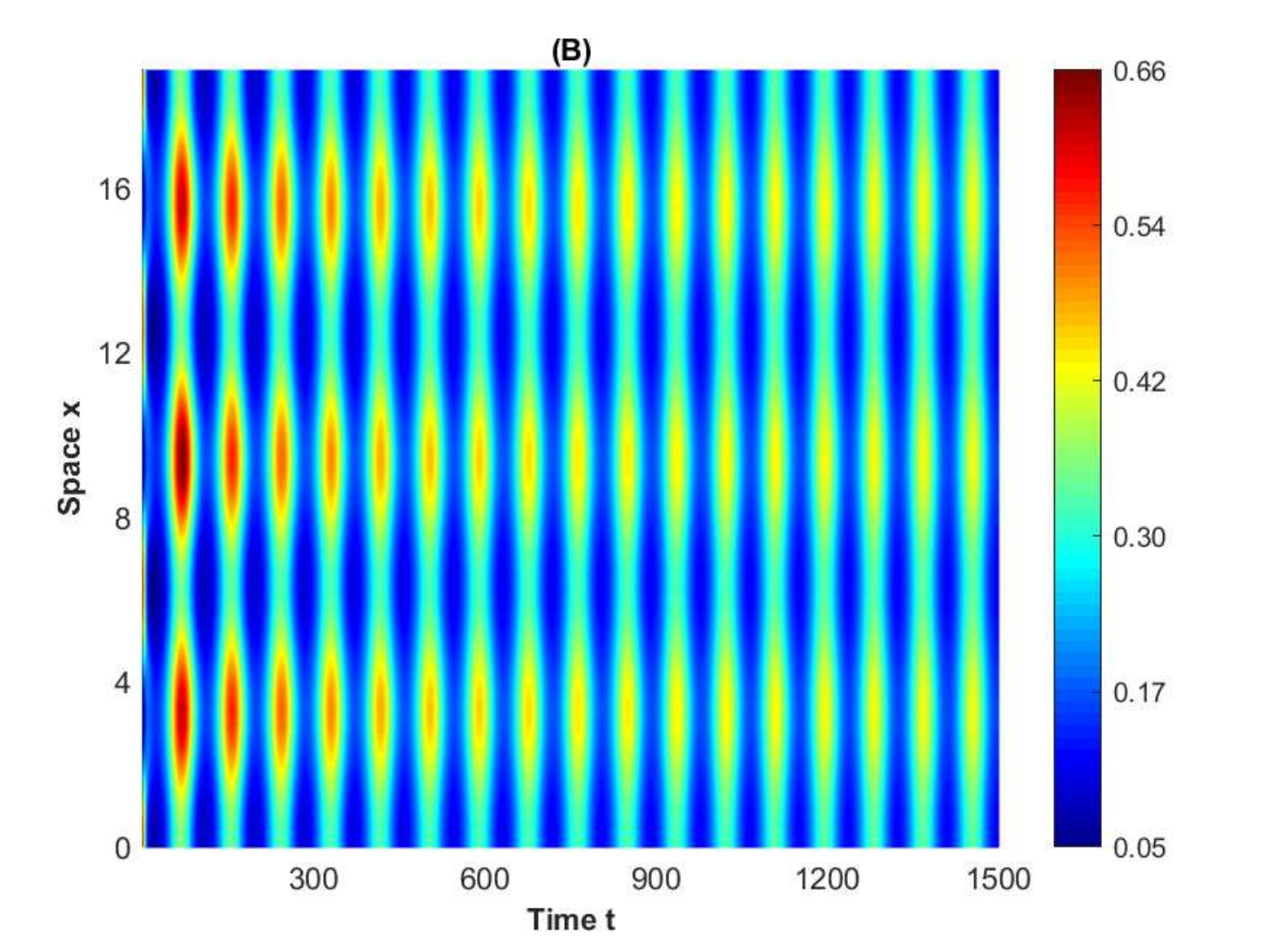}}\\
\subfigure{\includegraphics[width=2in]{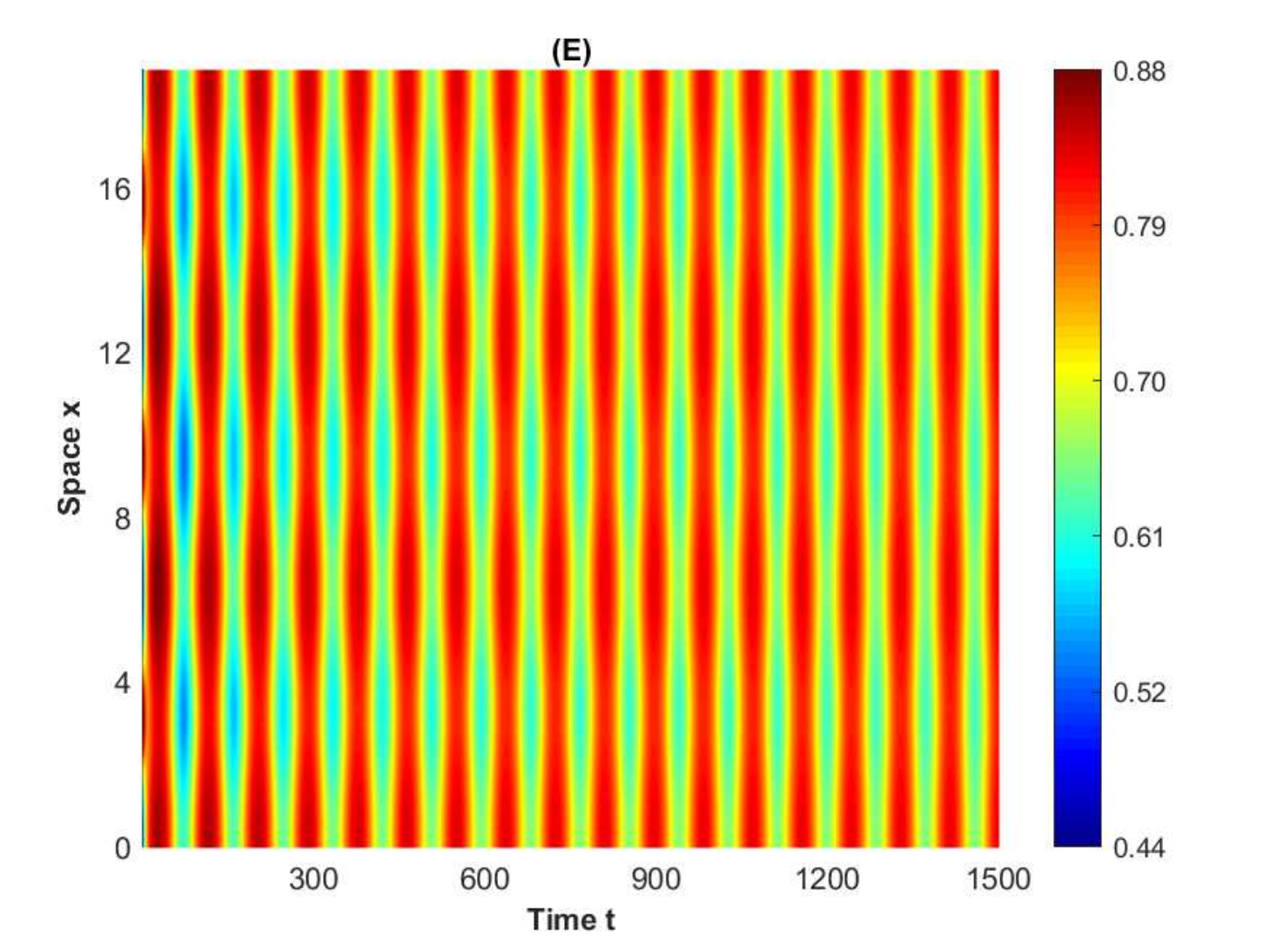}}
\end{center}

\begin{center}
\subfigure{\includegraphics[width=2in]{m_D3_final.pdf}}\\
\subfigure{\includegraphics[width=2in]{a_D3_final.pdf}}
\end{center}
\end{multicols}
 \caption{Stable spatially homogeneous periodic solution in $D_3$ with
 $(\tau_{\varepsilon}, d_{\varepsilon})=(0.5, -0.0005)\in D_3$, and the initial function are $(m^*+0.1+0.3\cos x, a^*-0.1-0.3\cos x)$. (A)-(C): The dynamics of mussel; (D)-(F): The dynamics of algae.}\label{fig-D3}
\end{figure}

Fig.\ref{fig-D3} shows a stable spatially homogeneous periodic solution in $D_3$. Fig.\ref{fig-D5} shows a stable spatially homogeneous steady state in $D_5$. (A) and (D) represent the trends of pattern formation; (B) and (E) show the transformation process at the beginning; (C) and (F)
show the final stable behavior.

Fig.\ref{fig-D4_1} and Fig.\ref{fig-D4_2} shows the different evolutionary process of system \eqref{eq_ma_tau} with the same parameters but slightly different initial functions. Fix $(\tau_{\varepsilon}, d_{\varepsilon})=(0.5, -0.0009)\in D_4$, one case is that we choose the initial function $(m^*+0.3+0.5\cos x, a^*-0.5\cos x)$, then after a period of time evolution, one can see a spatially inhomogeneous periodic solution appears
(see graph (B) and (E) of Fig.\ref{fig-D4_1}), but this is not the final state, the spatially inhomogeneous periodic solution disappears as time going on, and finally reach its stable state, a spatially homogeneous periodic solution. The other case is just the opposite. We choose$(m^*+0.3+0.5\cos x, a^*-0.1-0.5\cos x)$ as the initial functions, and the simulation shows the solution can also evolve into a spatially inhomogeneous periodic solution, but it ultimately becomes a spatially
inhomogeneous steady state when time is long enough.

\begin{figure}[htp]
\begin{multicols}{3}
\begin{center}
\subfigure{\includegraphics[width=2in]{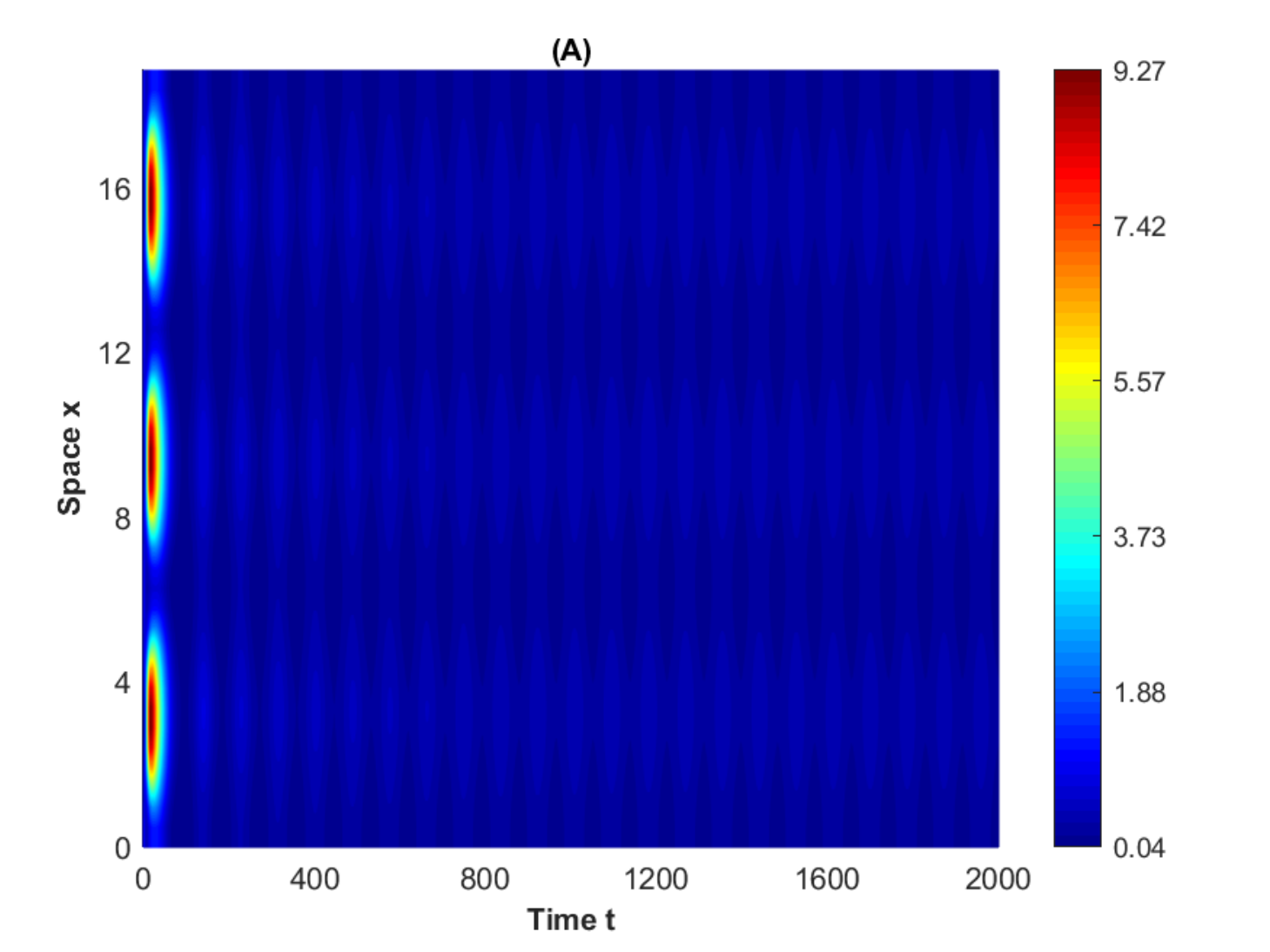}}\\
\subfigure{\includegraphics[width=2in]{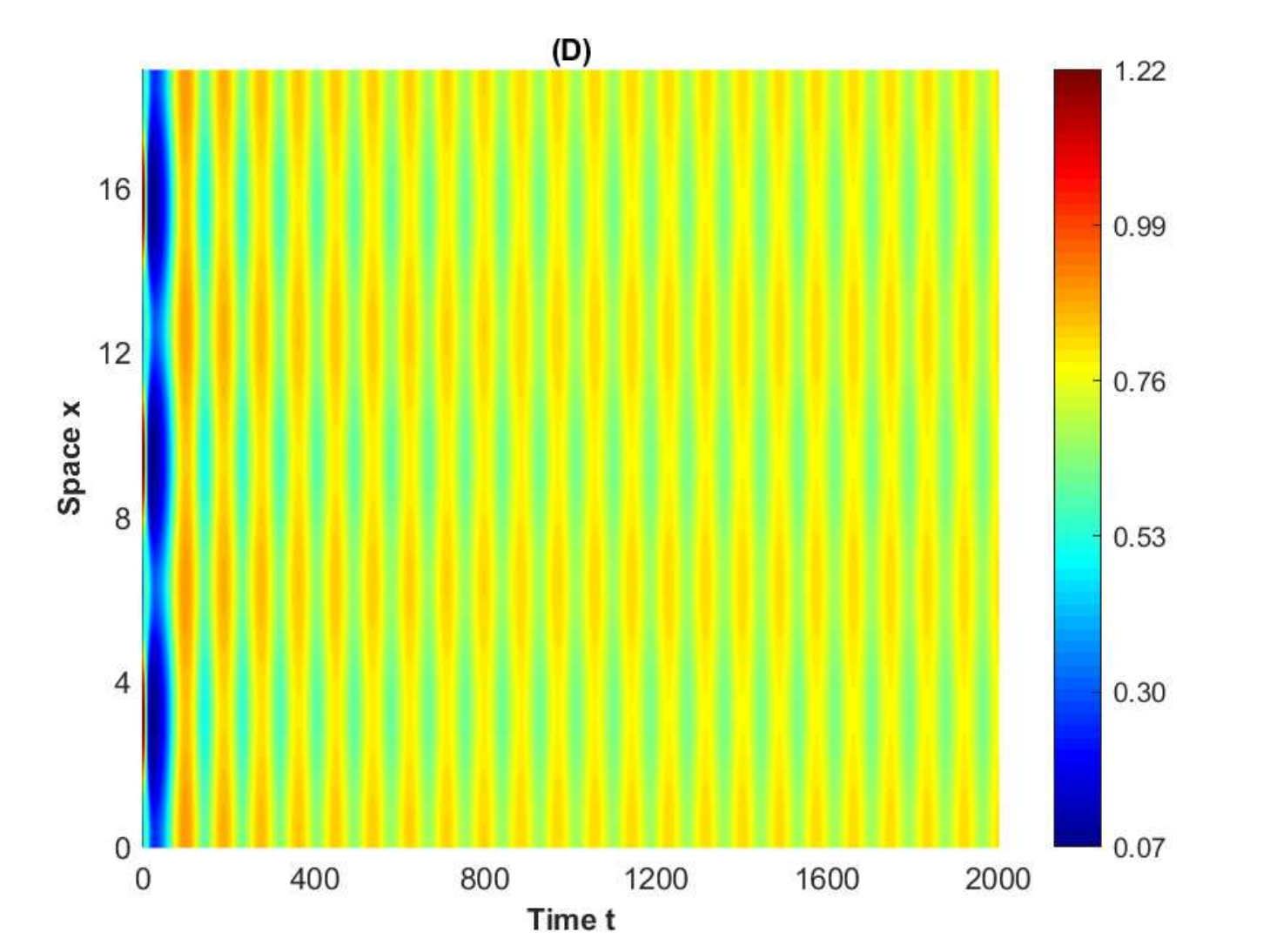}}
\end{center}

\begin{center}
\subfigure{\includegraphics[width=2in]{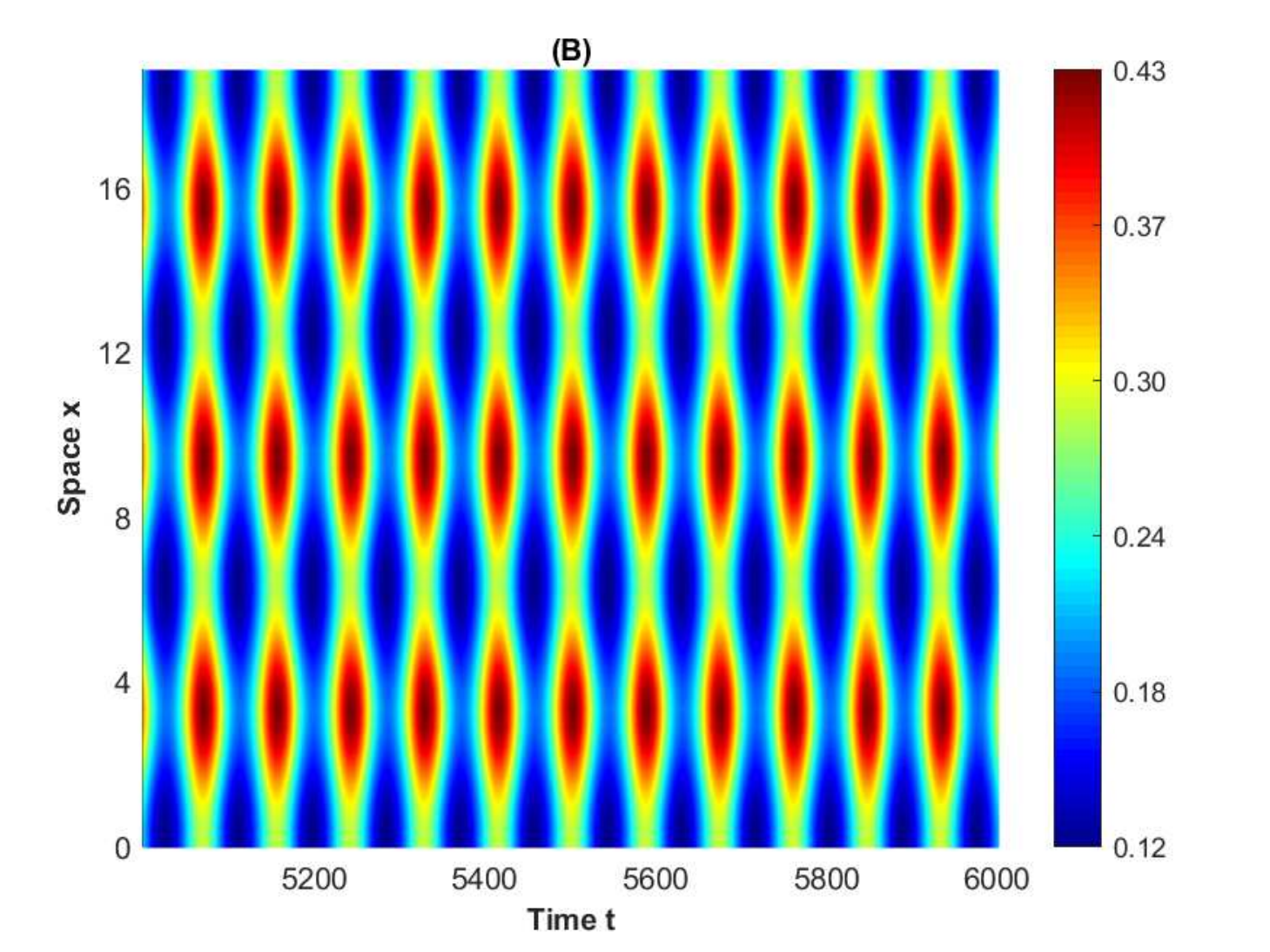}} \\ 
\subfigure{\includegraphics[width=2in]{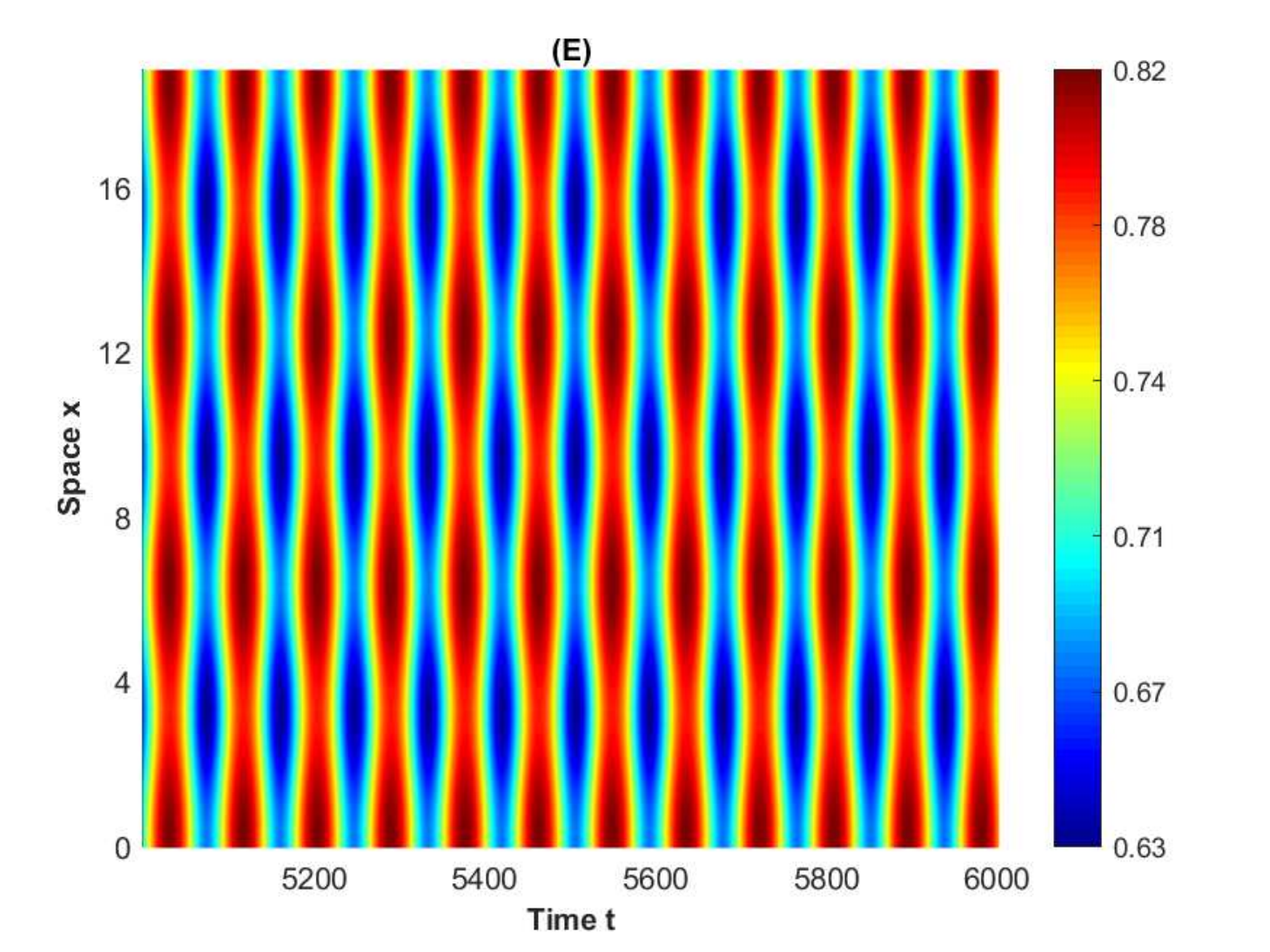}} \end{center}

\begin{center}
\subfigure{\includegraphics[width=2in]{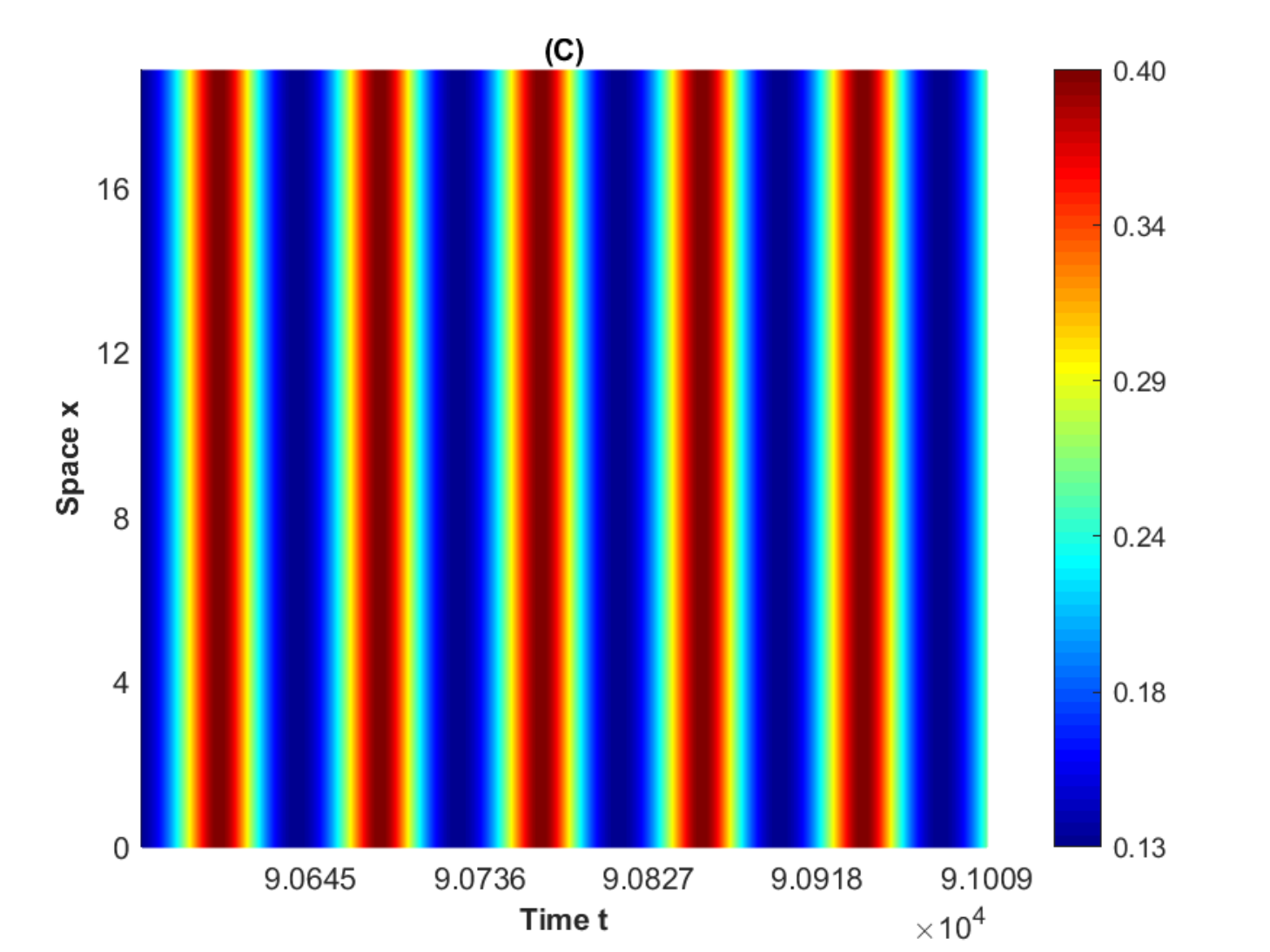}}\\
\subfigure{\includegraphics[width=2in]{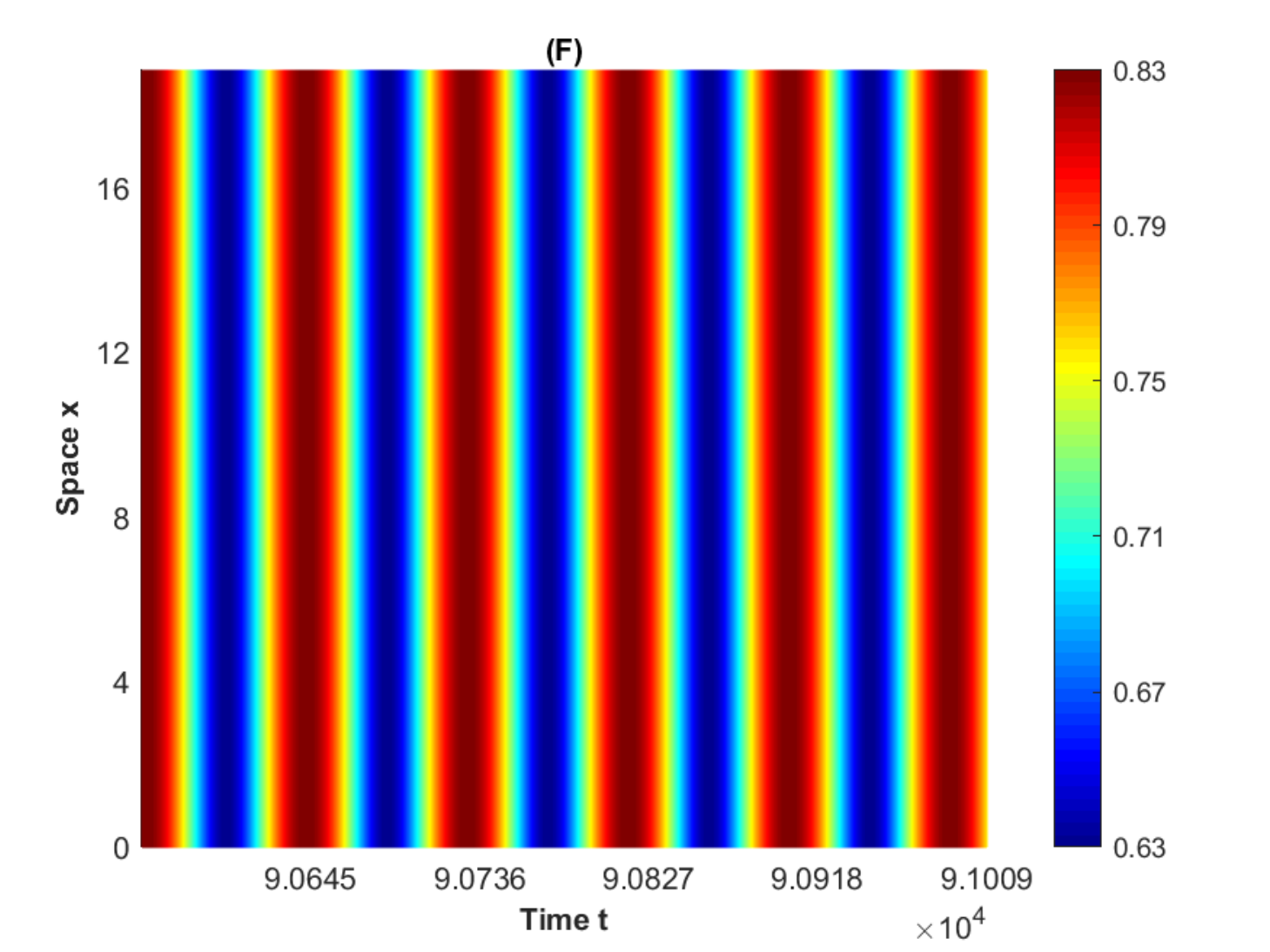}}
\end{center}

\end{multicols}
 \caption{Stable spatially homogeneous periodic solution and unstable spatially inhomogeneous periodic solutions in $D_4$ with
 $(\tau_{\varepsilon}, d_{\varepsilon})=(0.5, -0.0009)\in D_4$, and the initial function are $(m^*+0.3+0.5\cos x, a^*-0.5\cos x)$. (A)-(C): The dynamics of mussel; (D)-(F): The dynamics of algae.}\label{fig-D4_1}
\end{figure}

\begin{figure}[htp]
\begin{multicols}{3}
\begin{center}
\subfigure{\includegraphics[width=2in]{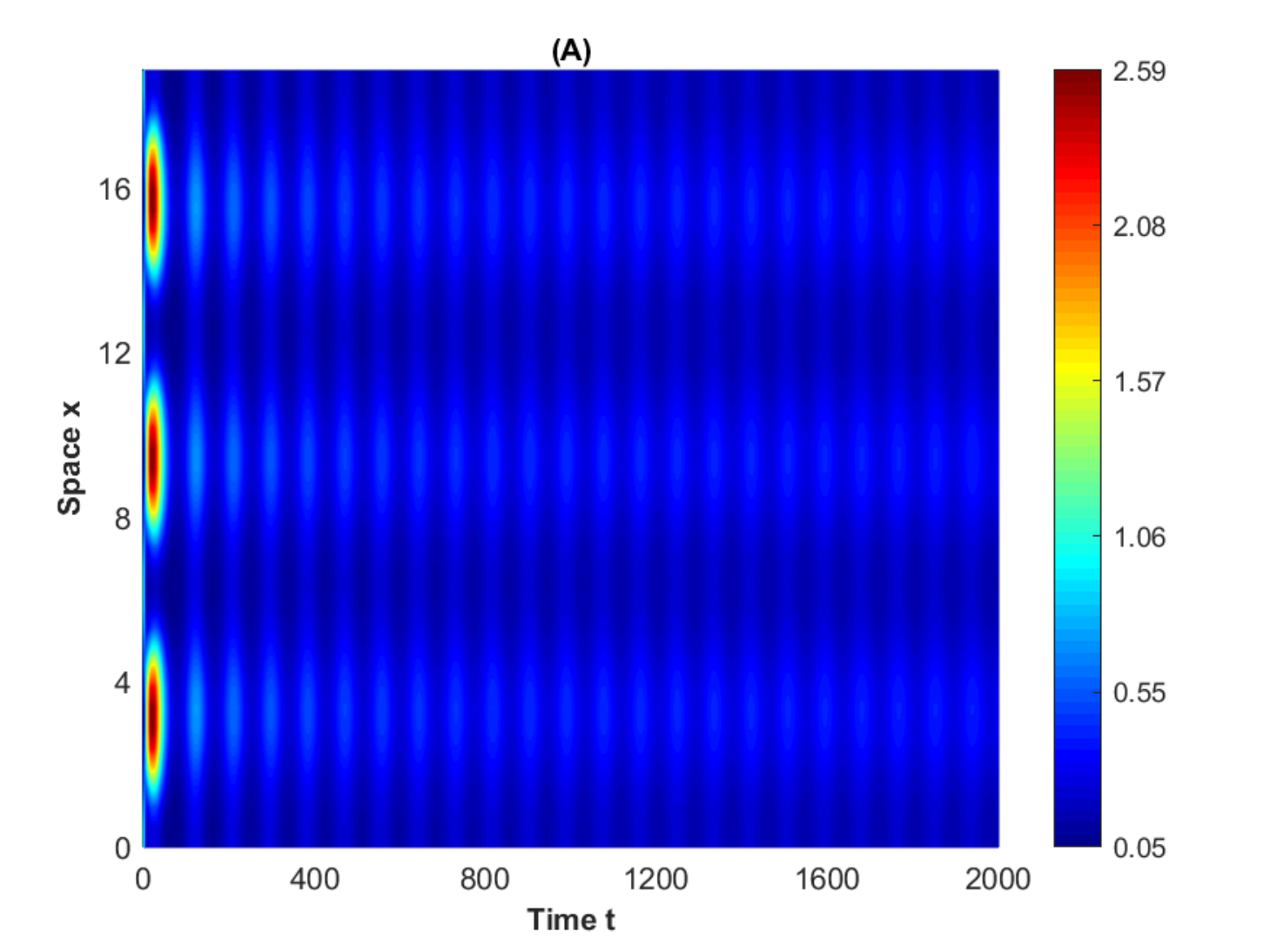}}\\
\subfigure{\includegraphics[width=2in]{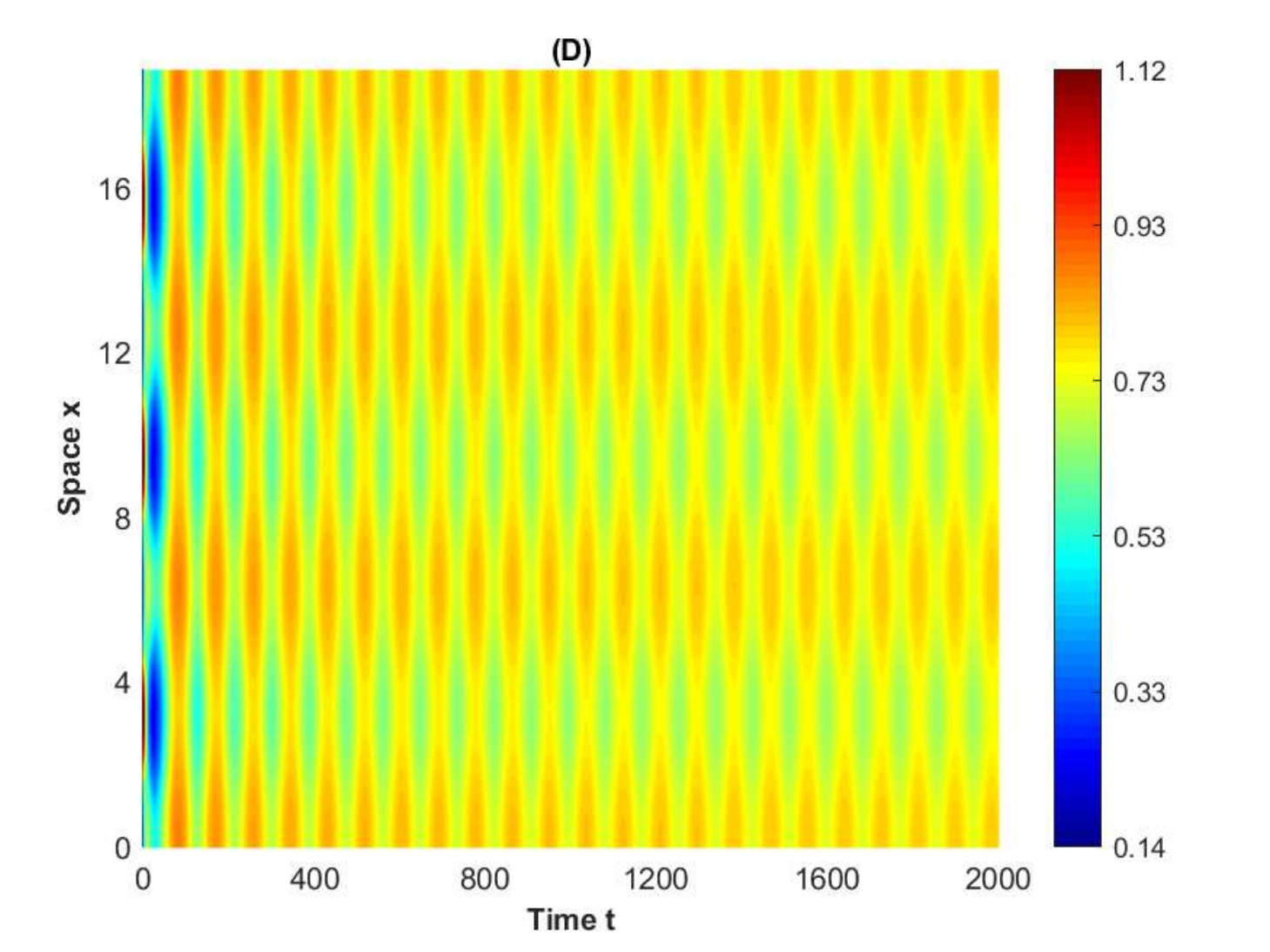}}
\end{center}

\begin{center}
\subfigure{\includegraphics[width=2in]{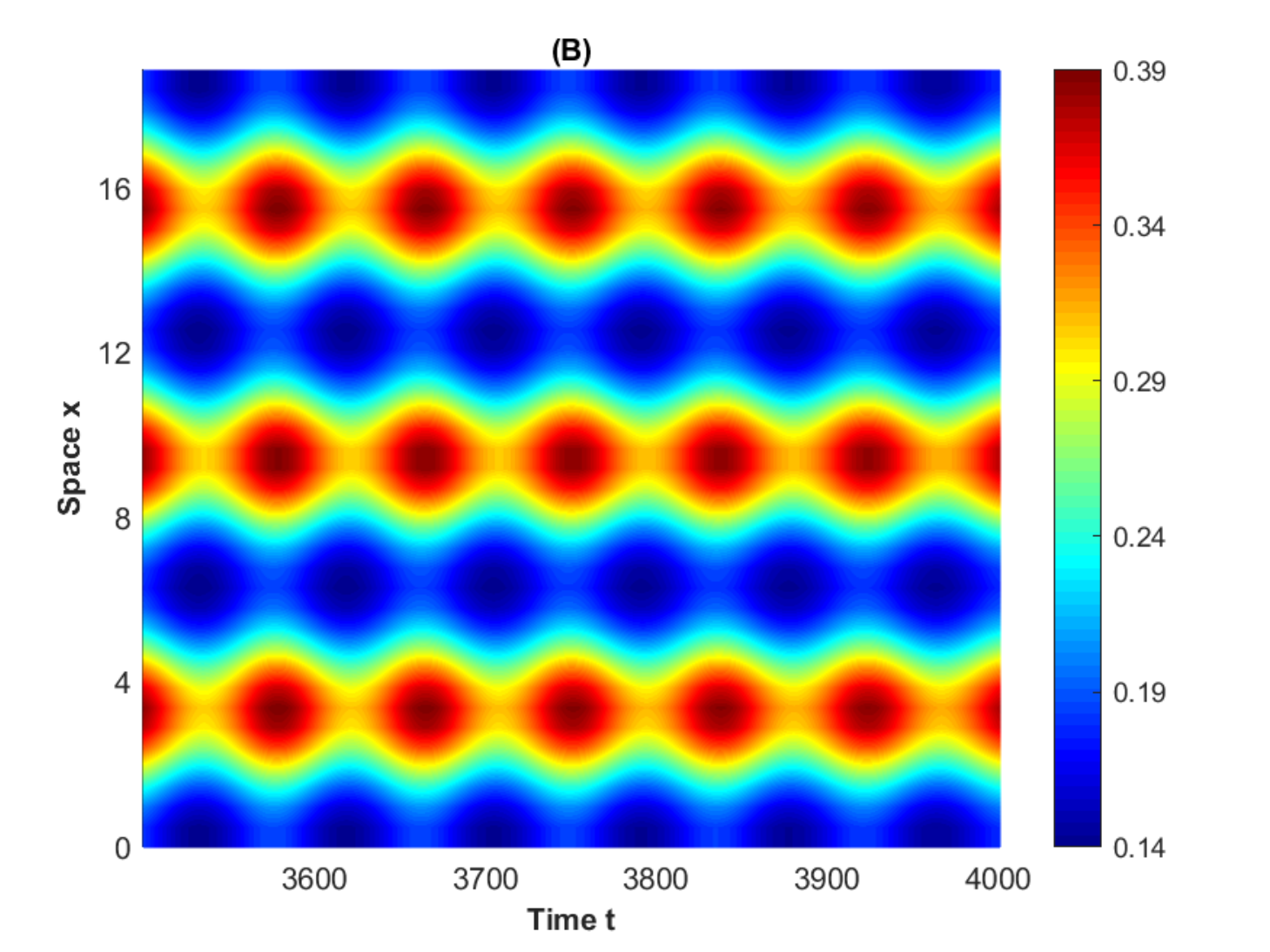}}\\
\subfigure{\includegraphics[width=2in]{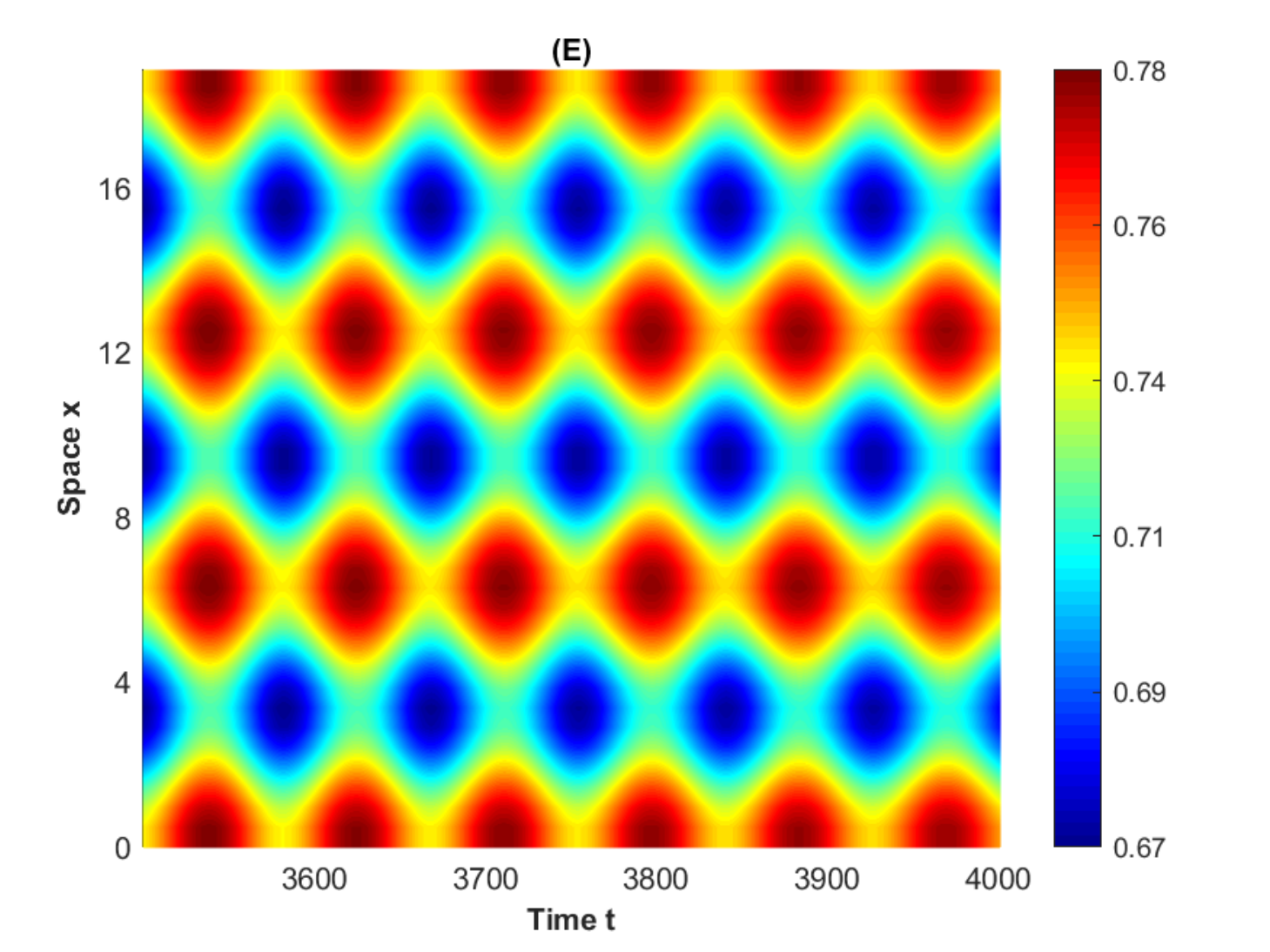}}
\end{center}

\begin{center}
\subfigure{\includegraphics[width=2in]{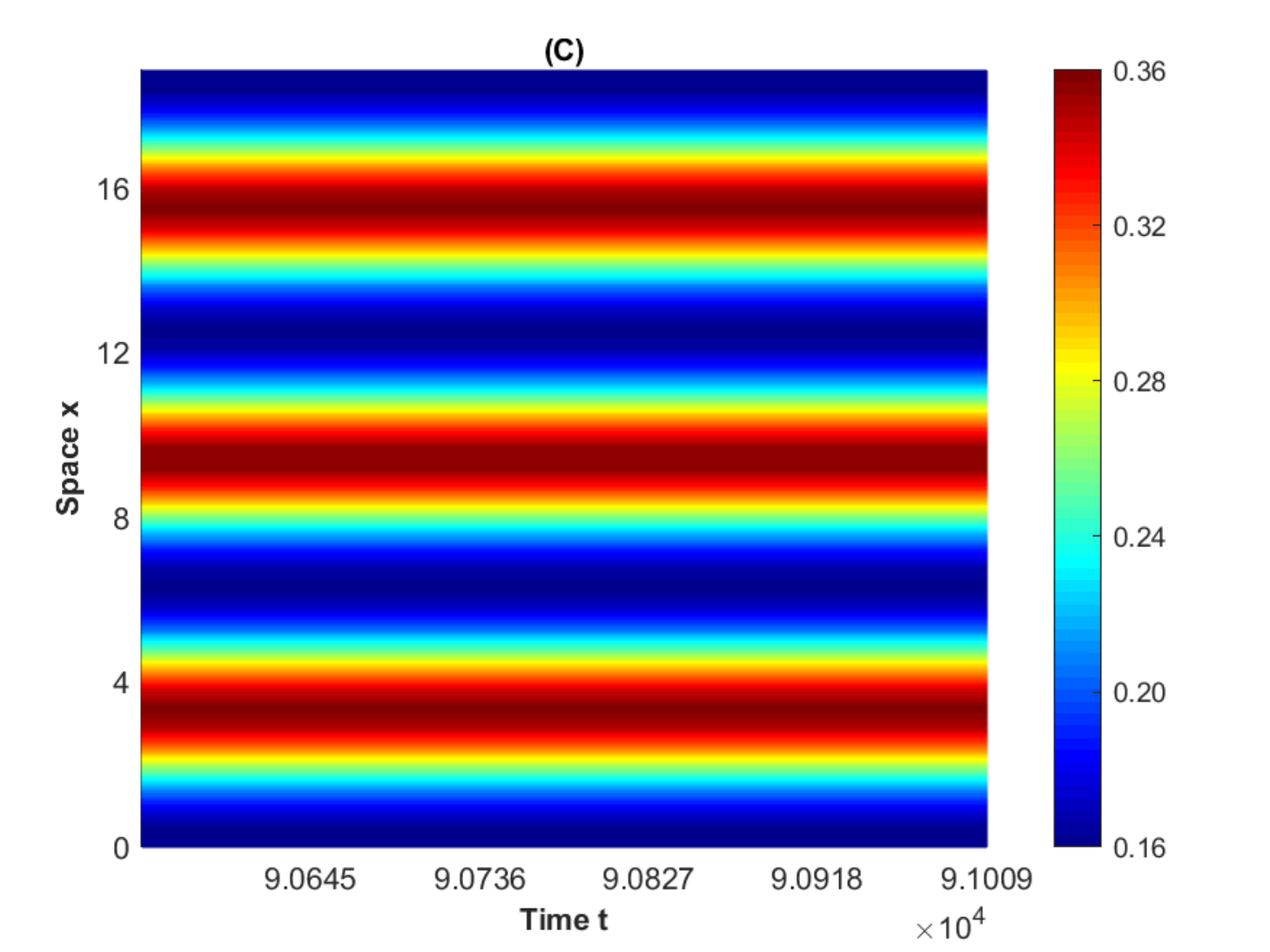}}\\
\subfigure{\includegraphics[width=2in]{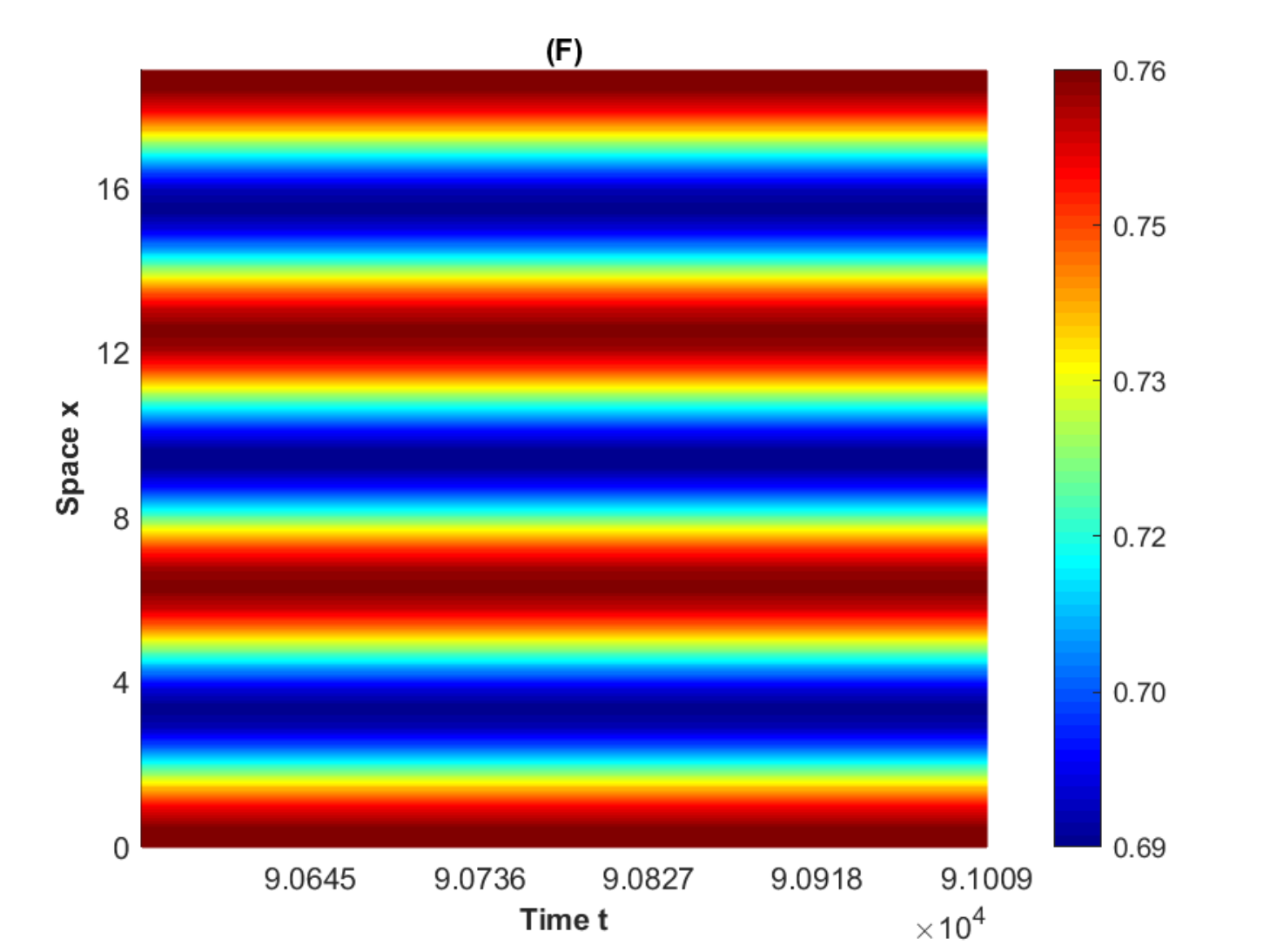}}
\end{center}
\end{multicols}
 \caption{Stable non-constant steady states in $D_4$ and unstable spatially inhomogeneous periodic solutions with
 $(\tau_{\varepsilon}, d_{\varepsilon})=(0.5, -0.0009)\in D_4$, and the initial function are $(m^*+0.3+0.5\cos x, a^*-0.1-0.5\cos x)$. (A)-(C): The dynamics of mussel; (D)-(F): The dynamics of algae.}\label{fig-D4_2}
\end{figure}

\begin{figure}[htp]
\begin{multicols}{3}
\begin{center}
\subfigure{\includegraphics[width=2in]{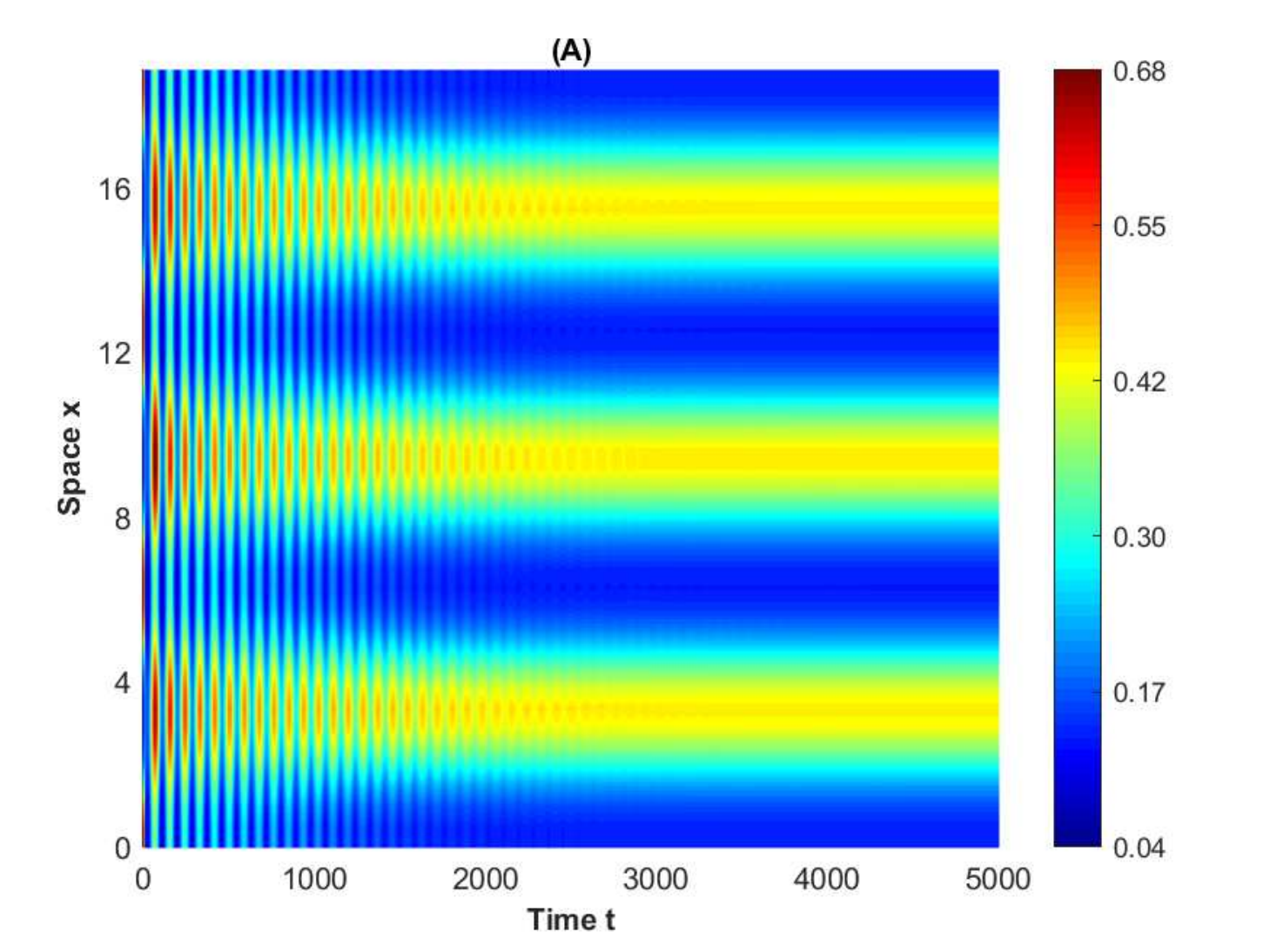}}\\
\subfigure{\includegraphics[width=2in]{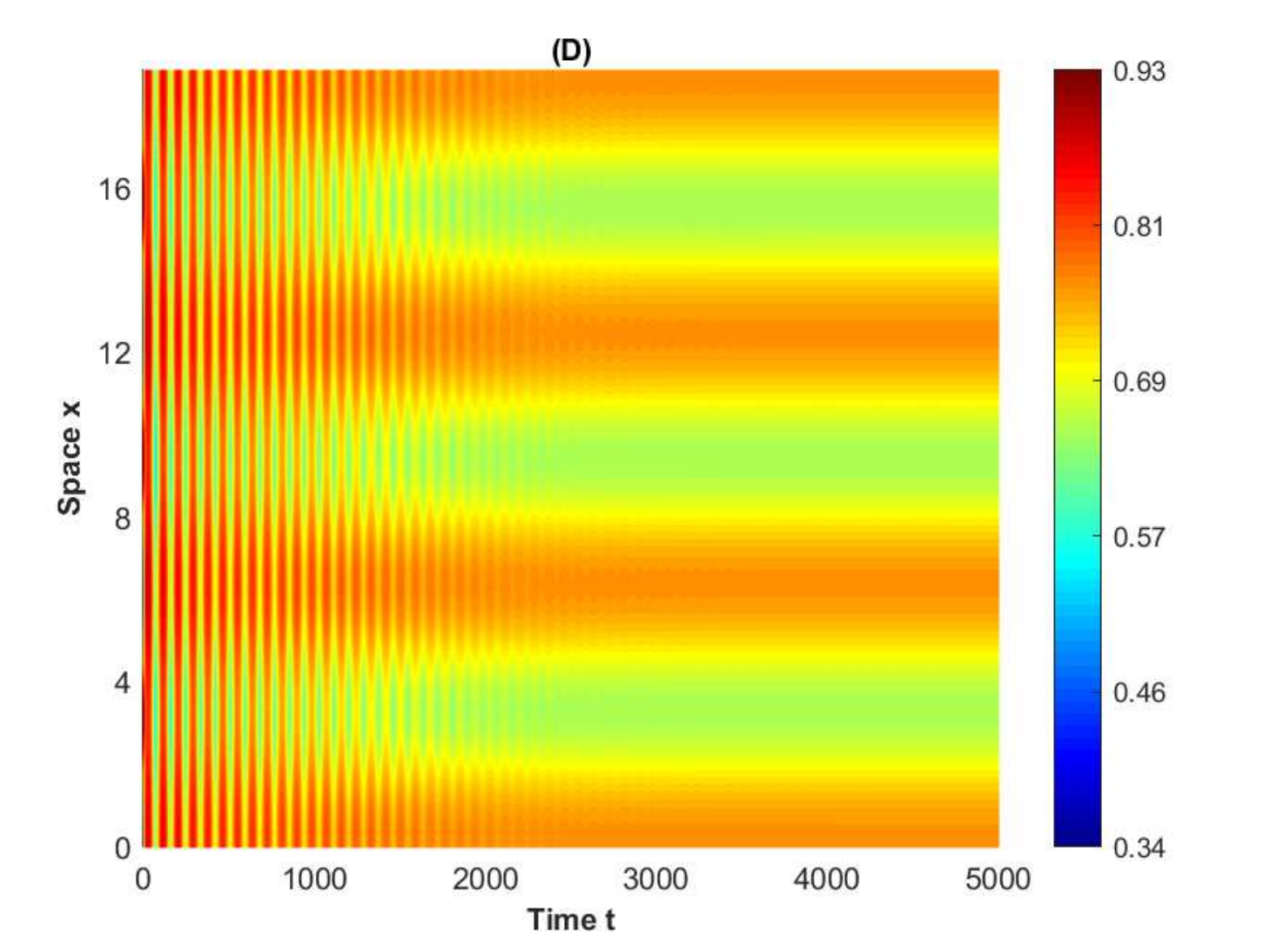}}
\end{center}

\begin{center}
\subfigure{\includegraphics[width=2in]{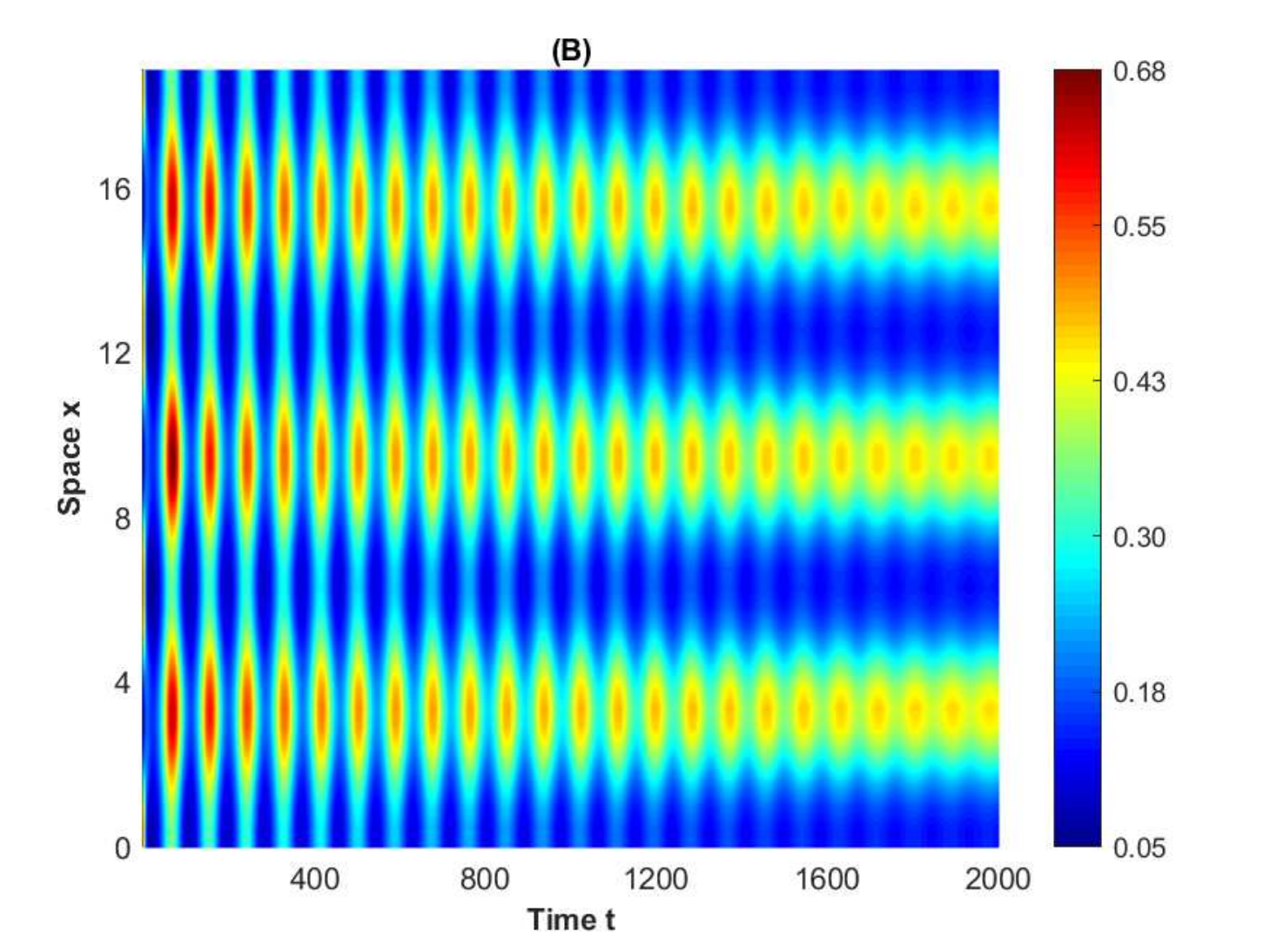}}\\
\subfigure{\includegraphics[width=2in]{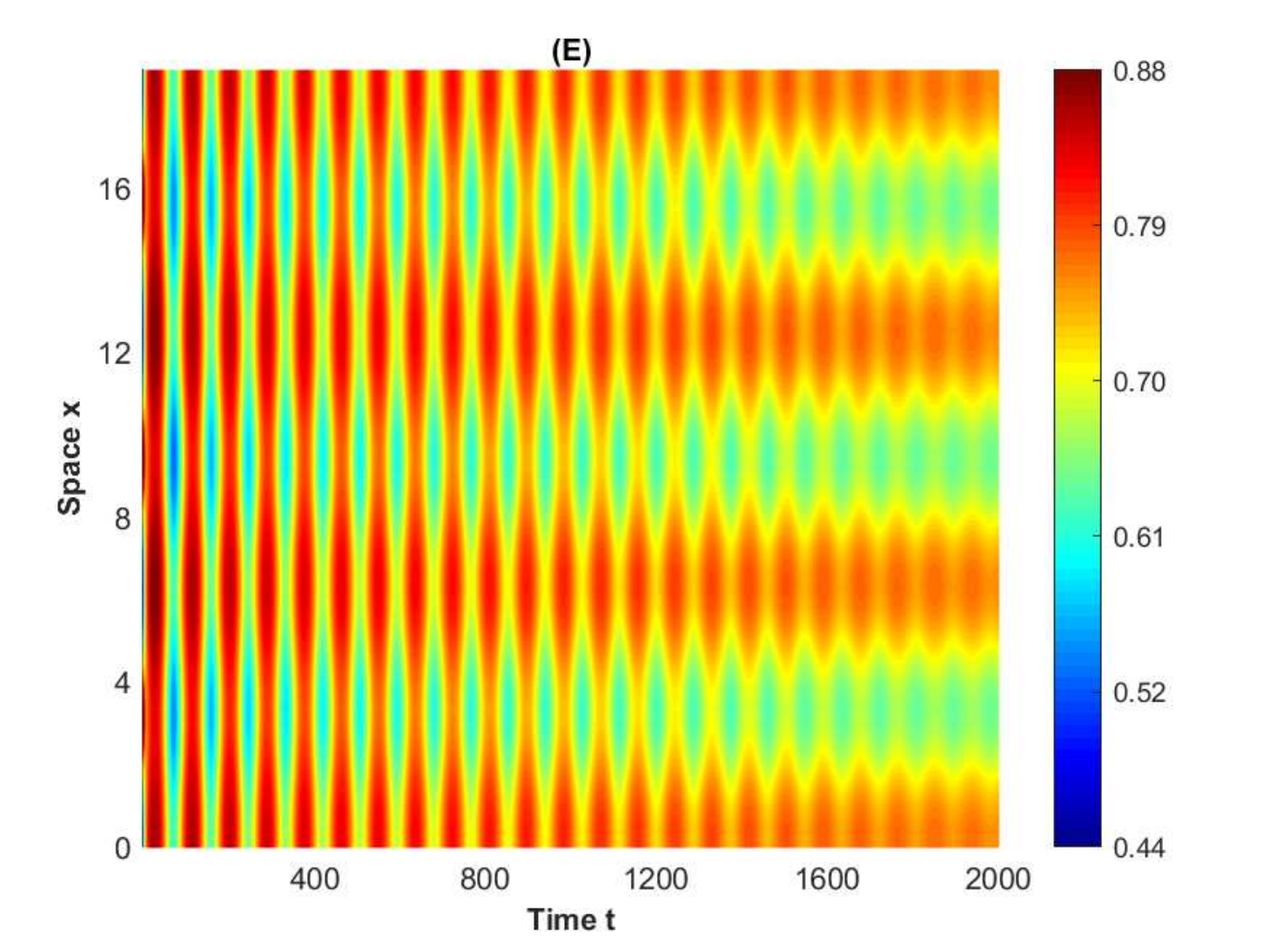}}
\end{center}

\begin{center}
\subfigure{\includegraphics[width=2in]{m_D5_final.pdf}}\\
\subfigure{\includegraphics[width=2in]{a_D5_final.pdf}}
\end{center}
\end{multicols}
 \caption{Stable spatially homogeneous periodic solution in $D_5$ with
 $(\tau_{\varepsilon}, d_{\varepsilon})=(0.5, -0.002)\in D_5$, and the initial function are $(m^*+0.1+0.3\cos x, a^*-0.1-0.3\cos x)$. (A)-(C): The dynamics of mussel; (D)-(F): The dynamics of algae.}\label{fig-D5}
\end{figure}

\begin{figure}[htp]
\begin{center}
\subfigure{\includegraphics[width=2.5in]{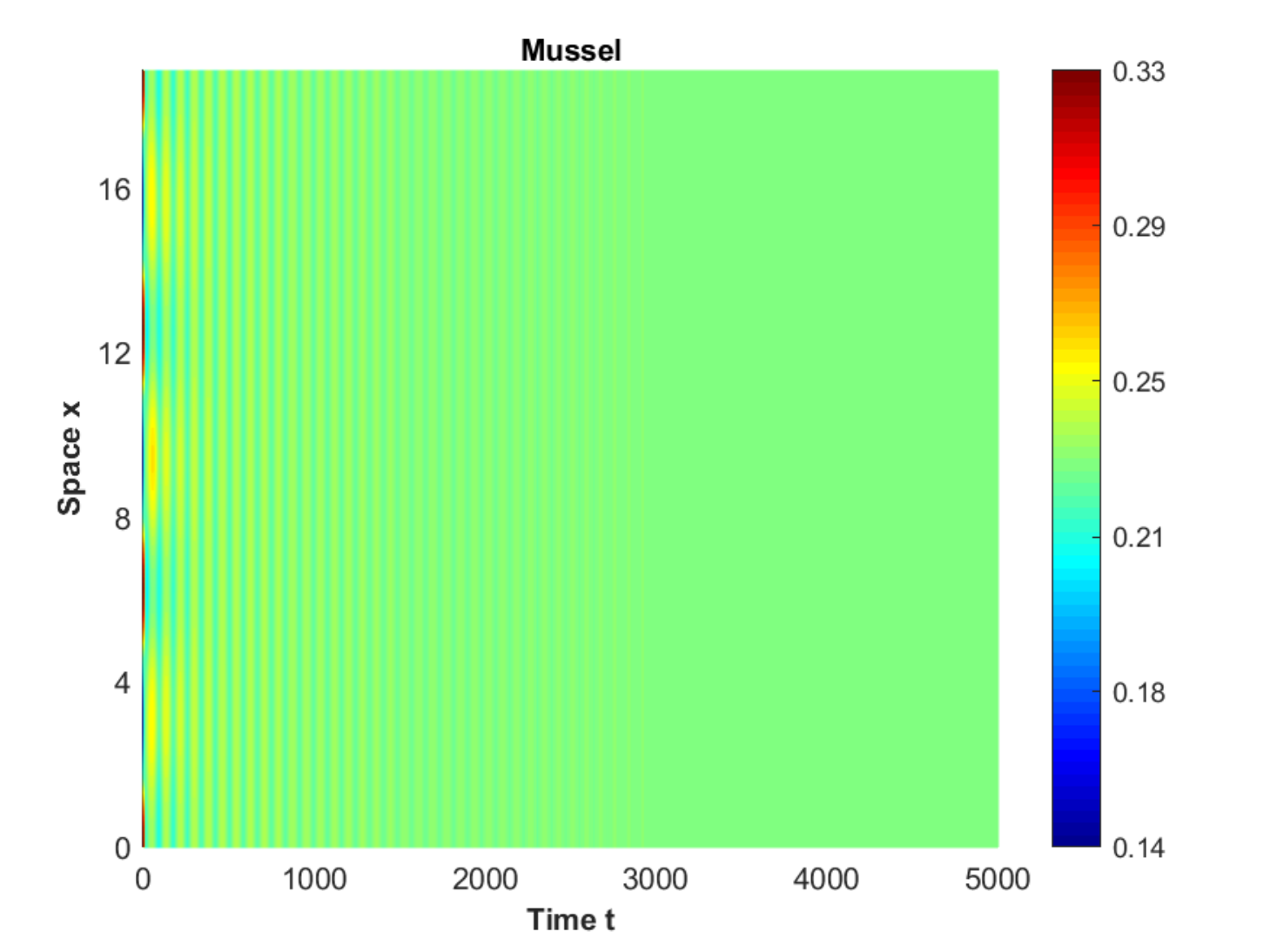}}
\subfigure{\includegraphics[width=2.5in]{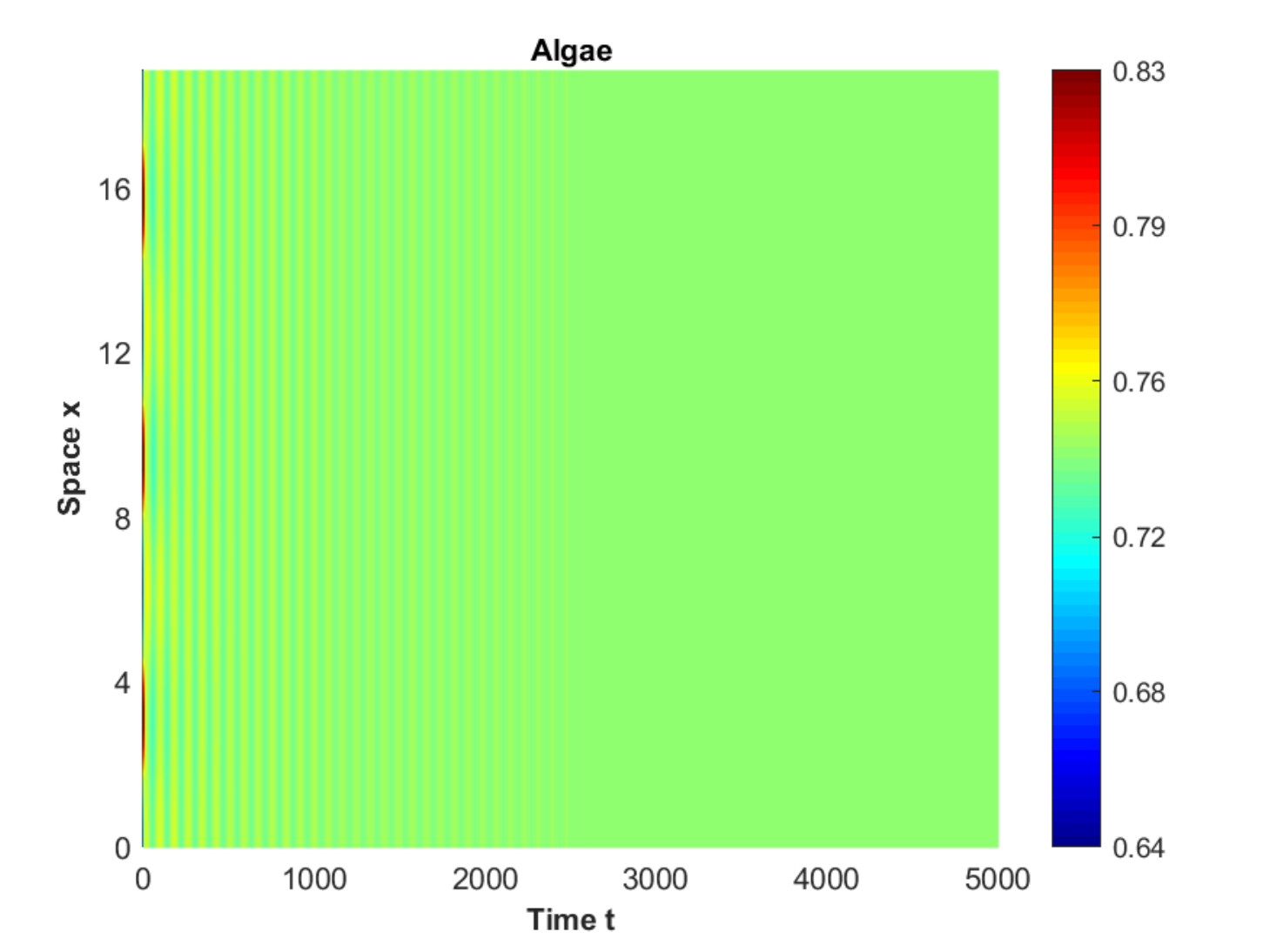}}
\end{center}
\caption{The positive constant steady state $E_*(m^*, a^*)$ of system \eqref{eq_ma_tau} is locally asymptotically stable in $D_1$ with
 $(\tau_{\varepsilon}, d_{\varepsilon})=(-0.5, 0.01)\in D_1$, and the initial function are $(m^*+0.1\cos x, a^*-0.1\cos x)$.}\label{fig-D1}
\end{figure}

\newpage
\section{Discussion and conclusion}
In this paper, we investigate the spatiotemporal patterns induced by the Turing-Hopf bifurcation for a mussel-algae model with delay and diffusion.

We first show the global existence of solutions of system \eqref{eq_ma_tau}. But, the boundedness of mussel $m(x,t)$ is still unknown. A reason is that the death rate of mussel depends on the density of mussels themselves. If the mortality of mussel is a constant, then the estimate of mussel can be obtained without difficulty. Hence, a open mathematical question for this model is the global stability of the positive spatially homogeneous steady state.

Under the assumption \textsc{(H1)} and \textsc{(H2)}, the positive spatially homogeneous steady state is locally asymptotically stable under a linear homogeneous perturbation when $\tau=0$. But when $\tau$ get the critical value $\tau_0$, the positive spatially homogeneous steady state will lose its stability , at the same time,
a positive spatially homogeneous periodic solution appears and the system undergoes a Hopf bifurcation which is induced by the delay.

To investigate the Turing instability of system \eqref{eq_ma_tau}, we discuss the effect of diffusion coefficient $d$. If $d>d_0$, there is no Turing instability; and if $d<d_0$, one can always find a wave number $k$ such that Turing instability occurs. It's nothing that $d$ is not the true diffusivity ratio, actually, it is only the diffusion coefficient of the predator, mussel, and $\frac{1}{\gamma}$ is another diffusion coefficient belongs to the prey, algae. For fixed $\gamma$, if $d$ is sufficiently large, which means the diffusivity ratio $d\gamma$ is sufficient large, and by our result, there is no Turing instability. According to the mechanism of pattern formation presented by Turing in \cite{Tur}, the mussel represents the ``activator" while algae, the``inhibitor". It is somewhat different from the general predator-prey model.

The dynamics near the Turing-Hopf bifurcation is discussed in detail by using the method of normal form for partial functional differential equations. We divide the $\tau_{\varepsilon}-d_{\varepsilon}$ plane into six regions with the phase portraits of each region are different. There are four types of patterns: spatially homogeneous / inhomogeneous steady state;  spatially homogeneous / inhomogeneous periodic solutions. From the numerical simulations, one can easily see that the delay $\tau$ and diffusion coefficient $d$ could result in complex spatiotemporal dynamics.

The interaction between mussel and algae contains a wealth of information.
Considering the mechanisms of flow motion \cite{ShM} and formation of mussel bed \cite{Liu-2014}, there are still many problems to be solved. For example, If the advection term is added, how will it affect the dynamics of system? when the
space domain expand to 2-dimension, what are the effects of time delay, diffusion coefficient and the advection ? and how do they interact each other ?

Within restoration ecology, the mussel beds are typical and active research system \cite{Don, Liu-2014}.
Also, because of the high edible and medicinal value, mussel fisheries plays an important role in fiscal revenue in many coastal areas.
The formation of spatiotemporal patterns may affect both the resilience and productivity of mussel beds.
Hence, studying the mussel-algae model and the formation of different patterns has important biological and economic significance and we
need more realistic and detailed models to depict those behaviors in the following work.

\section{Appendix}

The coefficient vectors $F_{y_i(\theta)z_j}$, $F_{mnk}$ presented in normal form \eqref{z_norm} and therein can be obtained by using the following calculation formulas, where $F_{mm}=\cfrac{\partial^2}{\partial m^2}F(0,0)$, $F(\mu_{\varepsilon}, U_t)$ is defined by \eqref{F}, and others can be deduced by analogy.

\begin{equation*}
  \begin{array}{ll}
  F_{y_1(0)z_1}=2(F_{mm}+F_{ma}q_1+F_{mm_{\tau}}e^{-i\omega_0\tau_0}+F_{ma_{\tau}}q_1e^{-i\omega_0\tau_0}),\\
  F_{y_1(-1)z_1}=2(F_{mm_{\tau}}+F_{m_{\tau}a}q_1+F_{m_{\tau}m_{\tau}}e^{-i\omega_0\tau_0}+F_{m_{\tau}a_{\tau}}q_1e^{-i\omega_0\tau_0}),\\
  F_{y_2(0)z_1}=2(F_{ma}+F_{aa}q_1+F_{m_{\tau}a}e^{-i\omega_0\tau_0}+F_{aa_{\tau}}q_1e^{-i\omega_0\tau_0}),\\
  F_{y_2(-1)z_1}=2(F_{ma_{\tau}}+F_{aa_{\tau}}q_1+F_{m_{\tau}a_{\tau}}e^{-i\omega_0\tau_0}+F_{a_{\tau}a_{\tau}}q_1e^{-i\omega_0\tau_0}),\\
  F_{y_1(0)z_2}=2(F_{mm}+F_{mm_{\tau}}+F_{ma}p_1+F_{ma_{\tau}}p_1),\\
  F_{y_1(-1)z_2}=2(F_{mm_{\tau}}+F_{m_{\tau}m_{\tau}}+F_{m_{\tau}a_{\tau}}p_1+F_{m_{\tau}a}p_1),\\
  F_{y_2(0)z_2}=2(F_{ma}+F_{m_{\tau}a}+F_{aa}p_1+F_{aa_{\tau}}p_1,\\
  F_{y_2(-1)z_2}=2(F_{ma_{\tau}}+F_{m_{\tau}a_{\tau}}+F_{aa_{\tau}}p_1+F_{a_{\tau}a_{\tau}}p_1),\\ \end{array}
\end{equation*}
and
\begin{equation*}
  \begin{array}{ll}
   F_{200}=&F_{mm}+F_{aa}q_1^2+F_{m_{\tau}m_{\tau}}e^{-2i\omega_0\tau_0}+F_{a_{\tau}a_{\tau}}q_1^2e^{-2i\omega_0\tau_0}+2(F_{ma}q_1+
             F_{mm_{\tau}}e^{-i\omega_0\tau_0}\\
             &+F_{ma_{\tau}}q_1e^{-i\omega_0\tau_0}+F_{m_{\tau}a}q_1e^{-i\omega_0\tau_0}+F_{aa_{\tau}}q_1^2e^{-i\omega_0\tau_0}+F_{m_{\tau}a_{\tau}}
              q_1e^{-2i\omega_0\tau_0}),\\
   F_{110}=&2\big[F_{mm}+F_{aa}q_1\bar{q}_1+F_{m_{\tau}m_{\tau}}+F_{a_{\tau}a_{\tau}}q_1\bar{q}_1+F_{ma}(q_1+\bar{q}_1)
             +F_{mm_{\tau}}(e^{-i\omega_0\tau_0}+e^{i\omega_0\tau_0})\\
             &+F_{ma_{\tau}}(q_1e^{-i\omega_0\tau_0}+\bar{q}_1e^{i\omega_0\tau_0})+F_{m_{\tau}a}(q_1e^{i\omega_0\tau_0}+\bar{q}_1e^{-i\omega_0\tau_0})
             +F_{aa_{\tau}}q_1\bar{q}_1(e^{-i\omega_0\tau_0}+e^{i\omega_0\tau_0})\\
             &+F_{m_{\tau}a_{\tau}}(q_1+\bar{q}_1)\big],\\
   F_{101}=&2\big[F_{mm}+F_{aa}q_1 p_1+F_{m_{\tau}m_{\tau}}e^{-i\omega_0\tau_0}+F_{a_{\tau}a_{\tau}}q_1p_1e^{-i\omega_0\tau_0}
              +F_{ma}(q_1+p_1)+F_{mm_{\tau}}(1+e^{-i\omega_0\tau_0})\\
             &+F_{ma_{\tau}}(p_1+q_1e^{-i\omega_0\tau_0})+F_{m_{\tau}a}(q_1+p_1e^{-i\omega_0\tau_0})
             +F_{aa_{\tau}}q_1p_1(1+e^{-i\omega_0\tau_0})\\
             &+F_{m_{\tau}a_{\tau}}(q_1+p_1)e^{-i\omega_0\tau_0}\big],\\
   F_{002}=&F_{mm}+F_{aa}p_1^2+F_{m_{\tau}m_{\tau}}+F_{a_{\tau}a_{\tau}}p_1^2+2(F_{ma}p_1+
             F_{mm_{\tau}}+F_{ma_{\tau}}p_1+F_{m_{\tau}a}p_1\\
             &+F_{aa_{\tau}}p_1^2+F_{m_{\tau}a_{\tau}}p_1),\\
   F_{020}=&\overline{F_{200}},\\
   F_{011}=&\overline{F_{101}}.
  \end{array}
\end{equation*}
and
\begin{equation*}
  \begin{array}{ll}
    F_{210}=&3\big[F_{mmm}+F_{mma}(2q_1+\bar{q}_1)+F_{mmm_{\tau}}(2e^{-i\omega_0\tau_0}+e^{i\omega_0\tau_0})+F_{maa}q_1(2\bar{q}_1+q_1)
             +F_{mma_{\tau}}(2q_1e^{-i\omega_0\tau_0}\\
             &+\bar{q}_1e^{i\omega_0\tau_0})+2F_{mam_{\tau}}(q_1e^{-i\omega_0\tau_0}+q_1e^{i\omega_0\tau_0}+\bar{q}_1e^{-i\omega_0\tau_0})
              +2F_{maa_{\tau}}q_1(q_1e^{-i\omega_0\tau_0}+\bar{q}_1e^{i\omega_0\tau_0}+\bar{q}_1e^{-i\omega_0\tau_0})\\
              &+F_{mm_{\tau}m_{\tau}}(e^{-2i\omega_0\tau_0}+2)+2F_{mm_{\tau}a_{\tau}}(q_1+\bar{q}_1+q_1e^{-2i\omega_0\tau_0})
              +F_{ma_{\tau}a_{\tau}}q_1(2\bar{q_1}+q_1e^{-2i\omega_0\tau_0})+F_{aaa}q_1^2\bar{q}_1\\
              &+F_{aam_{\tau}}q_1(2\bar{q}_1e^{-i\omega_0\tau_0}+q_1e^{i\omega_0\tau_0})
              +F_{aaa_{\tau}}q_1^2\bar{q}_1(2e^{-i\omega_0\tau_0}+e^{i\omega_0\tau_0})+F_{am_{\tau}m_{\tau}}(2q_1+\bar{q}_1e^{-2i\omega_0\tau_0})\\
              &+2F_{am_{\tau}a_{\tau}}q_1(\bar{q}_1+\bar{q}_1e^{-2i\omega_0\tau_0}
              +q_1)+F_{aa_{\tau}a_{\tau}}q_1^2\bar{q}_1(2+e^{-2i\omega_0\tau_0})+F_{m_{\tau}m_{\tau}m_{\tau}}e^{-i\omega_0\tau_0}\\
               &+F_{m_{\tau}m_{\tau}a_{\tau}}(2q_1e^{-i\omega_0\tau_0}+\bar{q}_1e^{-i\omega_0\tau_0})
               +F_{m_{\tau}a_{\tau}a_{\tau}}q_1e^{-i\omega_0\tau_0}(2\bar{q}_1+q_1)
              +F_{a_{\tau}a_{\tau}a_{\tau}}q_1^2\bar{q}_1e^{-i\omega_0\tau_0}\big];\\
  \end{array}
\end{equation*}
\begin{equation*}
  \begin{array}{ll}
 F_{102}=&3\big[F_{mmm}+F_{mma}(q_1+2p_1)+F_{mmm_{\tau}}(e^{-i\omega_0\tau_0}+2)+F_{maa}p_1(2q_1+p_1)
             +F_{mma_{\tau}}(2p_1+q_1e^{-i\omega_0\tau_0})\\
             &+2F_{mam_{\tau}}(q_1+p_1+p_1e^{-i\omega_0\tau_0})
              +2F_{maa_{\tau}}p_1(q_1+p_1+q_1e^{-i\omega_0\tau_0})\\
              &+F_{mm_{\tau}m_{\tau}}(2e^{-i\omega_0\tau_0}+1)+2F_{mm_{\tau}a_{\tau}}(p_1+p_1e^{-i\omega_0\tau_0}+q_1e^{-i\omega_0\tau_0})
              +F_{ma_{\tau}a_{\tau}}p_1(p_1+2q_1e^{-i\omega_0\tau_0})\\
              &+F_{aaa}q_1p^2_1+F_{aam_{\tau}}p_1(2q_1+p_1e^{-i\omega_0\tau_0})
              +F_{aaa_{\tau}}q_1p^2_1(2+e^{-i\omega_0\tau_0})+F_{am_{\tau}m_{\tau}}(q_1+2p_1e^{-i\omega_0\tau_0})\\
              &+2F_{am_{\tau}a_{\tau}}p_1(q_1+q_1e^{-i\omega_0\tau_0}
              +p_1e^{-i\omega_0\tau_0})+F_{aa_{\tau}a_{\tau}}q_1p_1^2(1+2e^{-i\omega_0\tau_0})+F_{m_{\tau}m_{\tau}m_{\tau}}e^{-i\omega_0\tau_0}\\
               &+F_{m_{\tau}m_{\tau}a_{\tau}}(q_1e^{-i\omega_0\tau_0}+2p_1e^{-i\omega_0\tau_0})
               +F_{m_{\tau}a_{\tau}a_{\tau}}p_1e^{-i\omega_0\tau_0}(2q_1+p_1)
              +F_{a_{\tau}a_{\tau}a_{\tau}}q_1p_1^2e^{-i\omega_0\tau_0}\big],
  \end{array}
\end{equation*}
\begin{equation*}
  \begin{array}{ll}
    F_{111}=&6\Big\{F_{mmm}+F_{mma}(q_1+\bar{q}_1+p_1)+F_{mmm_{\tau}}(e^{-i\omega_0\tau_0}+e^{i\omega_0\tau_0}+1)
             +F_{maa}(q_1\bar{q}_1+q_1p_1+p_1\bar{q}_1)\\
             &+F_{mma_{\tau}}(q_1e^{-i\omega_0\tau_0}+\bar{q}_1e^{i\omega_0\tau_0}+p_1)
              +F_{mam_{\tau}}\big[q_1(1+e^{i\omega_0\tau_0})+\bar{q}_1(1+e^{-i\omega_0\tau_0})+p_1(e^{i\omega_0\tau_0}+e^{-i\omega_0\tau_0})\big]\\
              &+F_{maa_{\tau}}\big[q_1p_1(1+e^{-i\omega_0\tau_0})+q_1\bar{q}_1(e^{i\omega_0\tau_0}+e^{-i\omega_0\tau_0})
                +\bar{q}_1p_1(1+e^{i\omega_0\tau_0})\big]
              +F_{mm_{\tau}m_{\tau}}(e^{i\omega_0\tau_0}+e^{-i\omega_0\tau_0}+1)\\
              &+F_{mm_{\tau}a_{\tau}}\big[q_1(1+e^{-i\omega_0\tau_0})+\bar{q}_1(1+e^{i\omega_0\tau_0})
              +p_1(e^{i\omega_0\tau_0}+e^{-i\omega_0\tau_0})\big]\\
              &+F_{ma_{\tau}a_{\tau}}(q_1\bar{q_1}+\bar{q}_1p_1e^{i\omega_0\tau_0}+q_1p_1e^{-i\omega_0\tau_0})
              +F_{aaa}q_1\bar{q}_1p_1
              +F_{aam_{\tau}}(q_1\bar{q}_1+\bar{q}_1p_1e^{-i\omega_0\tau_0}+q_1p_1e^{i\omega_0\tau_0})\\
              &+F_{aaa_{\tau}}q_1\bar{q}_1p_1(1+e^{-i\omega_0\tau_0}+e^{i\omega_0\tau_0})
              +F_{am_{\tau}m_{\tau}}(p_1+q_1e^{i\omega_0\tau_0}+\bar{q}_1e^{-i\omega_0\tau_0})\\
              &+F_{am_{\tau}a_{\tau}}\big[(q_1\bar{q}_1(e^{i\omega_0\tau_0}+e^{-i\omega_0\tau_0})+q_1p_1(1+e^{i\omega_0\tau_0})
              +\bar{q}_1p_1(1+e^{-i\omega_0\tau_0})\big]\\
              &+F_{aa_{\tau}a_{\tau}}q_1\bar{q}_1p_1(1+e^{i\omega_0\tau_0}+e^{-i\omega_0\tau_0})+F_{m_{\tau}m_{\tau}m_{\tau}}
               +F_{m_{\tau}m_{\tau}a_{\tau}}(q_1+\bar{q}_1+p_1)\\
               &+F_{m_{\tau}a_{\tau}a_{\tau}}(q_1\bar{q}_1+q_1p_1+\bar{q}_1p_1)
              +F_{a_{\tau}a_{\tau}a_{\tau}}q_1\bar{q}_1p_1\Big\};\\
    \end{array}
  \end{equation*}
\begin{equation*}
\begin{array}{ll}
F_{003}=&F_{mmm}+3F_{mmm_{\tau}}+3F_{mm_{\tau}m_{\tau}}+F_{m_{\tau}m_{\tau}m_{\tau}}+3F_{mma}p_1+3F_{mma_{\tau}}p_1+6F_{mam_{\tau}}p_1
            +6F_{mm_{\tau}a_{\tau}}p_1\\
            &+3F_{m_{\tau}m_{\tau}a_{\tau}}p_1+3F_{am_{\tau}m_{\tau}}p_1+3F_{maa}p_1^2+6F_{maa_{\tau}}p_1^2+
             3F_{ma_{\tau}a_{\tau}}p_1^2
             +F_{aaa}p_1^3+3F_{m_{\tau}a_{\tau}a_{\tau}}p_1^2\\
             &+3F_{aam_{\tau}}p_1^2+3F_{aaa_{\tau}}p_1^3+6F_{am_{\tau}a_{\tau}}p_1^2+
             3F_{aa_{\tau}a_{\tau}}p_1^3+F_{a_{\tau}a_{\tau}a_{\tau}}p_1^3.
  \end{array}
\end{equation*}


\begin{thebibliography}{99}

\bibitem{AnJ}
    An. Q. \& Jiang W. H. [2018]
    ``Spatiotemporal attractors generated by the Turing-Hopf bifurcation in a time-delayed reaction-diffusion system''
    {\it Discrete} \& {\it Continuous Dynamical Systems-B},  220--229.

\bibitem{BGF}
     Baurmann, M., Gross, T. \& Feudel, U. [2007]
     ``Instabilities in spatially extended predator每prey systems: Spatio-temporal patterns in the neighborhood of Turing每Hopf bifurcations,''
     {\it J. Math. Biol.}, {\bf 245}, 220-229.

\bibitem{Can}
      Cangelosi, R. A., Wollkind, D. J., Kealy-Dichone, B. J. \& Chaiya I. [2015]
     `` Nonlinear stability analyses of Turing patterns for a mussel-algae model,''
    {\it J. Math. Biol.}, {\bf 70}, 1249--1294.


\bibitem{ChY}
     Chen, S. S.  \& Yu, J. S. [2016]
     ``Stability and bifurcations in a nonlocal delayed reaction每diffusion population model,''
    {\it J. Differential Equations}, {\bf 260}, 218--240.


\bibitem{Cooke}
    Cooke, K. L. \& Grossman Z. [1982]
     ``Discrete delay, distributed delay and stability switches,''
     {\it J. Math. Anal. Appl.}, {\bf 86}, 592--627.

\bibitem{DLDB}
     De Wit, A., Lima, D., Dewel, G., \& Borckmans, P.[1996]
     ``Spatiotemporal dynamics near a codimension-two point,''
     {\it Phys. Rev. E}, {\bf54}, 261--271.

\bibitem{Don}
     Donker, J. J. A. [2015]
     ``Hydrodynamic processes and the stability of intertidal mussel beds in the Dutch Wadden Sea,''
     PhD thesis, Utrecht University, Netherlands.

\bibitem{Far}
     Faria T. [2000]
     ``Normal forms and Hopf bifurcation for partial differential equations with delays,''
     {\it Trans. Amer. Math. Soc.}, {\bf 352}, 2217--2238.


\bibitem{GuH}
      Guckenheimer, J. \& Holmes, P. [1983]
     {\it Nonlinear Oscillations, Dynamical Systems and Bifurcations of Vector Fields}
     (Springer Verlag, New York).

\bibitem{GhM}
    Ghazaryan, A. \& Manukian, V. [2015]
     ``Coherent structures in a population model for mussel-algae interaction,''
     {\it SIAM J. Appl. Dyn. Syst.}, {\bf 14}, 893--913.

\bibitem{HaR}
     Hadeler, K. P., \& Ruan, S. G. [2007]
     ``Interaction of diffusion and delay,''
     {\it Discrete Contin. Dyn. Syst. Ser. B}, {\bf 8}, 95--105.

\bibitem{Kla}
    Klausmeier C. A. [1999]
     ``Regular and irregular patterns in semiarid vegetation,''
     {\it Science}, {\bf 284}, 1826--1828.

\bibitem{Liu-2012}
     Liu, Q. X., Weerman, E. J., Herman, P. M. J., Han, O. \& Johan, V. D. K. [2012]
     ``Alternative mechanisms alter the emergent properties of self-organization in mussel beds,''
    {\it Proc. R. Soc. B.}, {\it 279}, 2744--2753.

\bibitem{Liu-2013}
     Liu, Q. X., Doelman, A., Rottsch\"{a}fer, V., Jager, M. D. \& Herman, P.M.J. [2013]
    ``Phase separation explains a new class of self-organized spatial patterns in ecological systems,''
     {\it Proc. Natl. Acad. Sci. USA}, {\bf110}, 11905--11910


\bibitem{Liu-2014}
      Liu, Q. X., Herman, P. M., Mooij, W. M., Huisman, J., Scheffer, M., Olff, H. \& van de Koppel, J. [2014]
      ``Pattern formation at multiple spatial scales drives the resilience of mussel bed ecosystems,''
      {\it Nature communications}, {\bf 5}, 5234.


\bibitem{LSW}
     Lin, X. D., So, J. W. H. \& Wu, J. H. [1992]
     ``Centre manifolds for partial differential equations with delays,''
     {\it Proc. Roy. Soc. Edinburgh Sect. A}, {\bf 122}, 237--254.


\bibitem{NiT}
     Ni, W. M. \& Tang, M. [2005]
     ``Turing patterns in the Lengyel-Epstein system for the CIMA reaction,''
     {\it Trans. Amer. Math. Soc.}, {\bf 357}, 3953--3969.

\bibitem{OuS}
     Ouyang, Q. \& Swinney, H. L. [1991]
     ``Transition from a uniform state to hexagonal and striped Turing patterns,''
     {\it Nature}, {\bf 352}, 610.

\bibitem{Pao-1996}
     Pao, C. V. [1996]
     ``Dynamics of Nonlinear Parabolic Systems with Time Delays,''
       {\it Journal of Mathematical Analysis} \& {\it Applications}, {\bf198}, 751--779.

\bibitem{Paz}
    Pazy, A. [1983]
     {\it Semigroups of Linear Operators and Applications to Partial Differential Equations}
    (Springer-Verlag, New York).

\bibitem{ShW-2018}
     Shen Z. L. \& Wei J. J. [2018]
     ``Hopf bifurcation analysis in a diffusive predator-prey system with delay and surplus killing effect,''
     {\it Math. Biosci. Eng.}, {\bf15}, 693--715.

\bibitem{ShW-2}
     Shen Z. L. \& Wei J. J. [submitted]
     ``Bifurcation Analysis in A Diffusive Mussel-Algae Model with Delay,''
     {\it submitted.}.

\bibitem{ShM}
     Sherratt, J. A. \& Mackenzie, J. J. [2016]
     ``How does tidal flow affect pattern formation in mussel beds ?,''
     {\it J. Theoret. Biol.}, {\bf406}, 83--92.

\bibitem{SJL}
      Song, Y. L., Jiang, H. P., Liu, Q. X. \& Yuan, Y. [2017]
     ``Spatiotemporal dynamics of the diffusive mussel-algae model near Turing-Hopf bifurcation,''
     {\it SIAM J. Appl. Dyn. Syst.}, {\bf16}, 2030--2062.

\bibitem{SoZ}
      Song, Y. L., \& Zou, X. F. [2014]
     ``Spatiotemporal dynamics in a diffusive ratio-dependent predator每prey model near a Hopf每Turing bifurcation point,''
     {\it Comput. Math. Appl}, {\bf67}, 1978--1997.

\bibitem{Tay}
     Taylor, M. E. [2010]
     {\it Partial Differential Equations III: Nonlinear Equations (Applied Mathematical Science)}
     (Springer-Verlag, New York).

\bibitem{Tur}
     Turing, A. M. [1952]
     ``The chemical basis of morphogenesis,''
     {\it Philos. Trans. R. Soc. Lond. Ser. A}, {\bf237}, 37--72.

\bibitem{Kop-2005}
     van de Koppel, J., Rietkerk, M., Dankers, N. \& Herman, P. M. J. [2005]
    ``Self-dependent feedback and regular spatial patterns in young mussel beds,''
     {\it Am. Nat.}, {\bf165}, E66--77.

\bibitem{Kop-2008}
     van de Koppel, J., Gascoigne, J. C., Theraulaz, G., Rietkerk, M., Mooij, W. M. \& Herman, P.M.J. [2008]
     ``Experimental evidence for spatial self-organization in mussel bed ecosystems,''
     {\it Science}, {\bf 322}, 739--742.

\bibitem{WLS}
     Wang, R. H., Liu, Q. X., Sun, G. Q., Jin, Z. \& van de Koppel, J. [2009]
     ``Nonlinear dynamic and pattern bifurcations in a model for spatial patterns in young mussel beds,''
     {\it J. R. Soc. Interface}, {\bf6}, 705--718.

\bibitem{Wu}
     Wu, J. H. [1996]
     {\it Theory and Applications of Partial Functional Differential Equations}
     (Springer, New York).

\bibitem{XuW}
     Xu, X. F. \& Wei, J. J. [2017]
     ``Bifurcation analysis of a spruce budworm model with diffusion and physiological structures,''
     {\it J. Differential Equations}, {\bf262}, 5206--5230

\bibitem{YaS}
     Yang, R. \& Song, Y. L. [2016]
     ``Spatial resonance and Turing每Hopf bifurcations in the Gierer每Meinhardt model,''
     {\it Nonlinear Anal. Real World Appl.}, {\bf31}, 356--387

\end{thebibliography}
\end{document}